\documentclass[11pt]{amsart}
\usepackage{amsmath,amsthm,amscd,amsfonts,amssymb,epic,eepic,bbm, graphicx, youngtab,txfonts, ytableau}
\usepackage{tikz,color,graphicx,tikz-cd,xcolor,tkz-euclide}
\usepackage[draft]{hyperref}
\usepackage{enumerate}
\usepackage{array}
\usepackage{mathtools}

\usepackage{kotex-logo}

\usepackage[T1]{fontenc}
\PassOptionsToPackage{no-math}{fontspec}
\usepackage{iftex}

\ifPDFTeX
\usepackage[finemath]{kotex}
\usepackage{dhucs-nanumfont}
\fi

\ifXeTeX
\usepackage[unfonts]{kotex}
\xetexkofontregime{latin}[puncts=prevfont, colons=prevfont, cjksymbols=hangul]

\defaultfontfeatures+{
Ligatures          = TeX,
ItalicFont         = *,
ItalicFeatures     = {FakeSlant=.17}, 
BoldItalicFeatures = {FakeSlant=.17},
}

\setmainhangulfont{NanumMyeongjo}
\setsanshangulfont{NanumGothic}

\fi

\numberwithin{equation}{section}
\allowdisplaybreaks

\catcode`\@=11
\@namedef{subjclassname@2020}{%
  \textup{2020} Mathematics Subject Classification}
\catcode`\@=12

\newtheorem{thm}{Theorem}[section]
\newtheorem{cor}[thm]{Corollary}

\newtheorem{lem}[thm]{Lemma}

\newtheorem{prop}[thm]{Proposition}

\theoremstyle{definition}
\newtheorem{defn}[thm]{Definition}

\theoremstyle{remark}
\newtheorem{rem}[thm]{Remark}

\newtheorem{example}[thm]{Example}

\newcommand\peak{\operatorname{peak}}
\newcommand\comp{\operatorname{hox}}

\newcommand\height{\operatorname{ht}}
\newcommand\maxid{\operatorname{prmx}}
\newcommand\block{\operatorname{block}}

\newcommand\asc{\operatorname{asc}}
\newcommand\cons{\operatorname{cons}}
\newcommand\hdd{\operatorname{hdd}}
\newcommand\Z{\mathbb{Z}}
\newcommand\db{{d_{\rm b}}}
\newcommand\dr{{d_{\rm r}}}

\newcommand\last{\operatorname{last}}
\newcommand\bval{\operatorname{bval}}
\newcommand\first{\operatorname{first}}
\newcommand\single{\operatorname{single}}
\newcommand\wt{\operatorname{wt}}
\newcommand\one{\operatorname{wone}}

\newcommand\leaf{\operatorname{leaf}}
\newcommand\north{\operatorname{north}}

\newcommand\dasc{\operatorname{dasc}}
\newcommand\aone{\operatorname{aone}}
\newcommand\bone{\operatorname{bone}}
\newcommand\omi{\operatorname{omit}}

\newcommand\crit{\operatorname{crit}}

\newcommand\schp{\mathcal{P}}
\newcommand\Fp{\mathcal{F}}
\newcommand\bdp{\mathcal{B}}
\newcommand\perm{\mathcal{S}}

\newcommand\invi{\mathcal{I}}
\newcommand\invj{\mathcal{J}}
\newcommand\wot{\mathcal{T}}
\newcommand\je{\varepsilon}
\newcommand\jf{\zeta}
\newcommand\jg{\eta}

\newcommand\set[1]{\left\{ #1 \right\}}
\newcommand\abs[1]{\left| #1 \right|}
\newcommand\deli{:}

\newcommand\black[1]{{\color{black} #1}}

\begin{document}

\title{Bijections on pattern avoiding inversion sequences and related objects}

\author{JiSun Huh}
\address[JiSun Huh]{Department of Mathematics, Ajou University, Suwon, 16499, South Korea}
\email{hyunyjia@ajou.ac.kr}

\author{Sangwook Kim$^\dag$}
\address[Sangwook Kim]{Department of Mathematics, Chonnam National University, Gwangju, 61186, South Korea}
\email{swkim.math@chonnam.ac.kr}
\thanks{\dag Corresponding author}

\author{Seunghyun Seo}
\address[Seunghyun Seo]{Department of Mathematics Education, Kangwon National University, Chuncheon, 24341, South Korea}
\email{shyunseo@kangwon.ac.kr}

\author{Heesung Shin}
\address[Heesung Shin]{Department of Mathematics, Inha University, Incheon, 22212, South Korea}
\email{shin@inha.ac.kr}

\keywords{
Schr\"{o}der paths, 
permutations, 
inversion sequences, 
ordered trees,
pattern avoidance, 
direct sum}
\subjclass[2020]{Primary 05A19; Secondary 05A05, 05A15}

\begin{abstract}
The number of inversion sequences avoiding two patterns $101$ and $102$ is known to be the same as
the number of permutations avoiding three patterns $2341$, $2431$, and $3241$.
This sequence also counts the number of Schr\"{o}der paths without triple descents, restricted bicolored Dyck paths,
$(101,021)$-avoiding inversion sequences, and weighted ordered trees.
We provide bijections to integrate them together by introducing $F$-paths.
Moreover, we define three kinds of statistics for each of the objects and count the number of each object with respect to these statistics.
We also discuss direct sums of each object.
\end{abstract}

\date{\today}
\maketitle

\tableofcontents


\section{Introduction}
The study of permutations without certain patterns has long history.
MacMahon~\cite{MacMahon60} enumerated permutations avoiding $123$ and 
Knuth~\cite{Knuth1, Knuth3} showed that for any $\tau \in \mathfrak{S}_3$, the number of permutations of length $n$ that avoid $\tau$ is given by the Catalan number.
The first systematic study of pattern avoidance in permutations was done by Simion and Schmidt~\cite{SimionSchmidt85} in 1985.
Since then, there have been many studies on the patterns of permutations.
We refer the reader to Kitaev~\cite{Kitaev11} for the survey.
The classes defined by avoiding two patterns of length four fall into 38 Wilf equivalence classes.
The largest of these Wilf classes consists of 10 symmetry classes including $(4123,4213)$-avoiding permutations and $(4132, 4213)$-avoiding permutations.
Kremer~\cite{Kremer00} showed that this class is enumerated by the large Schr\"{o}der numbers. 

A Schr\"{o}der path is a lattice path that starts at the origin, ends at the \( x \)-axis, never goes below the \( x \)-axis, and consists of steps \( u=(1,1) \), \( d=(1,-1) \), and \( h=(2,0) \). 
It is well known that the number of Schr\"{o}der paths is enumerated by the large Schr\"{o}der numbers.
There are several results about the number of Schr\"{o}der paths without certain subpaths.
Schr\"{o}der paths without horizontal steps are called Dyck paths, and they are enumerated by the Catalan numbers.
Yan~\cite{Yan09} showed that the set of Schr\"{o}der paths without subpath \( uh \), 
the set of Schr\"{o}der paths without peaks of even height, 
and the set of set partitions avoiding $12312$ have the same cardinality.
Kim~\cite{Kim11} showed that this sequence also counts the set of $2$-distant noncrossing partitions and the set of $3$-Motzkin paths.
Seo and Shin~\cite{SeoShin23} provided the number of $k$-Schr\"{o}der paths without peaks and valleys, for $k=1, 2$.

A sequence \( e=(e_1, e_2, \dots, e_n) \) is called an \emph{inversion sequence of length \( n \)} if \( 0 \leq e_i <i \) for 
all \( i \in [n] \).
The set of inversion sequences of length $n$ is in natural bijection with the set of permutations of $[n]$ via the well-known Lehmer code~\cite{Lehmer60}.
The study of pattern avoidance in inversion sequences was initiated recently by
Corteel, Martinez, Savage, and Weselcouch~\cite{CMSW16}, and 
Mansour and Shattuck~\cite{MansourShattuck15}, independently.
They studied inversion sequences avoiding patterns of length $3$.
In particular, the number of inversion sequences avoiding $021$ is given by the large Schr\"{o}der numbers.
Martinez and Savage~\cite{MartinezSavage18} reframed the notion of a length-three pattern from a word of length $3$ to a triple of binary relations.
Yan and Lin~\cite{YanLin20} classified the Wilf equivalences for inversion sequences avoiding pairs of length-three patterns.
The Wilf classification for patterns of length four was completed by Hong and Li~\cite{HongLi22}.\\

In this paper, we focus on several objects that are enumerated by the integer sequence A106228 in the OEIS \cite{oeis},
including Schr\"{o}der paths without triple descents and weighted ordered trees.
Albert, Homberger, Pantone, Shar, and Vatter~\cite{AHP18} showed that 
the \( (4123, 4132, 4213) \)-avoiding permutations are enumerated by this sequence. 
Martinez and Savage~\cite{MartinezSavage18} conjectured that the number of \( (101,102) \)-avoiding inversion sequences
is the same as the number of \( (4123, 4132, 4213) \)-avoiding permutations, and  
Cao, Jin, and Lin~\cite{CaoJinLin19} confirmed this conjecture using generating functions.
Yan and Lin~\cite{YanLin20} proved that the \( (101, 021) \)-avoiding inversion sequences 
and the restricted bicolored Dyck paths are also counted by this sequence by constructing a bijection 
from the set of weighted ordered trees to the set of restricted bicolored Dyck paths. 

The main goal of this paper is to provide bijections to integrate all these objects.  
To do this, we introduce an \emph{$F$-path}, a lattice path with steps in the set \( F := \{ (a, b) : a \ge 1, b \le 1 \} \cup \{ (0,1)\} \) that starts at the origin and never goes below the line \( y=x \), and then
construct bijections from each object to \( F \)-paths. 
Since \( (4123, 4132, 4213) \)-avoiding permutations can be obtained from \( (2341, 2431, 3241) \)-avoiding permutations by applying reversal and complement, 
we consider \( (2341, 2431, 3241) \)-avoiding permutations 
instead of \( (4123, 4132, 4213) \)-avoiding permutations for structural advantage. 
The following theorem states our main results.

\begin{thm}\label{thm:main}
Let \( n \) be a nonnegative integer. There is a bijection between any pair of the following sets:
\begin{enumerate}
\item The set \( \Fp_{n} \) of $F$-paths of length \( n \).
\item The set \( \schp_n \) of Schr\"{o}der paths of semilength \( n \) without triple descents.
\item The set \( \bdp_{n+1} \) of restricted bicolored Dyck paths of semilength \( n+1 \).
\item The set \(\perm_{n+1} \) of \( (2341, 2431, 3241) \)-avoiding permutations on 
\( \set{1,2,\dots,n+1} \).
\item The set \( \invi_{n+1} \) of \( (101,102) \)-avoiding inversion sequences of length \( n+1 \).
\item The set \( \invj_{n+1} \) of \( (101,021) \)-avoiding inversion sequences of length \( n+1 \).
\item The set \( \wot_{n+1} \) of weighted ordered trees with \( n+1 \) edges.
\end{enumerate}
\end{thm}

For an $F$-path $Q$ with the terminal point $(x_n, y_n)$, 
we define the height of $Q$ to be $\height(Q)\coloneqq y_{n}-x_{n}$, 
which plays a crucial role in the constructions of these bijections.
In particular, these bijections provide correspondences 
among Schr\"{o}der paths 
without triple descents having 
$k$ horizontal steps on the $x$-axis, 
$(2341, 2431, 3241)$-avoiding permutations with $k+1$ blocks, 
and $F$-paths with height $k$. 
For the other objects, we define corresponding statistics as well.
In addition, for each object, we provide additional statistics that our bijections preserve.
Moreover, we  enumerate the number of these objects with respect to several statistics
by applying Lagrange inversion formula to the generating function for \( F\)-paths.

The rest of the paper is organized as follows.
Sections~\ref{sec:F-P-B} through \ref{sec:T} contain the bijections from each object 
in Theorem~\ref{thm:main} to the set of $F$-paths.
In Section~\ref{sec:F-P-B}, we construct the bijections for Schr\"{o}der paths without triple descents and restricted bicolored Dyck paths.
Section~\ref{sec:S} is dedicated to $(2341, 2431, 3241)$-avoiding permutations.
In Section~\ref{sec:I-J}, 
$(101,102)$-avoiding inversion sequences and $(101,021)$-avoiding inversion sequences
are discussed. Sections~\ref{sec:T} contains the bijection from the set of weighted ordered trees to the set of $F$-paths.
In Section~\ref{sec:enum} we provide the number of objects in Theorem~\ref{thm:main} with respect to several statistics.
In the last section, we discuss the direct sums of these lattice paths and inversion sequences, 
which are consistent with the well-known decompositions of permutations and ordered trees.


\section{{\it F}-paths and related lattice paths}
\label{sec:F-P-B}

In \( \Z^2 \), a \emph{lattice path of length \( n \) with steps in the set \( S \)} is a sequence of points, 
\[ 
((x_0,y_0), (x_1,y_1), \dots,  (x_{n},y_{n})),
\]
such that each consecutive difference \( s_i=(x_i-x_{i-1},y_i-y_{i-1}) \) lies in \( S \). We also write a lattice path as a sequence of steps, \( s_1s_2 \dots s_{n} \). The starting point of a lattice path is assumed to be \( (0,0) \) or is given by the context. 

\subsection{{\it F}-paths}
\label{sec:F}
Let us define $F$-paths.
\begin{defn}
An \emph{\( F \)-path of length \( n \)} is a lattice path 
\[ Q=((0,0), (x_1,y_1), \dots,  (x_{n},y_{n})) \] 
with steps in the set 
\( F \coloneqq \{ (a,b) \deli a\geq 1, ~b\leq 1\} \cup \{ (0,1) \} \) 
that satisfies 
\( x_i \leq y_i \) for each \( i=1,2,\dots, n \). 
We denote by \(\Fp_{n} \) the set of \( F \)-paths of length \( n \). 
\end{defn}

Define the height of $Q =((0,0), (x_1,y_1), \dots,  (x_{n},y_{n})) \in\Fp_{n}$
as \( y_{n}-x_{n} \), and denote it by $\height(Q)$.
Let $\north(Q)$ be the number of steps $(0,1)$,
$\aone(Q)$ be the number of steps $(a,b)$ with $a=1$ in $Q$,
and $\bone(Q)$ be the number of steps $(a,b)$ with $b=1$ in $Q$.
It is easy to show that $$\height(Q) \le \north(Q) \le \bone(Q)$$ 
for any $F$-path $Q$.
Figure~\ref{fig:F2} shows all the \( F \)-paths of length \( 2 \). 
Note that if \( \height(Q)=m \), then it can be written uniquely as 
\[ Q=Q_1 (0,1) Q_2 (0,1) \dots (0,1) Q_{m+1}, \] 
where each segment \( Q_i \) is an \( F \)-path with \( \height(Q_i)=0 \) allowed to be empty.  
\begin{figure}[t]
\centering

\begin{tikzpicture}[scale=0.7]
\foreach \i in {0,1,2}
\foreach \j in {0,1,2}
\filldraw[fill=gray!70, color=gray!70] (\i,\j) circle (2.5pt);

\draw (0,0) -- (2,0);
\draw (0,0) -- (0,2);
\draw[color=gray!70] (0,0) -- (2,2);

\coordinate (0) at (0,0);
\coordinate (1) at (0,1);
\coordinate (2) at (1,1);

\foreach \i in {0,1,2}
\filldraw (\i) circle (2.5pt);
\draw[ultra thick]
(0) -- (1) -- (2);

\node at (1,-0.5) {\( Q^{(1)} \)};
\end{tikzpicture}
\qquad
\begin{tikzpicture}[scale=0.7]
\foreach \i in {0,1,2}
\foreach \j in {0,1,2}
\filldraw[fill=gray!70, color=gray!70] (\i,\j) circle (2.5pt);

\draw (0,0) -- (2,0);
\draw (0,0) -- (0,2);
\draw[color=gray!70] (0,0) -- (2,2);

\coordinate (0) at (0,0);
\coordinate (1) at (0,1);
\coordinate (2) at (2,2);

\foreach \i in {0,1,2}
\filldraw (\i) circle (2.5pt);
\draw[ultra thick]
(0) -- (1) -- (2);

\node at (1,-0.5) {\( Q^{(2)} \)};
\end{tikzpicture}
\qquad
\begin{tikzpicture}[scale=0.7]
\foreach \i in {0,1,2}
\foreach \j in {0,1,2}
\filldraw[fill=gray!70, color=gray!70] (\i,\j) circle (2.5pt);

\draw (0,0) -- (2,0);
\draw (0,0) -- (0,2);
\draw[color=gray!70] (0,0) -- (2,2);

\coordinate (0) at (0,0);
\coordinate (1) at (1,1);
\coordinate (2) at (2,2);

\foreach \i in {0,1,2}
\filldraw (\i) circle (2.5pt);
\draw[ultra thick]
(0) -- (1) -- (2);

\node at (1,-0.5) {\( Q^{(3)} \)};
\end{tikzpicture}
\qquad
\begin{tikzpicture}[scale=0.7]
\foreach \i in {0,1,2}
\foreach \j in {0,1,2}
\filldraw[fill=gray!70, color=gray!70] (\i,\j) circle (2.5pt);

\draw (0,0) -- (2,0);
\draw (0,0) -- (0,2);
\draw[color=gray!70] (0,0) -- (2,2);

\coordinate (0) at (0,0);
\coordinate (1) at (0,1);
\coordinate (2) at (1,2);

\foreach \i in {0,1,2}
\filldraw (\i) circle (2.5pt);
\draw[ultra thick]
(0) -- (1) -- (2);

\node at (1,-0.5) {\( Q^{(4)} \)};
\end{tikzpicture}
\qquad
\begin{tikzpicture}[scale=0.7]
\foreach \i in {0,1,2}
\foreach \j in {0,1,2}
\filldraw[fill=gray!70, color=gray!70] (\i,\j) circle (2.5pt);

\draw (0,0) -- (2,0);
\draw (0,0) -- (0,2);
\draw[color=gray!70] (0,0) -- (2,2);

\coordinate (0) at (0,0);
\coordinate (1) at (1,1);
\coordinate (2) at (1,2);

\foreach \i in {0,1,2}
\filldraw (\i) circle (2.5pt);
\draw[ultra thick]
(0) -- (1) -- (2);

\node at (1,-0.5) {\( Q^{(5)} \)};
\end{tikzpicture}
\qquad
\begin{tikzpicture}[scale=0.7]
\foreach \i in {0,1,2}
\foreach \j in {0,1,2}
\filldraw[fill=gray!70, color=gray!70] (\i,\j) circle (2.5pt);

\draw (0,0) -- (2,0);
\draw (0,0) -- (0,2);
\draw[color=gray!70] (0,0) -- (2,2);

\coordinate (0) at (0,0);
\coordinate (1) at (0,1);
\coordinate (2) at (0,2);

\foreach \i in {0,1,2}
\filldraw (\i) circle (2.5pt);
\draw[ultra thick]
(0) -- (1) -- (2);

\node at (1,-0.5) {\( Q^{(6)} \)};
\end{tikzpicture}

\caption{The \( F \)-paths of length \( 2 \).}\label{fig:F2}
\end{figure}
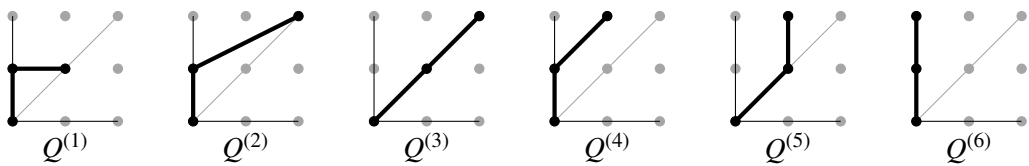 

There exists an involution $I$ on the set 
$F=\set{ (a,b) \deli a\geq 1, ~b\leq 1} \cup \set{ (0,1) }$ defined by 
$$
I(a,b) \coloneqq \begin{cases}
(0,1) &\text{if $(a, b) = (0, 1),$}\\
(2-b, 2-a) &\text{otherwise.}
\end{cases}
$$
Let us define an involution $\phi_{\rm F}$ on $\Fp_{n}$ by 
$\phi_{\rm F}(s_1\dots s_n) \coloneqq I(s_1) \dots I(s_n)$.
The involution $\phi_{\rm F}$ satisfies
\begin{align*}
\height(\phi_{\rm F}(Q)) &= \height(Q),\\
\north(\phi_{\rm F}(Q)) &= \north(Q),\\
\aone(\phi_{\rm F}(Q)) &= \bone(Q) - \north(Q),\\
\bone(\phi_{\rm F}(Q)) &= \aone(Q) + \north(Q).
\end{align*}


\subsection{Schr\"{o}der paths without triple descents}
\label{sec:P}
A Schr\"{o}der path of semilength \( n \) is a lattice path from \( (0,0) \) to \( (2n,0) \) that does not go below the \( x \)-axis and consists of up steps \( u=(1,1) \), down steps \( d=(1,-1) \), and horizontal steps \( h=(2,0) \). 
For a Schr\"{o}der path \( P \), 
two consecutive steps \( ud \), $du$, \( uu \) and \( dd \) are called a \emph{peak}, a \emph{valley}, a \emph{double ascent}, and a \emph{double descent}, respectively. 
The number of peaks in \( P \) is denoted by \( \peak(P) \). 
The total number of horizontal steps and double descents in $P$ is denoted by $\hdd(P)$.
A \emph{small Schr\"{o}der path} is a Schr\"{o}der path having no horizontal step on the \( x \)-axis. For a Schr\"{o}der path \( P \), let \( \comp(P) \) denote the number of horizontal steps on the \( x \)-axis. 
If \( \comp(P)=m \), then we decompose \( P \) into \( 2m+1 \) segments as \( P=P_1h P_2 \dots h P_{m+1} \), where each segment \( P_i \) is a small Schr\"{o}der path allowed to be empty.

For a Schr\"{o}der path \( P \), three consecutive steps \( ddd \) is called a \emph{triple descent}. 
We denote by \( \schp_n \) the set of Schr\"{o}der paths of semilength \( n \) without triple descents.
Figure~\ref{fig:S2} shows all the Schr\"{o}der paths in \( \schp_2 \).

\begin{figure}[t]
\centering

\begin{tikzpicture}[scale=0.7]
\foreach \i in {0,1,2,3,4}
\foreach \j in {0,1,2}
\filldraw[fill=gray!70, color=gray!70] (0.5*\i,0.5*\j) circle (2.5pt);

\draw (0,0) -- (2,0);
\draw (0,0) -- (0,1);

\coordinate (0) at (0,0);
\coordinate (1) at (1/2,1/2);
\coordinate (2) at (2/2,2/2);
\coordinate (3) at (3/2,1/2);
\coordinate (4) at (4/2,0/2);

\foreach \i in {0,1,2,3,4}
\filldraw (\i) circle (2.5pt);
\draw[ultra thick]
(0) -- (1) -- (2) -- (3) --(4);

\node at (1,-0.5) {\( P^{(1)} \)};
\end{tikzpicture}
\qquad 
\begin{tikzpicture}[scale=0.7]
\foreach \i in {0,1,2,3,4}
\foreach \j in {0,1,2}
\filldraw[fill=gray!70, color=gray!70] (0.5*\i,0.5*\j) circle (2.5pt);

\draw (0,0) -- (2,0);
\draw (0,0) -- (0,1);

\coordinate (0) at (0,0);
\coordinate (1) at (1/2,1/2);
\coordinate (2) at (3/2,1/2);
\coordinate (3) at (4/2,0/2);

\foreach \i in {0,1,2,3}
\filldraw (\i) circle (2.5pt);
\draw[ultra thick]
(0) -- (1) -- (2) -- (3);

\node at (1,-0.5) {\( P^{(2)} \)};
\end{tikzpicture}
\qquad
\begin{tikzpicture}[scale=0.7]
\foreach \i in {0,1,2,3,4}
\foreach \j in {0,1,2}
\filldraw[fill=gray!70, color=gray!70] (0.5*\i,0.5*\j) circle (2.5pt);

\draw (0,0) -- (2,0);
\draw (0,0) -- (0,1);

\coordinate (0) at (0,0);
\coordinate (1) at (1/2,1/2);
\coordinate (2) at (2/2,0/2);
\coordinate (3) at (3/2,1/2);
\coordinate (4) at (4/2,0/2);

\foreach \i in {0,1,2,3,4}
\filldraw (\i) circle (2.5pt);
\draw[ultra thick]
(0) -- (1) -- (2) -- (3) --(4);

\node at (1,-0.5) {\( P^{(3)} \)};
\end{tikzpicture}
\qquad
\begin{tikzpicture}[scale=0.7]
\foreach \i in {0,1,2,3,4}
\foreach \j in {0,1,2}
\filldraw[fill=gray!70, color=gray!70] (0.5*\i,0.5*\j) circle (2.5pt);

\draw (0,0) -- (2,0);
\draw (0,0) -- (0,1);

\coordinate (0) at (0,0);
\coordinate (1) at (2/2,0/2);
\coordinate (2) at (3/2,1/2);
\coordinate (3) at (4/2,0/2);

\foreach \i in {0,1,2,3}
\filldraw (\i) circle (2.5pt);
\draw[ultra thick]
(0) -- (1) -- (2) -- (3);

\node at (1,-0.5) {\( P^{(4)} \)};
\end{tikzpicture}
\qquad
\begin{tikzpicture}[scale=0.7]
\foreach \i in {0,1,2,3,4}
\foreach \j in {0,1,2}
\filldraw[fill=gray!70, color=gray!70] (0.5*\i,0.5*\j) circle (2.5pt);

\draw (0,0) -- (2,0);
\draw (0,0) -- (0,1);

\coordinate (0) at (0,0);
\coordinate (1) at (1/2,1/2);
\coordinate (2) at (2/2,0/2);
\coordinate (3) at (4/2,0/2);

\foreach \i in {0,1,2,3}
\filldraw (\i) circle (2.5pt);
\draw[ultra thick]
(0) -- (1) -- (2) -- (3);

\node at (1,-0.5) {\( P^{(5)} \)};
\end{tikzpicture}
\qquad
\begin{tikzpicture}[scale=0.7]
\foreach \i in {0,1,2,3,4}
\foreach \j in {0,1,2}
\filldraw[fill=gray!70, color=gray!70] (0.5*\i,0.5*\j) circle (2.5pt);

\draw (0,0) -- (2,0);
\draw (0,0) -- (0,1);

\coordinate (0) at (0,0);
\coordinate (1) at (2/2,0/2);
\coordinate (2) at (4/2,0/2);

\foreach \i in {0,1,2}
\filldraw (\i) circle (2.5pt);
\draw[ultra thick]
(0) -- (1) -- (2);

\node at (1,-0.5) {\( P^{(6)} \)};
\end{tikzpicture}
\caption{The Schr\"{o}der paths of semilength $2$ without triple descents.}\label{fig:S2}
\end{figure}
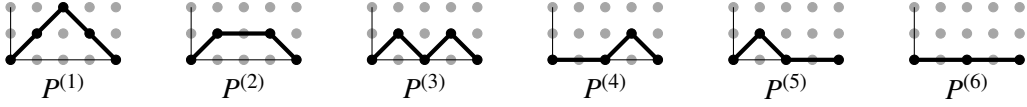 

Let \( \mathcal{L}A \) denote the set of paths \( L\in\mathcal{L} \) that ends with a sequence of steps \( s\in A \), where \( \mathcal{L} \) is a set of lattice paths and \( A \) is a set of sequences of steps. Note that \( \schp_n \) can be partitioned into five subsets \( \schp_n \{ s \} \) for \( s\in \{ h, ud, hd, udd, hdd \} \). The following theorem gives several pairs of one-to-one correspondence sets.

\begin{thm}\label{thm:P-F}
For a nonnegative integer \( n \), the following hold.
\begin{enumerate} 
\item \( |\schp_n\{ h \}|=|\mathcal{F}_n\{ (0,1) \}|\).
\item \( |\schp_n\{ ud \}|=|\mathcal{F}_n\{ (1,1) \}|\).
\item \( |\schp_n\{ hd \}|=|\mathcal{F}_n\{ (a,1) \deli a\geq 2 \}|\).
\item \( |\schp_n\{ udd \}|=|\mathcal{F}_n\{ (1,b) \deli b\leq 0 \}|\).
\item \( |\schp_n\{ hdd \}|=|\mathcal{F}_n\{ (a,b) \deli a\geq 2, ~b\leq 0 \}|\).
\end{enumerate}
Consequently, \( |\schp_n|=|\mathcal{F}_n| \). 
\end{thm}

\begin{proof} 
From Schr\"{o}der paths without triple descents, 
we construct a bijection to $F$-paths. Define
\[ 
\phi_{\rm P} : \bigcup_{n\geq 0}\schp_n \to \bigcup_{n\geq 0} \mathcal{F}_n 
\] 
such that
$$\comp(P)=\height(\phi_{\rm P}(P))$$
recursively as follows.
Let $\phi_{\rm P}(\emptyset) \coloneqq \emptyset$ with $\comp(\emptyset)=\height(\emptyset)=0$. 
Now suppose that $n \geq 1$ and $P \in \schp_n$.

\begin{enumerate}[(1)]
\item \label{P1} If \( P \in \schp_n\{ h \} \), 
then we can decompose $P$ into two segments as $P=X h$ 
such that \( X \) is a Schr\"{o}der path without triple descents.
Define 
$$\phi_{\rm P}(P) \coloneqq \phi_{\rm P}(\hat{P})(0,1),$$
where $\hat{P} \coloneqq X$.
Due to 
\( \hat{P}\in \schp_{n-1} \), it is clear that \( \phi_{\rm P}(P) \in \Fp_{n} \) and
$\comp(\hat{P})=\height(\phi_{\rm P}(\hat{P}))$. 
Thus,
$$\comp(P)=\comp(\hat{P})+1=\height(\phi_{\rm P}(\hat{P}))+1=\height(\phi_{\rm P}(P)).$$

\[
\begin{tikzpicture}[scale=0.35]

\draw [-stealth] (0,0) -- (20,0);
\draw [-stealth] (0,0) -- (0,4);

\coordinate (0) at (0,0);
\coordinate (1) at (16,0);
\coordinate (2) at (18,0);

\coordinate (a) at (0,3);
\coordinate (b) at (16,3);

\foreach \i in {0,...,2}{
\filldraw (\i) circle (5pt);
}


\draw[ultra thick] 
	(0) .. controls (a) and (b) .. (1) -- (2);
\draw[thick] 
	(0) -- (1);

\node at (8, 1) {\( X \)};

\node at (4, -1) [right] {\( P = X h \)};
\end{tikzpicture}
\qquad
\begin{tikzpicture}[scale=0.35]

\draw [-stealth] (0,0) -- (18,0);
\draw [-stealth] (0,0) -- (0,4);

\coordinate (0) at (0,0);
\coordinate (1) at (16,0);

\coordinate (a) at (0,3);
\coordinate (b) at (16,3);

\foreach \i in {0,...,1}{
\filldraw (\i) circle (5pt);
}


\draw[ultra thick] 
	(0) .. controls (a) and (b) .. (1);
\draw[thick] 
	(0) -- (1);

\node at (8, 1) {\( X \)};

\node at (4, -1) [right] {\( \hat{P} \coloneqq X \)};
\end{tikzpicture}
\]

\item \label{P2} If \( P \in \schp_n\{ ud \} \), 
then we can decompose $P$ into three segments as $P=X u d$ 
such that \( X \) is a Schr\"{o}der path without triple descents.
Define 
$$\phi_{\rm P}(P) \coloneqq \phi_{\rm P}(\hat{P})(1,1),$$
where $\hat{P} \coloneqq X$.
Again, due to 
\( \hat{P}\in \schp_{n-1} \), we have \( \phi_{\rm P}(P) \in \Fp_{n} \) and 
$$\comp(P)=\comp(\hat{P})=\height(\phi_{\rm P}(\hat{P}))=\height(\phi_{\rm P}(P)).$$

\[
\begin{tikzpicture}[scale=0.35]

\draw [-stealth] (0,0) -- (20,0);
\draw [-stealth] (0,0) -- (0,4);

\coordinate (0) at (0,0);
\coordinate (1) at (16,0);
\coordinate (2) at (17,1);
\coordinate (3) at (18,0);

\coordinate (a) at (0,3);
\coordinate (b) at (16,3);

\foreach \i in {0,...,3}{
\filldraw (\i) circle (5pt);
}


\draw[ultra thick] 
	(0) .. controls (a) and (b) .. (1) -- (2) -- (3);
\draw[thick] 
	(0) -- (1);

\node at (8, 1) {\( X \)};

\node at (4, -1) [right] {\( P = X u d \)};
\end{tikzpicture}
\qquad
\begin{tikzpicture}[scale=0.35]

\draw [-stealth] (0,0) -- (18,0);
\draw [-stealth] (0,0) -- (0,4);

\coordinate (0) at (0,0);
\coordinate (1) at (16,0);

\coordinate (a) at (0,3);
\coordinate (b) at (16,3);

\foreach \i in {0,...,1}{
\filldraw (\i) circle (5pt);
}


\draw[ultra thick] 
	(0) .. controls (a) and (b) .. (1);
\draw[thick] 
	(0) -- (1);

\node at (8, 1) {\( X \)};

\node at (4, -1) [right] {\( \hat{P} \coloneqq X \)};
\end{tikzpicture}
\]

\item \label{P3}
If \( P \in \schp_n\{ hd \} \), 
then we can decompose $P$ into five segments as $P=X u Y h d$
such that \( X \) and \( Y \) are Schr\"{o}der paths without triple descents.
Define 
$$\phi_{\rm P}(P) \coloneqq \phi_{\rm P}(\hat{P})(k+2,1),$$
where $\hat{P} \coloneqq XhY$ and $\comp(Y) = k$.
Due to
\( \hat{P}\in \schp_{n-1} \) and \( \height(\phi_{\rm P}(\hat{P}))\geq k+1 \), we have \( \phi_{\rm P}(P) \in \Fp_{n} \) and
$$\comp(P)=\comp(\hat{P})-(k+1)=\height(\phi_{\rm P}(\hat{P}))-(k+1)=\height(\phi_{\rm P}(P)).$$

\[
\begin{tikzpicture}[scale=0.35]

\draw [-stealth] (0,0) -- (20,0);
\draw [-stealth] (0,0) -- (0,4);

\coordinate (0) at (0,0);
\coordinate (1) at (7,0);
\coordinate (2) at (8,1);
\coordinate (3) at (15,1);
\coordinate (4) at (17,1);
\coordinate (5) at (18,0);

\coordinate (a) at (0,3);
\coordinate (b) at (7,3);
\coordinate (c) at (8,4);
\coordinate (d) at (15,4);

\foreach \i in {0,...,5}{
\filldraw (\i) circle (5pt);
}


\draw[ultra thick] 
	(0) .. controls (a) and (b) .. (1) --
	(2) .. controls (c) and (d) .. (3) --
	(4) -- (5);
\draw[thick] 
	(0) -- (1)
	(2) -- (3);

\node at (7/2, 1) {\( X \)};
\node at (23/2, 2) {\( Y \)};

\node at (4, -1) [right] {\( P = X u Y h d \)};
\end{tikzpicture}
\qquad
\begin{tikzpicture}[scale=0.35]

\draw [-stealth] (0,0) -- (18,0);
\draw [-stealth] (0,0) -- (0,4);

\coordinate (0) at (0,0);
\coordinate (1) at (7,0);
\coordinate (2) at (9,0);
\coordinate (3) at (16,0);

\coordinate (a) at (0,3);
\coordinate (b) at (7,3);
\coordinate (c) at (9,3);
\coordinate (d) at (16,3);

\foreach \i in {0,...,3}{
\filldraw (\i) circle (5pt);
}


\draw[ultra thick] 
	(0) .. controls (a) and (b) .. (1) --
	(2) .. controls (c) and (d) .. (3);
\draw[thick] 
	(0) -- (1)
	(2) -- (3);
	
\node at (7/2, 1) {\( X \)};
\node at (25/2, 1) {\( Y \)};

\node at (4, -1) [right] {\( \hat{P} \coloneqq X h Y \)};
\end{tikzpicture}
\]

\item \label{P4} If \( P \in \schp_n\{ udd \} \), 
then we can decompose $P$ into six segments as $P=X u Z u d d$
such that \( X \) and \( Z \) are Schr\"{o}der paths without triple descents.
Define 
$$\phi_{\rm P}(P) \coloneqq \phi_{\rm P}(\hat{P})(1,-j),$$
where $\hat{P} \coloneqq X h Z$ and $\comp(Z) = j$.
Again, \( \hat{P} \in \schp_{n-1} \) and \( \height(\phi_{\rm P}(\hat{P}))\geq j+1 \) yields that \( \phi_{\rm P}(P) \in \Fp_{n} \) and
$$\comp(P)=\comp(\hat{P})-(j+1)=\height(\phi_{\rm P}(\hat{P}))-(j+1)=\height(\phi_{\rm P}(P)).$$

\[
\begin{tikzpicture}[scale=0.35]

\draw [-stealth] (0,0) -- (20,0);
\draw [-stealth] (0,0) -- (0,4);

\coordinate (0) at (0,0);
\coordinate (1) at (7,0);
\coordinate (2) at (8,1);
\coordinate (3) at (15,1);
\coordinate (4) at (16,2);
\coordinate (5) at (17,1);
\coordinate (6) at (18,0);

\coordinate (a) at (0,3);
\coordinate (b) at (7,3);
\coordinate (c) at (8,4);
\coordinate (d) at (15,4);

\foreach \i in {0,...,6}{
\filldraw (\i) circle (5pt);
}


\draw[ultra thick] 
	(0) .. controls (a) and (b) .. (1) -- 
	(2) .. controls (c) and (d) .. (3) --
	(4) -- (5) -- (6);

\draw[thick] 
	(0) -- (1)
	(2) -- (3);

\node at (7/2, 1) {\( X \)};
\node at (23/2, 2) {\( Z \)};

\node at (4, -1) [right] {\( P = X u Z u d d \)};
\end{tikzpicture}
\qquad
\begin{tikzpicture}[scale=0.35]

\draw [-stealth] (0,0) -- (18,0);
\draw [-stealth] (0,0) -- (0,4);

\coordinate (0) at (0,0);
\coordinate (1) at (7,0);
\coordinate (2) at (9,0);
\coordinate (3) at (16,0);

\coordinate (a) at (0,3);
\coordinate (b) at (7,3);
\coordinate (c) at (9,3);
\coordinate (d) at (16,3);

\foreach \i in {0,...,3}{
\filldraw (\i) circle (5pt);
}


\draw[ultra thick] 
	(0) .. controls (a) and (b) .. (1) --
	(2) .. controls (c) and (d) .. (3);

\draw[thick] 
	(0) -- (1)
	(2) -- (3);

\node at (7/2, 1) {\( X \)};
\node at (25/2, 1) {\( Z \)};

\node at (4, -1) [right] {\( \hat{P} \coloneqq X h Z \)};
\end{tikzpicture}
\]

\item \label{P5} If \( P \in \schp_n\{ hdd \} \), 
then we can decompose $P$ into eight segments as $P=X u Y u Z h d d$
such that \( X \), \( Y \), and \( Z \) are Schr\"{o}der paths without triple descents.
Define 
$$\phi_{\rm P}(P) \coloneqq \phi_{\rm P}(\hat{P})(k+2,-j),$$
where $\hat{P} \coloneqq XhYhZ$, 
$\comp(Y) = k$, and $\comp(Z) = j$.
Due to
\( \hat{P} \in \schp_{n-1} \) and \( \height(\phi_{\rm P}(\hat{P}))\geq k+j+2 \), it follows that \( \phi_{\rm P}(P) \in \Fp_{n} \) and
$$\comp(P)=\comp(\hat{P})-(k+j+2)=\height(\phi_{\rm P}(\hat{P}))-(k+j+2)=\height(\phi_{\rm P}(P)).$$
\[
\begin{tikzpicture}[scale=0.35]

\draw [-stealth] (0,0) -- (20,0);
\draw [-stealth] (0,0) -- (0,4);

\coordinate (0) at (0,0);
\coordinate (1) at (4,0);
\coordinate (2) at (5,1);
\coordinate (3) at (9,1);
\coordinate (4) at (10,2);
\coordinate (5) at (14,2);
\coordinate (6) at (16,2);
\coordinate (7) at (17,1);
\coordinate (8) at (18,0);

\coordinate (a) at (0,3);
\coordinate (b) at (4,3);
\coordinate (c) at (5,4);
\coordinate (d) at (9,4);
\coordinate (e) at (10,5);
\coordinate (f) at (14,5);

\foreach \i in {0,...,8}{
\filldraw (\i) circle (5pt);
}


\draw[ultra thick] 
	(0) .. controls (a) and (b) .. (1) --
	(2) .. controls (c) and (d) .. (3) --
	(4) .. controls (e) and (f) .. (5) --
	(6) -- (7) -- (8);

\draw[thick] 
	(0) -- (1)
	(2) -- (3)
	(4) -- (5);

\node at (2, 1) {\( X \)};
\node at (7, 2) {\( Y \)};
\node at (12, 3) {\( Z \)};

\node at (4, -1) [right] {\( P = X u Y u Z h d d \)};
\end{tikzpicture}
\qquad
\begin{tikzpicture}[scale=0.35]

\draw [-stealth] (0,0) -- (18,0);
\draw [-stealth] (0,0) -- (0,4);

\coordinate (0) at (0,0);
\coordinate (1) at (4,0);
\coordinate (2) at (6,0);
\coordinate (3) at (10,0);
\coordinate (4) at (12,0);
\coordinate (5) at (16,0);

\coordinate (a) at (0,3);
\coordinate (b) at (4,3);
\coordinate (c) at (6,3);
\coordinate (d) at (10,3);
\coordinate (e) at (12,3);
\coordinate (f) at (16,3);

\foreach \i in {0,...,5}{
\filldraw (\i) circle (5pt);
}


\draw[ultra thick] 
	(0) .. controls (a) and (b) .. (1) --
	(2) .. controls (c) and (d) .. (3) --
	(4) .. controls (e) and (f) .. (5);

\draw[thick] 
	(0) -- (1)
	(2) -- (3)
	(4) -- (5);

\node at (2, 1) {\( X \)};
\node at (8, 1) {\( Y \)};
\node at (14, 1) {\( Z \)};

\node at (4, -1) [right] {\( \hat{P} \coloneqq X h Y h Z \)};
\end{tikzpicture}
\]

\end{enumerate}

Therefore, \( \phi_{\rm P} \) is a well-defined map with the property \( \comp(P)=\height(\phi_{\rm P}(P)) \).
Since the reverse process is clear, 
we can construct the inverse map $\psi_{\rm P}$ of $\phi_{\rm P}$ recursively as follows.
Let $\psi_{\rm P}(\emptyset) \coloneqq \emptyset$. 
Now suppose $n \geq 1$. For $Q \in \mathcal{F}_n$, $Q$ is one of the following five.
\begin{enumerate}[(1)]
\item If \( Q =\hat{Q}(0,1) \in \mathcal{F}_n\{ (0,1) \} \), then 
set \( X=\psi_{\rm P}(\hat{Q}) \) and define
$$\psi_{\rm P}(Q) \coloneqq X h.$$

\item If \( Q =\hat{Q}(1,1) \in \mathcal{F}_n\{ (1,1) \} \), then 
set \( X=\psi_{\rm P}(\hat{Q}) \) and define
$$\psi_{\rm P}(Q) \coloneqq Xud.$$

\item If \( Q =\hat{Q}(k+2,1) \in \mathcal{F}_n\{ (k+2,1) \} \) for 
some
$k \geq 0$, then 
$\comp(\psi_{\rm P}(\hat{Q})) = \height(\hat{Q})=\height(Q)+(k+1) \geq k+1$.
Hence, 
$\psi_{\rm P}(\hat{Q})$ can be expressed as 
$\psi_{\rm P}(\hat{Q}) = X h Y$ 
such that \( X \) and \( Y \) are Schr\"{o}der paths without triple descents
with $\comp(Y) = k$.
Define 
$$\psi_{\rm P}(Q) \coloneqq X u Y h d.$$

\item If \( Q =\hat{Q}(1,-j) \in \mathcal{F}_n\{ (1,-j) \} \) for 
some $j \geq 0$, then 
$\comp(\psi_{\rm P}(\hat{Q})) = \height(\hat{Q})=\height(Q)+(j+1) \geq j+1$.
Again, 
$\psi_{\rm P}(\hat{Q})$ can be expressed as 
$\psi_{\rm P}(\hat{Q}) = X h Z$
such that \( X \) and \( Z \) are Schr\"{o}der paths without triple descents
with $\comp(Z) = j$.
Define 
$$\psi_{\rm P}(Q) \coloneqq X u Z u d d.$$

\item If \( Q =\hat{Q}(k+2, -j) \in \mathcal{F}_n\{ (k+2, -j) \} \) 
for 
some $k \geq 0$ and $j \geq 0$, 
then 
$\comp(\psi_{\rm P}(\hat{Q})) = \height(\hat{Q})=\height(Q)+(k+j+2) \geq k+j+2$.
In this case, 
$\psi_{\rm P}(\hat{Q})$ can be expressed as 
$\psi_{\rm P}(\hat{Q}) = XhYhZ$ 
such that \( X \), \( Y \), and \( Z \) are Schr\"{o}der paths without triple descents
with $\comp(Y) = k$ and $\comp(Z) = j$.
Define 
$$\psi_{\rm P}(Q) \coloneqq  XuYuZ hdd.$$
\end{enumerate} 

Hence, 
$\psi_{\rm P}$ is a well-defined map 
such that $\psi_{\rm P} \circ \phi_{\rm P}$ is the identity on $\bigcup_{n\geq 0}\schp_n$.
So \( \phi_{\rm P} \) and \( \psi_{\rm P} \) are bijective, and this completes the proof. 
\end{proof}

\begin{rem}\label{rem:P-F}
\begin{enumerate}[(1)]
\item
Clearly, $\hdd(\emptyset)=\north(\emptyset).$
It is 
easy to check that
$\hdd(P)=\hdd(\hat{P})+1$ only in the case \eqref{P1} 
of the proof of Theorem~\ref{thm:P-F}
and $\hdd(P)=\hdd(\hat{P})$ in the rest of the cases. 
Therefore, for all $P \in \bigcup_{n\geq 0}\schp_n$,
we obtain $$\hdd(P)=\north(\phi_{\rm P}(P)).$$
\item
Clearly, $\peak(\emptyset)=\aone(\emptyset).$
Since it satisfies $\peak(P)=\peak(\hat{P})+1$ 
only in the cases \eqref{P2} and \eqref{P4} 
of the proof of Theorem~\ref{thm:P-F}
and $\peak(P)=\peak(\hat{P})$ in the rest of the cases, 
for all $P \in \bigcup_{n\geq 0}\schp_n$,
we obtain $$\peak(P)=\aone(\phi_{\rm P}(P)).$$
\end{enumerate}
\end{rem}

\begin{example}
There are six Schr\"{o}der paths in \( \schp_2 \) as shown in Figure~\ref{fig:S2}. 
From the construction of \( \phi_{\rm P} \), for each $i$, we have 
\begin{align*}
\phi_{\rm P}(P^{(i)})&=Q^{(i)},
\end{align*}
with
$\comp(P^{(i)})=\height(Q^{(i)})$,
$\hdd(P^{(i)})=\north(Q^{(i)})$, and
$\peak(P^{(i)})=\aone(Q^{(i)})$,
where each $P^{(i)}\in\schp_2$ and $Q^{(i)}\in\Fp_2$ are given in Figures~\ref{fig:F2} and~\ref{fig:S2}.
\end{example}


\subsection{Restricted bicolored Dyck paths}
\label{sec:B}
A Schr\"{o}der path having no horizontal step is called a Dyck path. 
We color each down step of a Dyck path by black or red and call such a Dyck path a \emph{bicolored Dyck path}. We denote down steps colored black and red by \( \db \) and \( \dr \), respectively. Let us denote by \( \mathcal{B}_{n} \) the set of bicolored Dyck paths \( B \) of semilength \( n \) of the form
\[
B=u^{i_1} {\dr}^{j_1} {\db}^{k_1}  \dots u^{i_{\ell-1}} \dr^{j_{\ell-1}} \db^{k_{\ell-1}} u^{i_{\ell}} \dr^{j_{\ell}},
\]
where 
\[ 
i_1+\cdots+i_{\ell-1}+i_{\ell}=j_1+\cdots+j_{\ell-1}+j_{\ell}+k_1+\cdots+k_{\ell-1}=n
\]
for positive exponents 
\( i_m \) and \( k_m \) and nonnegative exponents \( j_m \) with a positive integer \( \ell \). We call such a bicolored Dyck path $B$ a \emph{restricted bicolored Dyck path} and write \( \last(B)=j_{\ell} \).
Let $\dasc(B)$ be the number of double ascents $uu$ in $B$. 
We define the statistic \( \bval(B) \) to be the number of occurrences of three consecutive steps of the forms \( u\db u\) or \( \dr \db u \) in $B$.
Figure~\ref{fig:B3} shows all the restricted bicolored Dyck paths of semilength \( 3 \). Here, 
\(\bval(B^{(2)})=\bval(B^{(6)})=0 \), \( \bval(B^{(1)})=\bval(B^{(4)})=\bval(B^{(5)})=1 \), and \( \bval(B^{(3)})=2 \).      

\begin{figure}[t]
\centering

\begin{tikzpicture}[scale=0.5]
\foreach \i in {0,1,2,3,4,5,6}
\foreach \j in {0,1,2}
\filldraw[fill=gray!70, color=gray!70] (0.5*\i,0.5*\j) circle (2.5pt);

\draw (0,0) -- (3,0);
\draw (0,0) -- (0,1);

\coordinate (0) at (0,0);
\coordinate (1) at (1/2,1/2);
\coordinate (2) at (2/2,2/2);
\coordinate (3) at (3/2,1/2);
\coordinate (4) at (4/2,0/2);
\coordinate (5) at (5/2,1/2);
\coordinate (6) at (6/2,0/2);

\draw[ultra thick]
(0) -- (1) -- (2) -- (3) -- (4) -- (5) -- (6);
\draw[ultra thick, color=red]
(2)-- (3) (5) -- (6);

\foreach \i in {0,1,2,3,4,5,6}
\filldraw (\i) circle (2.5pt);

\node at (1.5,-0.6) {\( B^{(1)} \)};
\end{tikzpicture}
\qquad 
\begin{tikzpicture}[scale=0.5]
\foreach \i in {0,1,2,3,4,5,6}
\foreach \j in {0,1,2}
\filldraw[fill=gray!70, color=gray!70] (0.5*\i,0.5*\j) circle (2.5pt);

\draw (0,0) -- (3,0);
\draw (0,0) -- (0,1);

\coordinate (0) at (0,0);
\coordinate (1) at (1/2,1/2);
\coordinate (2) at (2/2,2/2);
\coordinate (3) at (3/2,1/2);
\coordinate (4) at (4/2,0/2);
\coordinate (5) at (5/2,1/2);
\coordinate (6) at (6/2,0/2);

\draw[ultra thick]
(0) -- (1) -- (2) -- (3) -- (4) -- (5) -- (6);
\draw[ultra thick, color=red]
(5) -- (6);

\foreach \i in {0,1,2,3,4,5,6}
\filldraw (\i) circle (2.5pt);

\node at (1.5,-0.6) {\( B^{(2)} \)};
\end{tikzpicture}
\qquad
\begin{tikzpicture}[scale=0.5]
\foreach \i in {0,1,2,3,4,5,6}
\foreach \j in {0,1,2}
\filldraw[fill=gray!70, color=gray!70] (0.5*\i,0.5*\j) circle (2.5pt);

\draw (0,0) -- (3,0);
\draw (0,0) -- (0,1);

\coordinate (0) at (0,0);
\coordinate (1) at (1/2,1/2);
\coordinate (2) at (2/2,0/2);
\coordinate (3) at (3/2,1/2);
\coordinate (4) at (4/2,0/2);
\coordinate (5) at (5/2,1/2);
\coordinate (6) at (6/2,0/2);

\draw[ultra thick]
(0) -- (1) -- (2) -- (3) -- (4) -- (5) -- (6);

\draw[ultra thick, color=red]
(5) -- (6);

\foreach \i in {0,1,2,3,4,5,6}
\filldraw (\i) circle (2.5pt);

\node at (1.5,-0.6) {\( B^{(3)} \)};
\end{tikzpicture}
\qquad
\begin{tikzpicture}[scale=0.5]
\foreach \i in {0,1,2,3,4,5,6}
\foreach \j in {0,1,2}
\filldraw[fill=gray!70, color=gray!70] (0.5*\i,0.5*\j) circle (2.5pt);

\draw (0,0) -- (3,0);
\draw (0,0) -- (0,1);

\coordinate (0) at (0,0);
\coordinate (1) at (1/2,1/2);
\coordinate (2) at (2/2,2/2);
\coordinate (3) at (3/2,1/2);
\coordinate (4) at (4/2,2/2);
\coordinate (5) at (5/2,1/2);
\coordinate (6) at (6/2,0/2);

\draw[ultra thick]
(0) -- (1) -- (2) -- (3) -- (4) -- (5) -- (6);
\draw[ultra thick, color=red]
(4)-- (5) -- (6);

\foreach \i in {0,1,2,3,4,5,6}
\filldraw (\i) circle (2.5pt);

\node at (1.5,-0.6) {\( B^{(4)} \)};
\end{tikzpicture}
\qquad
\begin{tikzpicture}[scale=0.5]
\foreach \i in {0,1,2,3,4,5,6}
\foreach \j in {0,1,2}
\filldraw[fill=gray!70, color=gray!70] (0.5*\i,0.5*\j) circle (2.5pt);

\draw (0,0) -- (3,0);
\draw (0,0) -- (0,1);

\coordinate (0) at (0,0);
\coordinate (1) at (1/2,1/2);
\coordinate (2) at (2/2,0/2);
\coordinate (3) at (3/2,1/2);
\coordinate (4) at (4/2,2/2);
\coordinate (5) at (5/2,1/2);
\coordinate (6) at (6/2,0/2);

\draw[ultra thick]
(0) -- (1) -- (2) -- (3) -- (4) -- (5) -- (6);
\draw[ultra thick, color=red]
(4)-- (5) -- (6);

\foreach \i in {0,1,2,3,4,5,6}
\filldraw (\i) circle (2.5pt);

\node at (1.5,-0.6) {\( B^{(5)} \)};
\end{tikzpicture}
\qquad
\begin{tikzpicture}[scale=0.5]
\foreach \i in {0,1,2,3,4,5,6}
\foreach \j in {0,1,2}
\filldraw[fill=gray!70, color=gray!70] (0.5*\i,0.5*\j) circle (2.5pt);

\draw (0,0) -- (3,0);
\draw (0,0) -- (0,1);

\coordinate (0) at (0,0);
\coordinate (1) at (1/2,1/2);
\coordinate (2) at (2/2,2/2);
\coordinate (3) at (3/2,3/2);
\coordinate (4) at (4/2,2/2);
\coordinate (5) at (5/2,1/2);
\coordinate (6) at (6/2,0/2);

\draw[ultra thick]
(0) -- (1) -- (2) -- (3) -- (4) -- (5) -- (6);
\draw[ultra thick, color=red]
(3) -- (4) -- (5) -- (6);

\foreach \i in {0,1,2,3,4,5,6}
\filldraw (\i) circle (2.5pt);

\node at (1.5,-0.6) {\( B^{(6)} \)};
\end{tikzpicture}

\caption{The restricted bicolored Dyck paths of semilength \( 3 \).}\label{fig:B3}
\end{figure}
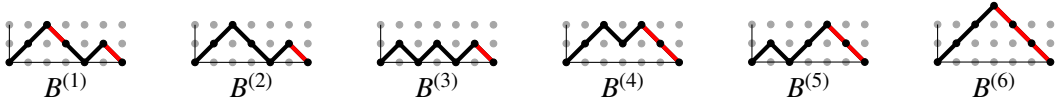 

From restricted bicolored Dyck paths, we construct a bijection to \( F \)-paths. Define
\[ 
\phi_{\rm B} : \bigcup_{n\geq 0}\mathcal{B}_{n+1} \to \bigcup_{n\geq 0} \mathcal{F}_n 
\] 
as follows. Let \( \phi_{\rm B}(u\dr)\coloneqq \emptyset \). Now suppose \( n \geq 1 \). For \( B\in \mathcal{B}_{n+1} \), we decompose \( B \) into \( n+1 \) segments as \( B=B_1 B_2 \dots B_{n+1} \), where \( B_i =u \dr^{1-b_i} \db^{a_i} \). Define \( \phi_{\rm B}(B)\coloneqq (a_1, b_1) (a_2, b_2) \dots (a_n,b_n) \). Note that \(\sum_{j \leq i} (a_j -b_j+1) \) is equal to the number of occurrences of down steps before the \( (i+1) \)st up step. Since \( B \) is a (bicolored) Dyck path, it satisfies that  \(\sum_{j \leq i} (a_j -b_j+1) \leq i \), or equivalently,
\(\sum_{j \leq i} a_j \leq \sum_{j \leq i} b_j \) for each $i=1,2,\dots, n$. Hence, the map is well-defined and bijective.  

\begin{example}
From the construction of \( \phi_{\rm B} \), we have 
\begin{align*}
\phi_{\rm B}(B^{(1)})&=(0,1)(1,0)=Q^{(1)}, &
\phi_{\rm B}(B^{(2)})&=(0,1)(2,1)=Q^{(2)}, &
\phi_{\rm B}(B^{(3)})&=(1,1)(1,1)=Q^{(3)}, \\
\phi_{\rm B}(B^{(4)})&=(0,1)(1,1)=Q^{(4)}, &
\phi_{\rm B}(B^{(5)})&=(1,1)(0,1)=Q^{(5)}, &
\phi_{\rm B}(B^{(6)})&=(0,1)(0,1)=Q^{(6)},
\end{align*}
where 
\( B^{(i)} \in  \mathcal{B}_3 \) and $Q^{(i)}\in\Fp_2$ are given in Figures~\ref{fig:F2} and \ref{fig:B3}. 
We also see that 
\( \height(\phi_{\rm B}(B^{(i)}))=\last(B^{(i)})-1 \),
\( \north(\phi_{\rm B}(B^{(i)}))=\dasc(B^{(i)}) \), and
\( \aone(\phi_{\rm B}(B^{(i)}))=\bval(B^{(i)}) \) 
for each \( i \).
\end{example}

From the construction of the bijection $\phi_{\rm B}$, we have the following results.

\begin{thm}\label{thm:B-F}
Let \( n \) be a nonnegative integer. For each step \( (a,b)\in F \), we have
\[
| \mathcal{B}_{n+1}\{ u \dr^{1-b} \db^{a} u \dr^{m+1}\deli m\ge 0 \}|=|\mathcal{F}_n\{ (a,b) \}|.
\]
Consequently, \( |\mathcal{B}_{n+1}|=|\mathcal{F}_n| \).
\end{thm}

\begin{prop}\label{prop:B-F}
For a nonnegative integer \( n \), let \( B\in \mathcal{B}_{n+1} \). If \( \phi_{\rm B}(B)=Q \), then the following hold.
\begin{enumerate}
\item \( \height(Q)=\last(B)-1 \).
\item $\north(Q)=\dasc(B)$. 
\item $\aone(Q) =\bval(B)$. 
\end{enumerate}
\end{prop}

\begin{proof}
First, note that 
\[
\height(\phi_{\rm B}(B))=\sum_{i\leq n}(b_i-a_i) = n- \sum_{i \leq n} (a_i -b_i+1) =\last(B)-1
\]
since \( \sum_{i \leq n} (a_i -b_i+1) \) counts the number of occurrences of down steps before the last up step.
The remaining is clear because the inverse map \( \phi_{\rm B}^{-1} \) sends the steps \( (0,1) \) and \( (1,b) \) to the segments \( u \)
and  \( u \dr^{1-b} \db \), respectively. The former contributes to \( uu \) and the latter to \( u \db u \) or \( \dr \db u \).   
\end{proof}


\section{(2341,2431,3241)-avoiding permutations}
\label{sec:S}

Let $\mathfrak{S}_n$ be the set of permutations on $[n] \coloneqq \set{1,2,\dots, n}$.
For $\pi \in \mathfrak{S}_n$, we write $\abs{\pi} = n$.
A permutation \( \pi=\pi(1)\pi(2) \dots \pi(n) \in \mathfrak{S}_n \) is said to \emph{contain} a pattern \( \sigma =\sigma(1) \sigma(2) \dots \sigma(k) \in \mathfrak{S}_k\) if there exist some indices \( i_1 < i_2 < \cdots <i_k \) such that \( \pi({i_a})<\pi({i_b}) \) if and only if \( \sigma(a) < \sigma(b) \) for every \( a, b \in [k] \). Otherwise, \( \pi \) is said to \emph{avoid} the pattern \( \sigma \).

For any permutations \( \pi\in \mathfrak{S}_n \) and \( \sigma\in \mathfrak{S}_m \), the \emph{direct sum} of \( \pi \) and \( \sigma \) is the permutation 
$\pi \oplus \sigma$ on \( [n+m] \) defined by 
\[
(\pi \oplus \sigma)(i)\coloneqq \begin{cases}
\pi(i) & \mbox{ if \( 1\leq i \leq n\),}\\
\sigma(i-n)+n & \mbox{ if \( n+1 \leq i \leq n+m \).}
\end{cases}
\]
For a permutation \( \pi \in \mathfrak{S}_n\) with \( n \geq 1 \), if there is no \( j \in [n-1]\) such that \( \{ \pi(i) \deli 1 \leq i \leq j \}=[j] \), then \( \pi \) is said to be \emph{indecomposable}. 
We decompose a permutation as \( \pi=\pi_1\oplus \pi_2 \oplus \cdots \oplus \pi_\ell \), 
where each \( \pi_i \) is an indecomposable permutation, 
and define \( \block(\pi):=\ell \). 
For the sake of simplicity, we write \( \block(\emptyset)=0 \).

\subsection{Shapes of (2341,2431,3241)-avoiding permutations}
A permutation that avoids all the patterns $2341$, $2431$, and $3241$ is called a $(2341, 2431, 3241)$-avoiding permutation. Let  $\mathfrak{S}_n(2341,2431,3241)$ be the set of all $(2341, 2431, 3241)$-avoiding permutations on $[n]$. We refer to $\mathfrak{S}_n(2341,2431,3241)$ as $\perm_n$ for short.
Note that, for any permutations 
\( \pi\in \mathfrak{S}_n \) and \( \sigma\in \mathfrak{S}_m \), \(\pi \oplus \sigma \in \perm_{n+m}\) 
if and only if \( \pi \in \perm_{n} \) and \( \sigma \in \perm_{m} \).

The next lemma allows us to understand the configuration of \( (2341,2431,3241) \)-avoiding permutations.
Figure~\ref{fig:decom} illustrates four shapes of \( (2341,2431,3241) \)-avoiding permutations.
\begin{lem}\label{lem:indecomperm}
For \( \pi \in \perm_{n+1} \) with $n \geq 0$, 
let 
\begin{align*}
&x \coloneqq \pi^{-1}(n+1), 
&z \coloneqq \max\set{ \pi(i)\deli 0\le i\le x-1},\\
&y \coloneqq\pi^{-1}(z),
&w \coloneqq \min\set{ \pi(i)\deli x\le i\le n+1},
\end{align*}
with $\pi(0)=0$.
Then 
$$x-1\leq z \leq n,\quad x-1 \leq w \leq x,$$
and the following hold.

\begin{enumerate}[(1)]

\item 
If $z=n$ and $z<w$,
then $x=n+1$, 
$w=n+1$, and
\begin{align*} 
\set{ \pi(i)\deli 1\le i \le n,~i\neq y }&=\set{1,2, \dots, n-1}.
\end{align*}

\item 
If $z<n$ and $z<w$,
then $1\leq x\leq n$, $z = x-1$, $w=x$, and 
\begin{align*}
\set{ \pi(i) \deli 1 \le i \le x-1, ~i\neq y } &=\set{1,2, \dots, x-2}, \\
\set{ \pi(i) \deli x+1 \leq i \leq n+1 }        &=\set{x,x+1, \dots, n}.
\end{align*}

\item 
If $z=n$ and $z>w$,
then $2\leq x \leq n$, 
$w=x-1$, and
\begin{align*}
\set{ \pi(i) \deli 1 \le i \le x-1, ~i\neq y } &=\set{1,2, \dots, x-2}, \\
\set{ \pi(i) \deli x+1    \leq i \leq n+1 }   &=\set{x-1,x, \dots, n-1}.
\end{align*}

\item 
If $z<n$ and $z>w$,
then $2\leq x \leq n-1$, $x-1 < z < n$, $w=x-1$, and
\begin{align*}
\set{ \pi(i) \deli 1 \le i \le x-1, ~i\neq y } &=\set{1,2, \dots, x-2}, \\
\set{ \pi(i) \deli x+1    \leq i \leq z+1 }   &=\set{x-1,x, \dots, z-1}, \\
\set{ \pi(i) \deli z+2 \leq i \leq n+1 }        &=\set{z+1,z+2, \dots, n}.
\end{align*}

\end{enumerate}
\end{lem}

\begin{figure}[t]
\centering
\begin{tikzpicture}[scale=0.32]

\node at (-18+1/2, -18-1/2)        {\tiny $1$\strut};
\node at (-14+1/2, -18-1/2)        {\tiny $y$\strut};
\node at (-1+1/2, -18-1/2)         {\tiny $n+1$\strut};

\node at (-18, -1+1/2)  [left] {\tiny $n+1$\strut};
\node at (-18, -2+1/2)  [left] {\tiny $z$\strut};
\node at (-18, -18+1/2) [left] {\tiny $1$\strut};

\draw 
	(0,-18) -- (0,0) -- (-18,0) -- (-18, -18) -- cycle;
\draw [fill=blue!30]
	(0,-1) -- (0,0) -- (-1,-0) -- (-1,-1) -- cycle;
\draw [fill=red!30]
	(-13, -1) rectangle (-14, -2)
	;
\draw [fill=brown!30] 
	(-1, -2) rectangle (-13, -18) 
	(-14,-2) rectangle (-18,-18);

\draw [fill=black]
	(-1+1/2, -1+1/2) circle (5pt)
	(-14+1/2, -2+1/2) circle (5pt);


\end{tikzpicture}
\qquad 
\begin{tikzpicture}[scale=0.32]

\node at (-18+1/2, -18-1/2)        {\tiny $1$\strut};
\node at (-14+1/2, -18-1/2)        {\tiny $y$\strut};
\node at (-9+1/2, -18-1/2)         {\tiny $x$\strut};
\node at (-1+1/2, -18-1/2)         {\tiny $n+1$\strut};

\node at (-18, -1+1/2)  [left] {\tiny $n+1$\strut};
\node at (-18, -9+1/2) [left] {\tiny $w=x$\strut};
\node at (-18, -10+1/2) [left] {\tiny $z$\strut};
\node at (-18, -18+1/2) [left] {\tiny $1$\strut};
\node at (-20, -9+1/2) [left] {};

\draw 
	(0,-18) -- (0,0) -- (-18,0) -- (-18, -18) -- cycle;
\draw [fill=pink!30] 
	(0,-1) -- (-8,-1) -- (-8,-9) -- (0,-9) -- cycle;
\draw [fill=blue!30]
	(-8,0) -- (-9,0) -- (-9,-1) -- (-8,-1) -- cycle;
\draw [fill=red!30]
	(-13, -9) rectangle (-14, -10)
	;
\draw [fill=brown!30] 
	(-9, -10) rectangle (-13, -18) 
	(-14,-10) rectangle (-18,-18);

\draw [fill=black]
	(-9+1/2, -1+1/2) circle (5pt)
	(-7+1/2, -9+1/2) circle (5pt)
	(-14+1/2, -10+1/2) circle (5pt);


\end{tikzpicture}\\[2ex]

\begin{tikzpicture}[scale=0.32]

\node at (-18+1/2, -18-1/2)        {\tiny $1$\strut};
\node at (-14+1/2, -18-1/2)        {\tiny $y$\strut};
\node at (-9+1/2, -18-1/2)         {\tiny $x$\strut};
\node at (-1+1/2, -18-1/2)         {\tiny $n+1$\strut};

\node at (-18, -1+1/2)  [left] {\tiny $n+1$\strut};
\node at (-18, -2+1/2)  [left] {\tiny $z$\strut};
\node at (-18, -9+1/2) [left] {\tiny $x$\strut};
\node at (-18, -10+1/2) [left] {\tiny $w$\strut};
\node at (-18, -18+1/2) [left] {\tiny $1$\strut};

\draw 
	(0,-18) -- (0,0) -- (-18,0) -- (-18, -18) -- cycle;
\draw [fill=gray!30] 
	(0,-2) -- (-8,-2) -- (-8,-10) -- (0,-10) -- cycle;
\draw [fill=blue!30]
	(-8,0) -- (-9,0) -- (-9,-1) -- (-8,-1) -- cycle;
\draw [fill=red!30]
	(-13, -1) rectangle (-14, -2)
	;
\draw [fill=brown!30] 
	(-9, -10) rectangle (-13, -18) 
	(-14,-10) rectangle (-18,-18);

\draw [fill=black]
	(-9+1/2, -1+1/2) circle (5pt)
	(-7+1/2, -10+1/2) circle (5pt)
	(-14+1/2, -2+1/2) circle (5pt);

\end{tikzpicture}
\qquad 
\begin{tikzpicture}[scale=0.32]

\node at (-18+1/2, -18-1/2)        {\tiny $1$\strut};
\node at (-14+1/2, -18-1/2)        {\tiny $y$\strut};
\node at (-9+1/2, -18-1/2)         {\tiny $x$\strut};
\node at (-4+1/2, -18-1/2) [left]  {\tiny $z+1$\strut};
\node at (-1+1/2, -18-1/2)         {\tiny $n+1$\strut};

\node at (-18, -1+1/2)  [left] {\tiny $n+1$\strut};
\node at (-18, -6+1/2)  [left] {\tiny $z$\strut};
\node at (-18, -9+1/2) [left] {\tiny $x$\strut};
\node at (-18, -10+1/2) [left] {\tiny $w$\strut};
\node at (-18, -18+1/2) [left] {\tiny $1$\strut};
\node at (-20, -9+1/2) [left] {};

\draw 
	(0,-18) -- (0,0) -- (-18,0) -- (-18, -18) -- cycle;
\draw [fill=pink!30] 
	(0,-1) -- (-4,-1) -- (-4,-5) -- (0,-5) -- cycle;
\draw [fill=gray!30] 
	(-4,-6) -- (-8,-6) -- (-8,-10) -- (-4,-10) -- cycle;
\draw [fill=blue!30]
	(-8,0) -- (-9,0) -- (-9,-1) -- (-8,-1) -- cycle;
\draw [fill=red!30]
	(-13, -5) rectangle (-14, -6)
	;
\draw [fill=brown!30] 
	(-9, -10) rectangle (-13, -18) 
	(-14,-10) rectangle (-18,-18);

\draw [fill=black]
	(-9+1/2, -1+1/2) circle (5pt)
	(-7+1/2, -10+1/2) circle (5pt)
	(-14+1/2, -6+1/2) circle (5pt);


\end{tikzpicture}

\caption{Four shapes of $\pi \in \perm_{n+1}$ in Lemma~\ref{lem:indecomperm}.}\label{fig:decom}
\end{figure}

\begin{proof}
If $z<x-1$, then it should be $\set{\pi(i) \deli x+1 \leq i \leq n+1} = \set{x-1,x,\dots, n}$,
which is a contradiction because the sets have different numbers of elements.
Thus, we have $x-1\leq z\leq n$.
Similarly, if $w>x$, then $\set{\pi(i) \deli 1 \leq i \leq x-1} = \set{1,2,\dots, x}$,
which is a contradiction. 
Thus, $w \leq x$.

Now we show that \( x-1\leq w \).
Assuming there exists \( a \) with \( a>w \) satisfying \( 1\leq \pi^{-1}(a)\leq y-1\)
yields a contradiction because \( \pi \) contains \( az (n+1)w \), which forms the pattern \( 2341 \). 
Thus, $\pi(i)<w$ for all $1 \leq i \leq y-1$.
Similarly, assuming there exists \( b \) with \( b>w \) satisfying \( y+1\leq \pi^{-1}(b)\leq x-1\)
yields a contradiction because \( \pi \) contains \( z b(n+1)w \), which forms the pattern \( 3241 \).
Thus, $\pi(i)<w$ for all $y+1 \leq i \leq x-1$.
So we have $x-1\leq w$ and
\begin{align}
\set{ \pi(i) \deli 1 \le i \le x-1, ~i\neq y } &=\set{1,2, \dots, x-2}, \label{eq:sigma}
\end{align}
which is the first part of each case since if \( z=n \) and \( z<w \), 
then \( x=n+1 \). 

Now we consider the remaining parts. 
If there exists $d$ with $d<z$ satisfying $z+2\leq \pi^{-1}(d)\leq n+1$,
then there exists $c$ with $c>z$ satisfying $x+1\leq \pi^{-1}(c)\leq z+1$. 
This yields a contradiction because \( \pi \) contains \( z(n+1)cd \), which forms the pattern \( 2431 \). 
Hence, it follows that 
\begin{align}
\set{ \pi(i) \deli x+1    \leq i \leq z+1 }   &=\set{w,w+1, \dots, z-1}, \label{eq:omega} \\
\set{ \pi(i) \deli z+2 \leq i \leq n+1 }        &=\set{z+1,z+2, \dots, n}. \label{eq:tau}
\end{align}

\begin{enumerate}[(1)]
\item 
If $z=n$ and $z<w$, 
then it is clear that $w=n+1$ and $x=n+1$.
Note that both sides of \eqref{eq:omega} and \eqref{eq:tau} become empty sets.
\item 
If $z<n$ and $z<w$,
then it is clear that $z = x-1$, $w=x$, and $1\le x\le n$. 
Both sides of \eqref{eq:omega} become empty sets, and \eqref{eq:tau} becomes
\begin{align*}
\set{ \pi(i) \deli x+1 \leq i \leq n+1 } &=\set{x,x+1, \dots, n}.
\end{align*}

\item 
If $z=n$ and $z>w$,
then it is clear that $w=x-1$ and $2 \leq x \leq n$.
In this case,
\eqref{eq:omega} becomes
\begin{align*}
\set{ \pi(i) \deli x+1    \leq i \leq n+1 }   &=\set{x-1,x, \dots, n-1},
\end{align*}
and both sides of \eqref{eq:tau} become empty sets.

\item 
If $z<n$ and $z>w$,
then it is clear that $w=x-1$ and $1 \leq x-1 < z < n$.
So
\eqref{eq:omega} and \eqref{eq:tau}
are what we want.
\end{enumerate}
This completes the proof. In addition,
for each of cases (1) to (4), both sides of \eqref{eq:sigma} become empty sets if $x=1$ or $x=2$.
\end{proof}

For $n \geq 1$ and $ \pi=\pi_1\oplus \pi_2 \oplus \cdots \oplus \pi_\ell \in \perm_{n+1}$ with $\block(\pi)=\ell$, 
letting $x$, $y$, $z$, and $w$ defined in Lemma~\ref{lem:indecomperm},
we are able to define a permutation $\hat{\pi} \in \perm_{n}$ from $\pi$ as follows.


\begin{enumerate}[(1)]
\item 
If $z=n$ and $z<w$, 
then 
define $\hat{\pi}\in\perm_n$ by
\begin{align*}
\hat{\pi}(i) &:= \pi(i)
\end{align*}
for  $1\leq i \leq n$.
Since $\pi = \hat{\pi} \oplus 1$, we have
\begin{align} \label{eq:blk_hatpi1}
\block(\pi) = \block(\hat{\pi})+1.
\end{align}

\item 
If $z<n$ and $z<w$,
then 
define $\hat{\pi}\in\perm_n$ by
\[
\hat{\pi}(i)\coloneqq 
\begin{cases}
\pi(i)      & \text{ if $1\leq i \leq x-1$,}\\
\pi(i+1)    & \text{ if $x\leq i \leq n$,}
\end{cases}
\]
as shown in Figure~\ref{fig:pihat2}.
Here,  $\hat{\pi}$ can be denoted as 
\( \hat{\pi}=\sigma \oplus \tau \), 
where 
\( \sigma \in \perm_{x-1} \) and \( \omega \in \perm_{n-x+1} \). 
Since $\block(\pi) = \block(\sigma) + 1 $ 
 and $\block(\hat{\pi})=\block(\sigma)+\block(\tau)$, 
 we have
\begin{align} \label{eq:blk_hatpi2}
\block(\pi) = \block(\hat{\pi}) - \block(\tau) + 1.
\end{align}

\begin{figure}[t]
\begin{center}
\begin{tikzpicture}[scale=0.279]

\node at (-14+1/2, 1/2)        {\tiny $y$\strut};
\node at (-9+1/2, 1/2)         {\tiny $x$\strut};
\node at (-1+1/2, 1/2)         {\tiny $n+1$\strut};

\node at (0, -1+1/2) [right] {\tiny $n+1$\strut};
\node at (0, -2+1/2) [right] {\tiny $n$\strut};
\node at (0, -9+1/2) [right] {\tiny $w(=x)$\strut};
\node at (0, -10+1/2) [right] {\tiny $z(=x-1)$\strut};

\draw 
	(0,-20) -- (0,0) -- (-20,0);
\draw [fill=pink!30] 
	(0,-1) rectangle (-8,-9);
\draw [fill=blue!30]
	(-8,0) rectangle (-9,-1);
\draw [fill=red!30]
	(-13, -9) rectangle (-14, -10);
\path [fill=brown!30] 
	(-9, -10) rectangle (-13, -20) 
	(-14,-10) rectangle (-20,-20);
\draw 
	(-9, -20) -- (-9, -10) -- (-13, -10) -- (-13, -20)
	(-14, -20) -- (-14, -10) -- (-20, -10)
	(-9,-9) -- (-9, -20)
	;
\draw [ultra thick]
	(-20,-16) -- (-16,-16) -- (-16,-20)
	(0,0) rectangle (-9,-9) -- (0,-9)
	(-9,-9) rectangle (-16,-16);
\draw [fill=black]
	(-9+1/2, -1+1/2) circle (5pt)
	(-7+1/2, -9+1/2) circle (5pt)
	(-14+1/2, -10+1/2) circle (5pt);

\node at (0, -20) [above left] {\small $\pi \in \perm_{n+1}$};
\node at (0-1/3, -9+1/3) [above left] {\small $\pi_\ell$};


\end{tikzpicture}
\qquad
\begin{tikzpicture}[scale=0.279]

\node at (-13+1/2, 1/2)       {\tiny $y$\strut};
\node at (-8+1/2, 1/2)        {\tiny $x$\strut};
\node at (-1+1/2, 1/2)        {\tiny $n$\strut};

\node at (0, -1+1/2) [right] {\tiny $n$\strut};
\node at (0, -8+1/2) [right] {\tiny $x$\strut};
\node at (0, -9+1/2) [right] {\tiny $x-1$\strut};

\draw 
	(0,-19) -- (0,0) -- (-19,0);
\draw [fill=pink!30] 
	(0,0) rectangle (-8,-8);
\draw [ultra thick] 
	(0,0) -- (-2,0) -- (-2,-2) -- (0,-2) -- cycle
	(-2,-2) -- (-3,-2) -- (-3,-3) -- (-2,-3) -- cycle
	(-5,-5) -- (-8,-5) -- (-8,-8) -- (-5,-8) -- cycle;
\foreach \i in {1,...,5}{
\draw [fill=black]
	(-3-\i/3, -3-\i/3) circle (2pt);
}
\draw [fill=red!30]
	(-12, -8) rectangle (-13, -9);
\path [fill=brown!30] 
	(-8, -9) rectangle (-12, -19) 
	(-13,-9) rectangle (-19,-19);
\draw 
	(-8,-8) -- (-8, -19)
	(-8,-8) -- (-19, -8)
	(-8,-9) -- (-19, -9)
	(-12,-8) -- (-12, -19)
	(-13,-8) -- (-13, -19);
	
\draw [ultra thick]
	(-19,-15) -- (-15,-15) -- (-15,-19)
	(-8,-8) rectangle (-15,-15);

\draw [fill=black]
	(-7+1/2, -8+1/2) circle (5pt)
	(-13+1/2, -9+1/2) circle (5pt);

\node at (0, -19) [above left] {\small $\hat{\pi} \in \perm_{n}$};
\node at (0-1/3, -8+1/3) [above left] {\small $\tau$};
\node at (-8-1/3, -19+1/3) [above left] {\small $\sigma$};

\end{tikzpicture}
\end{center}
\caption{A shape of \( \hat{\pi} \) for \( \pi \) with \( z<n \) and \( z<w \). Here, a thick line border box represents a block.}\label{fig:pihat2}
\end{figure}

\item 
If $z=n$ and $z>w$, then
define $\hat{\pi}\in\perm_n$ by
\[
\hat{\pi}(i) \coloneqq 
\begin{cases}
\pi(i)       & \text{ if $ 1\leq i \leq x-1$ and $i\neq y$,}\\
x-1          & \text{ if $ i=y $,}\\
\pi(i+1)+1   & \text{ if $ x\leq i \leq n $},
\end{cases}
\]
as shown in Figure~\ref{fig:pihat3}.
Here,  $\hat{\pi}$ can be denoted as 
\( \hat{\pi}=\sigma \oplus \omega \), 
where 
\( \sigma \in \perm_{x-1} \) and \( \omega \in \perm_{n-x+1} \).
Since $\block(\pi) = \block(\sigma)$
 and $\block(\hat{\pi})=\block(\sigma)+\block(\omega)$,
 we have
\begin{align} \label{eq:blk_hatpi3}
\block(\pi) = \block(\hat{\pi}) - \block(\omega).
\end{align}

\begin{figure}[t]
\begin{center}
\begin{tikzpicture}[scale=0.279]

\node at (-14+1/2, 1/2)        {\tiny $y$\strut};
\node at (-9+1/2, 1/2)         {\tiny $x$\strut};
\node at (-1+1/2, 1/2)         {\tiny $n+1$\strut};

\node at (0, -1+1/2)  [right] {\tiny $n+1$\strut};
\node at (0, -2+1/2)  [right] {\tiny $z(=n)$\strut};
\node at (0, -9+1/2) [right] {\tiny $x$\strut};
\node at (0, -10+1/2) [right] {\tiny $w(=x-1)$\strut};

\draw 
	(0,-20) -- (0,0) -- (-20,0);
\draw [fill=gray!30] 
	(0,-2) -- (-8,-2) -- (-8,-10) -- (0,-10) -- cycle;
\draw [fill=blue!30]
	(-8,0) -- (-9,0) -- (-9,-1) -- (-8,-1) -- cycle;
\draw [fill=red!30]
	(-13, -1) rectangle (-14, -2)
	;
\path [fill=brown!30] 
	(-9, -10) rectangle (-13, -20) 
	(-14,-10) rectangle (-20,-20);
\draw 
	(-9, -20) -- (-9, -10) -- (-13, -10) -- (-13, -20)
	(-14, -20) -- (-14, -10) -- (-20, -10)
	;	
\draw [ultra thick]
	(-20,-16) -- (-16,-16) -- (-16,-20)
	(0,0) -- (-16,0) -- (-16,-16) -- (0,-16) -- cycle;

\draw [fill=black]
	(-9+1/2, -1+1/2) circle (5pt)
	(-7+1/2, -10+1/2) circle (5pt)
	(-14+1/2, -2+1/2) circle (5pt);

\node at (0, -20) [above left] {\small $\pi \in \perm_{n+1}$};
\node at (0-1/3, -16+1/3) [above left] {\small $\pi_\ell$};


\end{tikzpicture}
\qquad
\begin{tikzpicture}[scale=0.279]

\node at (-13+1/2, 1/2)        {\tiny $y$\strut};
\node at (-8+1/2, 1/2)         {\tiny $x$\strut};
\node at (-1+1/2, 1/2)         {\tiny $n$\strut};

\node at (0, -1+1/2) [right] {\tiny $n$\strut};
\node at (0, -8+1/2) [right] {\tiny $x$\strut};
\node at (0, -9+1/2) [right] {\tiny $x-1$\strut};

\draw 
	(0,-19) -- (0,0) -- (-19,0);
\draw [fill=gray!30] 
	(0,0) -- (-8,0) -- (-8,-8) -- (0,-8) -- cycle;
\draw [ultra thick] 
	(0,0) -- (-2,0) -- (-2,-2) -- (0,-2) -- cycle
	(-2,-2) -- (-3,-2) -- (-3,-3) -- (-2,-3) -- cycle
	(-5,-5) -- (-8,-5) -- (-8,-8) -- (-5,-8) -- cycle;
\foreach \i in {1,...,5}{
\draw [fill=black]
	(-3-\i/3, -3-\i/3) circle (2pt);
}
\draw [fill=red!30]
	(-12, -8) rectangle (-13, -9);
\path [fill=brown!30] 
	(-8, -9) rectangle (-12, -19) 
	(-13,-9) rectangle (-19,-19);
\draw 
	(-8,-8) -- (-8, -19)
	(-8,-8) -- (-19, -8)
	(-8,-9) -- (-19, -9)
	(-12,-8) -- (-12, -19)
	(-13,-8) -- (-13, -19);
	
\draw [ultra thick]
	(-19,-15) -- (-15,-15) -- (-15,-19)
	(-8,-8) -- (-15,-8) -- (-15,-15) -- (-8,-15) -- cycle;

\draw [fill=black]
	(-7+1/2, -8+1/2) circle (5pt)
	(-13+1/2, -9+1/2) circle (5pt);

\node at (0, -19) [above left] {\small $\hat{\pi} \in \perm_{n}$};
\node at (0-1/3, -8+1/3) [above left] {\small $\omega$};
\node at (-8-1/3, -19+1/3) [above left] {\small $\sigma$};

\end{tikzpicture}
\end{center}
\caption{A shape of \( \hat{\pi} \) for \( \pi \) with \( z=n \) and \( z>w \). Here, a thick line border box represents a block.}\label{fig:pihat3}
\end{figure}

\item 
If $z<n$ and $z>w$, 
then 
define $\hat{\pi}\in\perm_n$ by
\[
\hat{\pi}(i)\coloneqq \begin{cases}
\pi(i)       & \text{ if $ 1\leq i \leq x-1$ and $i\neq y$,}\\
x-1          & \text{ if \( i=y \),}\\
\pi(i+1)+1   & \text{ if \( x\leq i \leq z \),}\\
\pi(i+1)     & \text{ if $ z+1 \leq i \leq n $},
\end{cases}
\]
as shown in Figure~\ref{fig:pihat4}.
Here,  $\hat{\pi}$ can be denoted as 
\( \hat{\pi}=\sigma \oplus \omega \oplus \tau  \), 
where 
\( \sigma \in \perm_{x-1} \), \( \omega \in \perm_{z-x+1} \), and \( \tau \in \perm_{n-z} \).
Since $\block(\pi) = \block(\sigma)$  
 and $\block(\hat{\pi})=\block(\sigma)+\block(\omega)+\block(\tau)$, 
we have
\begin{align} \label{eq:blk_hatpi4}
\block(\pi) = \block(\hat{\pi}) - \block(\omega) - \block(\tau).
\end{align}

\begin{figure}[t]
\begin{center}
\begin{tikzpicture}[scale=0.279]

\node at (-14+1/2, 1/2)        {\tiny $y$\strut};
\node at (-9+1/2, 1/2)         {\tiny $x$\strut};
\node at (-5+1/2, 1/2)         {\tiny $z+1$\strut};
\node at (-1+1/2, 1/2)         {\tiny $n+1$\strut};

\node at (0, -1+1/2)  [right] {\tiny $n+1$\strut};
\node at (0, -2+1/2)  [right] {\tiny $n$\strut};
\node at (0, -6+1/2)  [right] {\tiny $z$\strut};
\node at (0, -9+1/2) [right] {\tiny $x$\strut};
\node at (0, -10+1/2) [right] {\tiny $w(=x-1)$\strut};

\draw 
	(0,-20) -- (0,0) -- (-20,0);
\draw [fill=pink!30] 
	(0,-1) -- (-4,-1) -- (-4,-5) -- (0,-5) -- cycle;
\draw [fill=gray!30] 
	(-4,-6) -- (-8,-6) -- (-8,-10) -- (-4,-10) -- cycle;
\draw [fill=blue!30]
	(-8,0) -- (-9,0) -- (-9,-1) -- (-8,-1) -- cycle;
\draw [fill=red!30]
	(-13, -5) rectangle (-14, -6)
	;
\path [fill=brown!30] 
	(-9, -10) rectangle (-13, -20) 
	(-14,-10) rectangle (-20,-20);
\draw 
	(-9, -20) -- (-9, -10) -- (-13, -10) -- (-13, -20)
	(-14, -20) -- (-14, -10) -- (-20, -10)
	;	

\draw [ultra thick]
	(-20,-16) -- (-16,-16) -- (-16,-20)
	(0,0) -- (-16,0) -- (-16,-16) -- (0,-16) -- cycle;

\draw [fill=black]
	(-9+1/2, -1+1/2) circle (5pt)
	(-7+1/2, -10+1/2) circle (5pt)
	(-14+1/2, -6+1/2) circle (5pt);

\node at (0, -20) [above left] {\small $\pi \in \perm_{n+1}$};
\node at (0-1/3, -16+1/3) [above left] {\small $\pi_\ell$};


\end{tikzpicture}
\qquad
\begin{tikzpicture}[scale=0.279]

\node at (-13+1/2, 1/2)        {\tiny $y$\strut};
\node at (-7-1/2, 1/2)         {\tiny $x$\strut};
\node at (-5+1/2, 1/2)         {\tiny $z$\strut};
\node at (-1+1/2, 1/2)         {\tiny $n$\strut};

\node at (0, -1+1/2) [right] {\tiny $n$\strut};
\node at (0, -5+1/2) [right] {\tiny $z$\strut};
\node at (0, -8+1/2) [right] {\tiny $x$\strut};
\node at (0, -9+1/2) [right] {\tiny $x-1$\strut};

\draw 
	(0,-19) -- (0,0) -- (-19,0);
\draw [fill=pink!30] 
	(0,0) -- (-4,0) -- (-4,-4) -- (0,-4) -- cycle;
\draw [fill=gray!30] 
	(-4,-4) -- (-8,-4) -- (-8,-8) -- (-4,-8) -- cycle;
\draw [fill=red!30]
	(-12, -8) rectangle (-13, -9);
\path [fill=brown!30] 
	(-8, -9) rectangle (-12, -19) 
	(-13,-9) rectangle (-19,-19);
\draw 
	(-8,-8) -- (-8, -19)
	(-8,-8) -- (-19, -8)
	(-8,-9) -- (-19, -9)
	(-12,-8) -- (-12, -19)
	(-13,-8) -- (-13, -19);
	
\draw [ultra thick] 
	(0,0) rectangle (-1-1/2,-1-1/2)
	(-1-1/2,-1-1/2) rectangle (-2,-2)
	(-3+1/4,-3+1/4) rectangle (-4,-4)
	(-4,-4) rectangle (-4-1,-4-1)
	(-6+1/4,-6+1/4) rectangle (-8,-8)
;
\foreach \i in {2,...,4}{
\draw [fill=black]
	(-1-1/2-\i/10*3, -1-1/2-\i/10*3) circle (2pt);
	}
\foreach \i in {2,...,4}{
\draw [fill=black]
	(-4-1/2-\i/10*3, -4-1/2-\i/10*3) circle (2pt);
}
\draw [ultra thick]
	(-19,-15) -- (-15,-15) -- (-15,-19)
	(-8,-8) -- (-15,-8) -- (-15,-15) -- (-8,-15) -- cycle;

\draw [fill=black]
	(-7+1/2, -8+1/2) circle (5pt)
	(-13+1/2, -9+1/2) circle (5pt);

\node at (0, -19) [above left] {\small $\hat{\pi} \in \perm_{n}$};
\node at (0, -4) [above left] {\small $\tau$};
\node at (-4+1/4, -8) [above left] {\small $\omega$};
\node at (-8-1/3, -19+1/3) [above left] {\small $\sigma$};

\end{tikzpicture}
\end{center}
\caption{A shape of \( \hat{\pi} \) for \( \pi \) with \( z<n \) and \( z>w \). Here, a thick line border box represents a block.}\label{fig:pihat4}
\end{figure}
\end{enumerate}

We give an example for $\hat{\pi}$.
\begin{example}
Consider the permutations \( \pi \in \perm_{9} \).
\begin{enumerate}
\item 
If \( \pi= 321547689 \), then $(x,y,z,w) = (9,8,8,9)$ and $\hat{\pi}= 32154768$.

\item 
If \( \pi= 321549768 \), then $(x,y,z,w) = (6,4,5,6)$ 
and $\hat{\pi}= 32154768 = \sigma \oplus \tau$,
where \( \sigma= 32154 \in  \perm_{5} \) and \( \tau= 213 \in \perm_{3} \).

\item 
If \( \pi= 321849657 \), then $(x,y,z,w) = (6,4,8,5)$
and $\hat{\pi}= 32154768 =\sigma \oplus \omega$, 
where \( \sigma= 32154 \in  \perm_{5} \) and \( \omega= 213 \in \perm_{3} \).  

\item 
If \( \pi= 321749658 \), then $(x,y,z,w) = (6,4,7,5)$
and \( \hat{\pi}= 32154768 = \sigma \oplus \omega \oplus \tau \), 
where \( \sigma=32154 \in  \perm_{5} \), \( \omega= 21 \in \perm_{2} \), and \( \tau= 1 \in \perm_{1} \). 
\end{enumerate}
\end{example}


\subsection{Bijection between (2341,2431,3241)-avoiding permutations and {\it F}-paths}
Let us define a map 
\[
\phi_{\rm S} : \bigcup_{n\geq 0}\perm_{n+1} \to \bigcup_{n\geq 0}\mathcal{F}_n
\]
such that
$$\block(\pi)=\height(\phi_{\rm S}(\pi)) + 1$$
recursively. 
Let \( \phi_{\rm S}(1)\coloneqq \emptyset \)
with $\block(1)=\height(\emptyset) + 1$. 
Now suppose \( n \geq 1 \). 

\begin{enumerate}[(1)]
\item 
For $\pi \in \perm_{n+1}$ with $z=n$ and $z<w$,
let $$\phi_{\rm S}(\pi)\coloneqq \phi_{\rm S}(\hat{\pi})(0,1).$$ 
Due to $\block(\hat{\pi})=\height(\phi_{\rm S}(\hat{\pi})) + 1$ and~\eqref{eq:blk_hatpi1}, we have
\begin{align*}
\block(\pi) 
&= \block(\hat{\pi})+1 \\ 
&= \height(\phi_{\rm S}(\hat{\pi})) + 2 = \height(\phi_{\rm S}(\pi)) + 1.
\end{align*}

\item
For $\pi \in \perm_{n+1}$ with $z<n$ and $z<w$,
let $$\phi_{\rm S}(\pi)\coloneqq \phi_{\rm S}(\hat{\pi})(1,2-\block(\tau)).$$ 
Due to $\block(\hat{\pi})=\height(\phi_{\rm S}(\hat{\pi})) + 1$ and~\eqref{eq:blk_hatpi2}, we have
\begin{align*}
\block(\pi) 
&= \block(\hat{\pi}) - \block(\tau) + 1\\
&= \height(\phi_{\rm S}(\hat{\pi})) - \block(\tau) + 2= \height(\phi_{\rm S}(\pi)) + 1.
\end{align*}

\item
For $\pi \in \perm_{n+1}$ with $z=n$ and $z>w$,
let $$\phi_{\rm S}(\pi)\coloneqq \phi_{\rm S}(\hat{\pi})(1+\block(\omega),1).$$ 
Due to $\block(\hat{\pi})=\height(\phi_{\rm S}(\hat{\pi})) + 1$ and~\eqref{eq:blk_hatpi3}, we have
\begin{align*}
\block(\pi) 
&= \block(\hat{\pi}) - \block(\omega)\\
&= \height(\phi_{\rm S}(\hat{\pi})) - \block(\omega) + 1 = \height(\phi_{\rm S}(\pi)) + 1.
\end{align*}

\item
For $\pi \in \perm_{n+1}$ with $z<n$ and $z>w$,
let $$\phi_{\rm S}(\pi)\coloneqq \phi_{\rm S}(\hat{\pi})(1+\block(\omega), 1-\block(\tau)).$$
Due to $\block(\hat{\pi})=\height(\phi_{\rm S}(\hat{\pi})) + 1$ and~\eqref{eq:blk_hatpi4}, we have
\begin{align*}
\block(\pi) &= \block(\hat{\pi})- \block(\omega) - \block(\tau) \\ 
&= \height(\phi_{\rm S}(\hat{\pi})) - \block(\omega) - \block(\tau) + 1 = \height(\phi_{\rm S}(\pi)) + 1.
\end{align*}
\end{enumerate}
Let $\phi_{\rm S}(\pi) = (a_1, b_1) \dots (a_n, b_n)$.
Because $b_n-a_n=\block(\pi)- \block(\hat{\pi})$ in any case,
we obtain $$\height(\phi_{\rm S}(\pi))=\block(\pi)-1 \geq 0$$
for all $\pi \in \perm_{n+1}$, 
so it is easy to check that  
\( \sum_{j \leq n} (b_j-a_j) \geq 0 \) for showing \( \phi_{\rm S} \) is well-defined. 
By proving the bijectivity of \( \phi_{\rm S} \), we obtain the following theorem.
\begin{thm}\label{thm:S-F}
For a nonnegative integer \( n \), 
the following hold.
\begin{enumerate} 
\item \( | \{ \pi \in \perm_{n+1} \deli \text{$z=n$, $z<w$} \}|
=|\mathcal{F}_n\{ (0,1) \}|\).
\item \( | \{ \pi \in \perm_{n+1} \deli \text{$z<n$, $z<w$} \}|
=|\mathcal{F}_n\{ (1,b) \deli b \leq 1 \}|\).
\item \( | \{ \pi \in \perm_{n+1} \deli \text{$z=n$, $z>w$} \}|
=|\mathcal{F}_n\{ (a,1) \deli a\geq 2 \}|\).
\item \( | \{ \pi \in \perm_{n+1} \deli \text{$z<n$, $z>w$} \}|
=|\mathcal{F}_n\{ (a,b) \deli a\geq 2, ~b\leq 0 \}|\).
\end{enumerate}
where $z$ and $w$ are defined in Lemma~\ref{lem:indecomperm}.
Consequently, \( |\perm_{n+1}|=|\mathcal{F}_n| \). 
\end{thm}

\begin{proof}
It suffices that 
we can construct the inverse map $\psi_{\rm S}$ of $\phi_{\rm S}$ recursively as follows.
Let $\psi_{\rm S}(\emptyset) \coloneqq 1$. 
Now suppose $n \geq 1$. 
For $Q = \hat{Q} (a_n, b_n) \in \mathcal{F}_n$, 
set $\hat{\pi} = \psi_{\rm S}(\hat{Q}) \in \perm_{n}$
with
$$\block(\hat{\pi}) = \height(\hat{Q}) + 1.$$
Due to $\height(\hat{Q}) + (b_n-a_n) = \height(Q) \geq 0$,
we know
$$\block(\hat{\pi}) = \height(\hat{Q}) + 1 \geq (a_n-b_n) +1.$$ 
The last step \( (a_n, b_n) \) of \( Q \) is one of the following four.
\begin{enumerate}[(1)]
\item 
For \( (a_n, b_n)=(0,1) \),
define $ \psi_{\rm S}(Q) \coloneqq \pi \in \perm_{n+1}$ 
by 
$\pi = \hat{\pi} \oplus 1$. 
Thus, we have $z=n$ and $w=n+1$ for $\pi$.

\item 
For \( a_n=1 \) and \( b_n = 1-j \leq 1 \) with $j \geq 0$,
due to 
$$\block(\hat{\pi}) \geq (a_n - b_n) +1 = j+1,$$
$\hat{\pi}$ can be denoted as 
\( \hat{\pi}=\sigma \oplus \tau \in  \perm_{n} \) 
satisfying 
$\block(\tau) = j+1$ 
with $\abs{\sigma}=x-1 \geq 0$ and $\abs{\tau}=n-x+1\geq 1$
for some $1 \leq x \leq n$. 
Here, $\sigma$ is allowed to be empty.
Define $ \psi_{\rm S}(Q) \coloneqq \pi \in \perm_{n+1}$ by
\[
\pi(i)=
\begin{cases}
\hat{\pi}(i) & \text{ if $1\leq i\leq x-1$,}\\
n+1 & \text{ if \(  i = x \),}\\
\hat{\pi}(i-1) & \text{ if \( x+1\leq i \leq n+1 \).}
\end{cases}
\]
Thus, we have $z=x-1 < n$ and $w=x$ for $\pi$.

\item 
For \( a_n = k+2 \geq 2\) and \( b_n=1 \) with $k \geq 0$, 
due to $$\block(\hat{\pi}) \geq (a_n - b_n) +1 = k+2,$$
$\hat{\pi}$ can be denoted as 
\( \hat{\pi}=\sigma \oplus \omega  \in  \perm_{n}\) 
satisfying 
$\block(\omega) = k+1$ with $\abs{\sigma}=x-1 \geq 1$ and $\abs{\omega}=n-x+1 \geq 1$
for some $2 \leq x \leq n$. 
Letting $y = \hat{\pi}^{-1}(x-1)$, 
define $ \psi_{\rm S}(Q) \coloneqq \pi \in \perm_{n+1}$ by 
\[
\pi(i) = 
\begin{cases}
\hat{\pi}(i) & \text{ if $1\leq i \leq x-1$ and $i\neq y$,}\\
n & \text{ if \(  i = y \),}\\
n+1 & \text{ if \(  i = x \),}\\
\hat{\pi}(i-1)-1 & \text{ if \( x+1\leq i \leq n+1 \).}\\
\end{cases}
\]
Thus, we have $z=n$ and $w=x-1$ for $\pi$.

\item 
For \( a_n = k+2 \geq 2 \) and \( b_n= -j \leq 0\)  with $k \geq 0$ and $j \geq 0$,
due to 
$$\block(\hat{\pi})  \geq (a_n - b_n) +1 = k+j+3,$$
$\hat{\pi}$ can be denoted as 
\( \hat{\pi}=\sigma \oplus \omega \oplus \tau \in \perm_{n}\) 
satisfying
$\block(\omega) = k+1$ and $\block(\tau) = j+1$ 
with $\abs{\sigma} = x-1 \geq 1$, $\abs{\omega} = z-x+1 \geq 1$, and $\abs{\tau} = n-z \geq 1$ 
for some $2 \leq x \leq z < n$. 
Letting $y = \hat{\pi}^{-1}(x-1)$, 
define $ \psi_{\rm S}(Q) \coloneqq \pi \in \perm_{n+1}$ by
\[
\pi(i) = 
\begin{cases}
\hat{\pi}(i) & \text{ if $1\leq i \leq x-1$ and $i\neq y$,}\\
z & \text{ if \(  i = y \),}\\
n+1 & \text{ if \(  i = x \),}\\
\hat{\pi}(i-1)-1 & \text{ if \( x+1 \leq i \leq z+1 \),}\\
\hat{\pi}(i-1) & \text{ if \( z+2 \leq i \leq n+1 \).}
\end{cases}
\]
Thus, we have $z<n$ and $w=x-1 < z$ for $\pi$.

\end{enumerate}
Hence, we conclude that 
\( \phi_{\rm S} \) and \( \psi_{\rm S} \) are bijective,
which completes the proof.

\end{proof}

\begin{example}
There are six \( (2341,2431,3241) \)-avoiding permutations of length \( 3 \). From the construction of \( \phi_{\rm S} \), we have 
\begin{align*}
\phi_{\rm S}(231)&=(0,1)(2,1)=Q^{(1)}, &
\phi_{\rm S}(312)&=(0,1)(1,0)=Q^{(2)}, &
\phi_{\rm S}(321)&=(1,1)(1,1)=Q^{(3)}, \\
\phi_{\rm S}(132)&=(0,1)(1,1)=Q^{(4)}, &
\phi_{\rm S}(213)&=(1,1)(0,1)=Q^{(5)}, &
\phi_{\rm S}(123)&=(0,1)(0,1)=Q^{(6)},
\end{align*}
where each \( Q^{(i)} \in  \mathcal{F}_2 \) is given in Figure~\ref{fig:F2}. We also check that 
\( \height(\phi_{\rm S}(\pi))=\block(\pi)-1 \) for $\pi\in\perm_3$.
\end{example}

\begin{rem}\label{rem:S-F}
\begin{enumerate}
\item 
The \emph{ascent number} of $\pi\in \mathfrak{S}_n$ is defined by 
\( \asc(\pi)\coloneqq |\{ i\in [n-1] \deli \pi(i)<\pi(i+1) \}| \).
It is easy to check that
$\asc(\pi)=\asc(\hat{\pi})+1$ only in the case $(a_n, b_n) = (0, 1)$ 
of the proof of Theorem~\ref{thm:S-F}
and $\asc(\pi)=\asc(\hat{\pi})$ in the rest of the cases.
So for all $\pi \in \bigcup_{n\geq 0}\perm_{n+1}$,
we obtain $$\asc(\pi)=\north(\phi_{\rm S}(\pi)).$$

\item 
For $\pi \in \mathfrak{S}_n$ and $i \in [n]$, 
we define $\pi(i)$ to be \emph{critical}
if $\pi(j)<\pi(i)$ and $\pi(k)<\pi(i)$ implies $\pi(j)<\pi(k)$ for all $j<i<k$.
For example, given $\pi=2413$, $\pi(1)$, $\pi(3)$, and $\pi(4)$ is critical. However, $\pi(2)=4$ is not critical because $\pi(1)<\pi(2)$ and $\pi(3)<\pi(2)$ does not imply $\pi(1)<\pi(3)$.  
Let $\crit(\pi)$ denote the number of critical elements in $\pi$.
It is easy to check that
$\crit(\pi)=\crit(\hat{\pi})+1$
only in the case $(a_n, b_n) = (0, 1)$ or $(1, b)$ for $b \leq 1$ 
of the proof of Theorem~\ref{thm:S-F}
and $\crit(\pi)=\crit(\hat{\pi})$ in the rest of the cases.
So for all $\pi \in \bigcup_{n\geq 0}\perm_{n+1}$,
we obtain $$\crit(\pi)=\aone(\phi_{\rm S}(\pi))+\north(\phi_{\rm S}(\pi))+1.$$
In conclusion, we also obtain 
\begin{align*}
\height(\phi_{\rm S}(\pi)) &= \block(\pi)-1,\\
\north(\phi_{\rm S}(\pi)) &= \asc(\pi),\\
\aone(\phi_{\rm S}(\pi)) &= \crit(\pi)-\asc(\pi)-1,\\
\bone((\phi_{\rm F}\circ\phi_{\rm S})(\pi)) &= \crit(\pi)-1,
\end{align*}
where $\phi_{\rm F}$ is the involution on $ \bigcup_{n\geq 0}\Fp_n$ introduced in subsection 2.1.
\end{enumerate}
\end{rem}


\section{Pattern avoiding inversion sequences}
\label{sec:I-J}

A sequence \( e=(e_1, e_2, \dots, e_n) \) is called an \emph{inversion sequence of length \( n \)} if \( 0 \leq e_i <i \) for all \( i \in [n] \). We denote the largest letter in \( e \) by \( \max(e) \). 
If the position of the rightmost occurrence of \( \max(e) \) is \( k \), i.e., \( e_k= \max(e) \) and  \( e_i< \max(e) \) for all \( i>k \), then we define \( \maxid(e)\coloneqq k \) and 
\[ \hat{e}\coloneqq(e_1, \dots, e_{k-1}, e_{k+1}, \dots, e_n). \] 
Clearly, we have
$$\max(e)\ge \max(\hat{e})\quad\text{and}\quad\max(e)\le \maxid(e)-1.$$

Pattern avoidance in inversion sequences can be defined in a similar way in permutations. Given a word \( w \in \{ 0,1, \dots, k-1 \}^{k} \), the \emph{reduction} of \( w \) is the word obtained by replacing the \( i \)th smallest entries of \( w \) with \( i-1 \). 
We say that an inversion sequence \( e=(e_1, e_2, \dots, e_n)\) \emph{contains} the pattern \( w \) if there exist some indices \( i_1 < i_2 < \cdots< i_k \) such that the reduction of \( e_{i_1} e_{i_2} \dots e_{i_k} \) is \( w \). Otherwise, \( e \) is said to \emph{avoid} the pattern \( w \). 
When an inversion sequence $e$ avoids two patterns \( w_1 \) and \( w_2 \), $e$ is called a $(w_1, w_2)$-avoiding inversion sequence.


\subsection{(101,102)-avoiding inversion sequences}
\label{sec:I}
Let $\invi_n$ be the set of \( (101,102) \)-avoiding inversion sequences of length $n$.
We construct a simple bijection from \( (101,102) \)-avoiding inversion sequences to \( F \)-paths. 
Define
\[ 
\phi_{\rm I} : \bigcup_{n\geq 0}\invi_{n+1} \to \bigcup_{n\geq 0} \mathcal{F}_n 
\] 
recursively as follows. 
Let \( \phi_{\rm I}(0)\coloneqq \emptyset \). 
Now suppose \( n \geq 1 \). 
For \( e=(e_1,e_2,\dots, e_{n+1})\in \invi_{n+1} \) 
with
\begin{align}
\max(e)-\max(\hat{e})=a_n \quad\text{ and }\quad \maxid(e)-\maxid(\hat{e})=b_n, \label{eq:I-F}
\end{align}
let \( \phi_{\rm I}(e) \coloneqq \phi_{\rm I}(\hat{e}) (a_n,b_n) \). 
Note that \( \hat{e} \in \mathcal{I}_n(101,102) \) 
and \( (a_n,b_n)\in F \) since \( e \) avoids the patterns \( 101 \) and \( 102 \).
If \( \phi_{\rm I}(\hat{e})=(a_1,b_1)(a_2,b_2)\dots (a_{n-1},b_{n-1}) \), 
then it satisfies that \( \sum_{j \leq i} a_j \leq \sum_{j \leq i} b_j \) for each \( i\in[n-1] \), 
and \( \sum_{j\leq n-1}a_j=\max(\hat{e}) \) 
and \( \sum_{j \leq n-1} b_j =\maxid(\hat{e})-1 \) 
by the construction. 
Moreover, we also have 
\begin{align*} 
\sum_{j\le n} a_j 
= \max(\hat{e})+\left(\max(e)-\max(\hat{e})\right) 
&= \max(e)\\
&\le \maxid(e)-1
= \maxid(\hat{e})+\left(\maxid(e)-\maxid(\hat{e})\right)-1  
= \sum_{j\le n} b_j 
\end{align*}
so that \( \phi_{\rm I}(e) \in \mathcal{F}_n \). 
Hence, the map is well-defined and injective. 
Furthermore, we can recover \( e \) by finding \( \max(e) \) and \( \maxid(e) \) using \eqref{eq:I-F}.
Hence, \( \phi_{\rm I} \) is a bijection 
with the property 
\[ \height(\phi_{\rm I}(e))=\maxid(e)-\max(e)-1.\]

\begin{example}
There are six \( (101,102) \)-inversion sequences of length \( 3 \). From the construction of \( \phi_{\rm I} \), we have 
\begin{align*}
\phi_{\rm I}(0,1,0)&=(0,1)(1,0)=Q^{(1)}, &
\phi_{\rm I}(0,0,2)&=(0,1)(2,1)=Q^{(2)}, &
\phi_{\rm I}(0,1,2)&=(1,1)(1,1)=Q^{(3)}, \\
\phi_{\rm I}(0,0,1)&=(0,1)(1,1)=Q^{(4)}, &
\phi_{\rm I}(0,1,1)&=(1,1)(0,1)=Q^{(5)}, &
\phi_{\rm I}(0,0,0)&=(0,1)(0,1)=Q^{(6)},
\end{align*}
where each $Q^{(i)}\in\Fp_2$ is given in Figure~\ref{fig:F2}.
We also check that \( \height(\phi_{\rm I}(e))=\maxid(e)-\max(e)-1\) for $e\in \invi_{3}$. 
\end{example}

From this construction, we have the following result.

\begin{thm}\label{thm:I-F}
Let \( n \) be a nonnegative integer. For each step \( (a,b)\in F \), we have
\[
|\{ e\in \invi_{n+1} \deli \max(e)-\max(\hat{e})=a, ~\maxid(e)-\maxid(\hat{e})=b \}|=|\mathcal{F}_n\{ (a,b) \}|.
\]
Consequently, \( |\invi_{n+1}|=|\mathcal{F}_n| \).
\end{thm}

\begin{rem}\label{rem:I-F}
For $e\in\invi_{n}$, we define
$\omi(e)\coloneqq |\{ i\in [n] \deli i \notin e \}|$ and
$\cons(e)\coloneqq |\{ i\in [n-1] \deli i-1, i \in e \}|$.
By Theorem~\ref{thm:I-F}, it
is easy to check that
$$\omi(e)=\north(\phi_{\rm I}(e))\quad\text{and}\quad\cons(e)=\aone(\phi_{\rm I}(e))$$
for $e \in \bigcup_{n\geq 0}\invi_{n+1}$.
\end{rem}

%


\subsection{(101,021)-avoiding inversion sequences}
\label{sec:J}
Let \( \invj_n \) be the set of \( (101,021) \)-avoiding inversion sequences of length \( n \).
For $\je \in \invj_n$, we will refer to $\je=(\je_1,\je_2,\dots,\je_n)$ as $\je=\je_1 \je_2\dots \je_n$ for convenience.
Corteel et al. \cite{CMSW16} showed that an inversion sequence is \( 021 \)-avoiding if and only if its positive elements are weakly increasing. Hence, an inversion sequence \( \je  \in \invj_{n+1}\) is \( (101,021) \)-avoiding if and only if 
$\je$ is of the form 
\begin{align}\label{jinv-form}
0^{i_0}1^{j_1}0^{i_1}\cdots n^{j_n} 0^{i_n} 
\end{align}
with nonnegative exponents 
such that $j_k=0$ implies $i_k=0$. We write \( \first(\je)=i_0 \) and 
\( \single(\je)=| \{ k \in [n] \deli j_k=1 \}| \). 
Now we construct a simple bijection from \( (101,021) \)-avoiding inversion sequences to \( F \)-paths. Define
\[ 
\phi_{\rm J} : \bigcup_{n\geq 0}\invj_{n+1} \to \bigcup_{n\geq 0} \mathcal{F}_n 
\] 
as follows. Let \( \phi_{\rm J}(0)\coloneqq \emptyset \). Now suppose that \( n \geq 1 \). 
If $\je\in\invj_{n+1}$ is of the form
\[ 
0^{\first(\je)}1^{a_n}0^{1-b_n}2^{a_{n-1}}0^{1-b_{n-1}}\dots n^{a_1}0^{1-b_1},
\] 
then we define \( \phi_{\rm J}(\je) \coloneqq  (a_1,b_1)(a_2,b_2)\dots(a_n,b_n) \). 

\begin{example}
There are six \( (101,021) \)-inversion sequences of length \( 3 \). From the construction of \( \phi_{\rm J} \), we have 
\begin{align*}
\phi_{\rm J}(010)&=(0,1)(1,0)=Q^{(1)}, &
\phi_{\rm J}(011)&=(0,1)(2,1)=Q^{(2)}, &
\phi_{\rm J}(012)&=(1,1)(1,1)=Q^{(3)}, \\
\phi_{\rm J}(001)&=(0,1)(1,1)=Q^{(4)}, &
\phi_{\rm J}(002)&=(1,1)(0,1)=Q^{(5)}, &
\phi_{\rm J}(000)&=(0,1)(0,1)=Q^{(6)},
\end{align*}
where each $Q^{(i)}\in\Fp_2$ is given in Figure~\ref{fig:F2}.
One can see that 
\( \height(\phi_{\rm J}(\je))=\first(\je)-1\),
\( \north(\phi_{\rm J}(\je))=\omi(\je)\), and 
\( \aone(\phi_{\rm J}(\je))=\single(\je)\)
for $\je\in\invj_3$.
\end{example}

In \cite[Section 4.1]{YanLin20}, Yan and Lin constructed a simple bijection \( \mathcal{H} \) from the set \( \invj_{n} \) to the set \( \mathcal{D}_n \) of \emph{colored Dyck paths} of semilength \( n \) consisting of east \( (1,0) \) and north \( (0,1) \) steps satisfying 

\begin{itemize}
\item all steps are weakly below the line \( y=x \), 
\item each east step is colored black or red,
\item all the east step of height \( 0 \) are colored red, 
\item the first east step of each positive height is colored black,
\item red east steps of the same height are connected.   
\end{itemize}

We note that these paths \( D \) are obtained by rotating the restricted bicolored paths \( B \) of semilength \( n \) clockwise by 135 degrees. We write \( \rho(D)=B \). The map \( {\phi_{\rm J}} \) can be described as \( \phi_{\rm B} \circ \rho \circ \mathcal{H} \). Thus, \( \phi_{\rm J} \) is a bijection.

From the construction of \( \phi_{\rm J} \), we have the following result.

\begin{thm}\label{thm:J-F}
Let \( n \) be a nonnegative integer. For each step \( (a,b)\in F \), we have
\[
|\{ 0^{i_0}1^{j_1}0^{i_1}\cdots n^{j_n} 0^{i_n}\in \invj_{n+1} \deli  j_1=a,~i_1=1-b\}|=|\mathcal{F}_n\{ (a,b) \}|.
\]
Consequently, \( |\invj_{n+1}|=|\mathcal{F}_n| \).
\end{thm}

\begin{prop}\label{prop:J-F}
For a nonnegative integer \( n \), let \( \je\in \invj_{n+1} \). If \( \phi_{\rm J}(\je)=Q \), then the following hold.
\begin{enumerate}
\item \( \height(Q)=\first(\je)-1 \),
\item \( \north(Q)=\omi(\je) \), 
\item \( \aone(Q)=\single(\je) \). 
\end{enumerate}
\end{prop}

\begin{proof}
If 
\(\je= 0^{\first(\je)}1^{a_n}0^{1-b_n}2^{a_{n-1}}0^{1-b_{n-1}} \dots n^{a_1}0^{1-b_1} \), 
then \( Q=(a_1,b_1)(a_2,b_2) \dots (a_n,b_n) \) and
\[
\height(\phi_{\rm J}(\je))=\sum_{i\leq n}(b_i-a_i) = n- \sum_{i \leq n} (a_i -b_i+1) =\first(\je)-1.
\]
The remaining follows since \( \phi_{\rm J}^{-1} \) sends the step \( (a_i,b_i) \) to the segment \(  (n-i+1)^{a_i}0^{1-b_i} \). 
\end{proof}


\section{Weighted ordered trees}
\label{sec:T}

An \emph{ordered tree} (also called a \emph{plane tree} or \emph{Catalan tree}) \( T \) can be defined recursively as follows. One special vertex is called the \emph{root} of \( T \) and denoted by \( r_T \). Either \( T \) consists of the single vertex \( r_T \) or else it has a sequence \( (T_1, T_2, \dots, T_m) \) of subtrees \( T_i \) for \( 1\leq i \leq m \), each of which is an ordered tree. The root \( r_T \) is written on the top, with an edge drawn from \( r_T \) to the root of each of its subtrees, which are written in the  \emph{right-to-left} order. 
In an ordered tree \( T \), the number of descendants of a vertex \( v\in T \) is called the \emph{outdegree} of \( v \) and denoted by \( \deg^{+}(v) \). We call a vertex \( v \) with \( \deg^{+}(v)=0 \) a \emph{leaf}. Let us denote the number of leaves in \( T \) by \( \leaf(T) \).
A \emph{weighted ordered tree} is an ordered tree, where each interior vertex (non-root, non-leaf) \( v \) is weighted by a positive integer \( \wt(v) \) less than or equal to its outdegree. For convenience, sometimes we write \( \wt(v)=0 \) for a leaf \( v \). 
Given a weighted ordered tree $T$, let $\one(T)$ be the number of vertices $v$ in $T$ with $\wt(v)=1$.
Let \( \mathcal{T}_{n} \) denote the set of all the weighted ordered trees with \( n \) edges. Figure~\ref{fig:T3} shows the elements in \( \mathcal{T}_3 \).

Yan and Lin \cite[Theorem 6.1]{YanLin20} also constructed a bijection \( \Phi \) from the set \( \mathcal{T}_{n} \) to the set \( \mathcal{D}_n \). We define a bijection 
\[ 
\phi_{\rm T} : \bigcup_{n\geq 0}\mathcal{T}_{n+1} \to \bigcup_{n\geq 0} \mathcal{F}_n 
\] 
by \( \phi_{\rm B} \circ \rho \circ \Phi \), or equivalently, as follows. For \( T\in \mathcal{T}_{1} \), let \(  \phi_{\rm T}(T)\coloneqq \emptyset \). Now consider \( T\in \mathcal{T}_{n+1} \) with \( n\geq 1 \). 
Let \( v_0(=r_T), v_1, \dots, v_{n+1} \) be the vertices of \( T \) in \emph{preorder} (also known as \emph{depth first search}). 
We let \( \phi_{\rm T}(T)\coloneqq (a_1,b_1)(a_2,b_2)\dots(a_n,b_n)\), where 
\[ 
(a_i,b_i)=\begin{cases}
(\wt(v),\wt(v)-\deg^{+}(v)+1) & \mbox{ if \( v_{n-i+2} \) is the left-most child of an interior vertex \( v \),}\\
(0,1) & \mbox{ otherwise}.
\end{cases}
\]

\begin{figure}[t]
\centering

\begin{tikzpicture}[scale=0.7]
\coordinate (0) at (0,0);
\coordinate (1) at (0,-1);
\coordinate (2) at (-0.7,-2);
\coordinate (3) at (0.7,-2);

\foreach \i in {0,1,2,3}
\filldraw (\i) circle (2.5pt);
\draw[ultra thick]
(0) -- (1) 
(1) -- (2) 
(1) -- (3);

\node[above] at (0) {\scriptsize \color{gray} \( v_0\)};
\node[left] at (1) {\scriptsize \color{gray} \( v_1\)};
\node[below] at (2) {\scriptsize \color{gray} \( v_2\)};
\node[below] at (3) {\scriptsize \color{gray} \( v_3\)};

\node[above=2pt, right] at (1) {\( 1 \)};

\node at (0,-4) {\( T^{(1)} \)};
\end{tikzpicture}
\qquad 
\begin{tikzpicture}[scale=0.7]
\coordinate (0) at (0,0);
\coordinate (1) at (0,-1);
\coordinate (2) at (-0.7,-2);
\coordinate (3) at (0.7,-2);

\foreach \i in {0,1,2,3}
\filldraw (\i) circle (2.5pt);
\draw[ultra thick]
(0) -- (1) 
(1) -- (2) 
(1) -- (3);

\node[above] at (0) {\scriptsize \color{gray} \( v_0\)};
\node[left] at (1) {\scriptsize \color{gray} \( v_1\)};
\node[below] at (2) {\scriptsize \color{gray} \( v_2\)};
\node[below] at (3) {\scriptsize \color{gray} \( v_3\)};

\node[above=2pt, right] at (1) {\( 2 \)};

\node at (0,-4) {\( T^{(2)} \)};
\end{tikzpicture}
\qquad 
\begin{tikzpicture}[scale=0.7]
\coordinate (0) at (0,0);
\coordinate (1) at (0,-1);
\coordinate (2) at (0,-2);
\coordinate (3) at (0,-3);

\foreach \i in {0,1,2,3}
\filldraw (\i) circle (2.5pt);
\draw[ultra thick]
(0) -- (1) 
(1) -- (2) 
(1) -- (3);

\node[above] at (0) {\scriptsize \color{gray} \( v_0\)};
\node[left] at (1) {\scriptsize \color{gray} \( v_1\)};
\node[left] at (2) {\scriptsize \color{gray} \( v_2\)};
\node[below] at (3) {\scriptsize \color{gray} \( v_3\)};

\node[right] at (1) {\( 1 \)};
\node[right] at (2) {\( 1 \)};
\node at (0,-4) {\( T^{(3)} \)};
\end{tikzpicture}
\qquad
\begin{tikzpicture}[scale=0.7]
\coordinate (0) at (0,0);
\coordinate (1) at (-0.7,-1);
\coordinate (2) at (0.7,-1);
\coordinate (3) at (-0.7,-2);

\foreach \i in {0,1,2,3}
\filldraw (\i) circle (2.5pt);
\draw[ultra thick]
(0) -- (1) 
(0) -- (2) 
(1) -- (3);

\node[above] at (0) {\scriptsize \color{gray} \( v_0\)};
\node[left] at (1) {\scriptsize \color{gray} \( v_1\)};
\node[below] at (2) {\scriptsize \color{gray} \( v_3\)};
\node[below] at (3) {\scriptsize \color{gray} \( v_2\)};

\node[right] at (1) {\( 1 \)};
\node at (0,-4) {\( T^{(4)} \)};
\end{tikzpicture}
\qquad
\begin{tikzpicture}[scale=0.7]
\coordinate (0) at (0,0);
\coordinate (1) at (-0.7,-1);
\coordinate (2) at (0.7,-1);
\coordinate (3) at (0.7,-2);

\foreach \i in {0,1,2,3}
\filldraw (\i) circle (2.5pt);
\draw[ultra thick]
(0) -- (1) 
(0) -- (2) 
(2) -- (3);

\node[above] at (0) {\scriptsize \color{gray} \( v_0\)};
\node[below] at (1) {\scriptsize \color{gray} \( v_1\)};
\node[right] at (2) {\scriptsize \color{gray} \( v_2\)};
\node[below] at (3) {\scriptsize \color{gray} \( v_3\)};

\node[left] at (2) {\( 1 \)};
\node at (0,-4) {\( T^{(5)} \)};
\end{tikzpicture}
\qquad
\begin{tikzpicture}[scale=0.7]
\coordinate (0) at (0,0);
\coordinate (1) at (-1,-1);
\coordinate (2) at (0,-1);
\coordinate (3) at (1,-1);

\node[above] at (0) {\scriptsize \color{gray} \( v_0\)};
\node[below] at (1) {\scriptsize \color{gray} \( v_1\)};
\node[below] at (2) {\scriptsize \color{gray} \( v_2\)};
\node[below] at (3) {\scriptsize \color{gray} \( v_3\)};

\foreach \i in {0,1,2,3}
\filldraw (\i) circle (2.5pt);
\draw[ultra thick]
(0) -- (1) 
(0) -- (2) 
(0) -- (3);

\node at (0,-4) {\( T^{(6)} \)};
\end{tikzpicture}

\caption{The weighted ordered trees with \( 4 \) vertices and their preorder.}\label{fig:T3}
\end{figure}
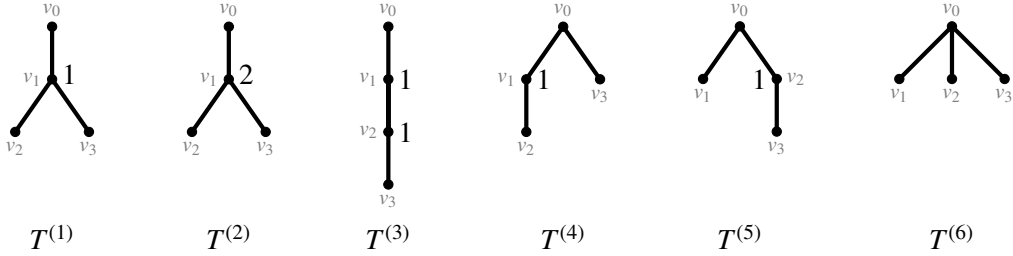


\begin{example}
For the six weighted ordered trees \( T^{(i)} \in \wot_3\) depicted in Figure~\ref{fig:T3}, we have 
\begin{align*}
\phi_{\rm T}(T^{(1)})&=(0,1)(1,0)=Q^{(1)}, &
\phi_{\rm T}(T^{(2)})&=(0,1)(2,1)=Q^{(2)}, &
\phi_{\rm T}(T^{(3)})&=(1,1)(1,1)=Q^{(3)}, \\
\phi_{\rm T}(T^{(4)})&=(0,1)(1,1)=Q^{(4)}, &
\phi_{\rm T}(T^{(5)})&=(1,1)(0,1)=Q^{(5)}, &
\phi_{\rm T}(T^{(6)})&=(0,1)(0,1)=Q^{(6)},
\end{align*}
where each $Q^{(i)}\in\Fp_2$ is given in Figure~\ref{fig:F2}.
One can notice that \( \height(Q^{(i)})=\deg^{+}(r_{T^{(i)}})-1\),
\( \north(Q^{(i)})=\leaf(T^{(i)})-1\), and \( \aone(Q^{(i)})=\one(T^{(i)})\). 
\end{example}

From the construction of the bijection $\phi_{\rm T}$, we have the following results.

\begin{thm}\label{thm:T-F}
Let \( n \) be a nonnegative integer. For each step \( (a,b)\in F \), we have
\[
|\{T \in \mathcal{T}_{n+1} \deli \wt(v_1)=a,~\deg^{+}(v_1)=a-b+1 \}|=|\mathcal{F}_n\{ (a,b) \}|,
\]
where \( v_1 \) is the left-most child of the root of $T$.
Consequently, \( |\mathcal{T}_{n+1}|=|\mathcal{F}_n| \).
\end{thm}

\begin{proof}
For $T\in\wot_{n+1}$, let \( \phi_{\rm T}(T)= (a_1,b_1)(a_2,b_2)\dots(a_n,b_n)\).
Since $v_2$ is the leftmost child of an interior vertex if and only if $v_2$ is a descendant of $v_1$, we have
\[ 
(a_n,b_n)=\begin{cases}
(\wt(v_1),\wt(v_1)-\deg^{+}(v_1)+1) & \text{if $v_1$ is an interior vertex,}\\
(0,1) & \text{if $v_1$ is a leaf.}
\end{cases}
\]
Hence, it is clear that $\wt(v_1)=a_n$ and $\deg^{+}(v_1) = a_n- b_n +1$.
\end{proof}

\begin{prop}\label{prop:T-F}
For a nonnegative integer \( n \), let \( T\in \mathcal{T}_{n+1} \). If \( \phi_{\rm T}(T)=Q \), then the following hold.
\begin{enumerate}
\item \( \height(Q)=\deg^{+}(r_T)-1 \).
\item \( \north(Q)=\leaf(T)-1 \).
\item \( \aone(Q)=\one(T) \).
\end{enumerate}
\end{prop}

\begin{proof}
First, 
\[
\height(\phi_{\rm T}(T))=\sum_{i\leq n}(b_i-a_i) = n-\left( \sum_{v:\,\text{interior}} \deg^{+}(v) \right)=\deg^{+}(r_T)-1.
\]
The second is easily obtained since for \( (a_i,b_i) \), \( a_i>0 \) if and only if the parent \( v \) of \( v_{n-i+2} \) is an interior vertex with \( \wt(v) \).
The third is clear by the definition of $\phi_{\rm T}$.
\end{proof}



\section{Enumerations with respect to statistics}
\label{sec:enum}

\begin{table}[t]
\begin{center}
\footnotesize{
  \begin{tabular}{c||>{\centering}m{7em}||>{\centering}m{4em}|>{\centering}m{6em}|>{\centering}m{7em}|>{\centering}m{7em}|m{8em}<{\centering}}
       $\Fp_n$ 
       & $({a},{b})=(a_n, b_n)$ 
       &  $(0,1)$ & $({1},1)$ & $({k}+2,1)$ & $({1},-{j})$ & $({k}+2, -{j})$ \\
    \hline\hline
       $\schp_n$ 
       & $\begin{matrix} \text{suffix} \\ \comp(Y)={a}-2,\\ \comp(Z)=-{b} \end{matrix}$
       & $\dots \black{h}$ 
       & $\dots \black{ud}$ 
       & $\begin{matrix}\dots u Y \black{hd}\\ \comp(Y)={k} \end{matrix}$
       & $\begin{matrix}\dots u Z \black{udd} \\  \comp(Z)={j} \end{matrix}$ 
       & $\begin{matrix}\dots u Y u Z \black{hdd} \\ \comp(Y)={k}  \\ \comp(Z)={j}  \end{matrix}$ \\
    \hline
       $\bdp_{n+1}$ 
    & suffix $\dots{{u{\dr}^{1-{b}}{\db}^{{a}}}}\,u\dr^{\ge1}$ 
	& $\dots\black{u}\,u\dr^{\ge1}$ 
	& $\dots u{{\db}^{1}}\,u\dr^{\ge1}$ 
       & $ \dots u{\db}^{{k}+2}\,u\dr^{\ge1}$ 
	& $ \dots u {\dr}^{{j}+1}{\db}^{{1}}\,u\dr^{\ge1}$ 
	& $ \dots u {\dr}^{{j}+1}{\db}^{{k}+2}\,u\dr^{\ge1}$ \\
    \hline
       $ \perm_{n+1}$ 
& 
	   $\block(\omega)$ $= {a}-1 + \delta_{{a},0}$, 
	   $\block(\tau)$  $= 1-{b} + \delta_{{a},1}$
& 
       $z=n$, $z<w$ 
& 
       $z<n$,\quad$z<w$, $\block(\tau)={1}$ 
& 
       $z=n$,\quad$z>w$, $\block(\omega)={k}+1$ 
& 
       $z<n$,\quad$z<w$, $\block(\tau)={j}+2$ 
& 
       $z<n$,\quad$z>w$, $\block(\omega)={k}+1$, $\block(\tau)={j}+1$
\\
    \hline
       $\invi_{n+1}$ 
       & ${a}=r-s$, ${b}=t-u$ \linebreak 
    where $r=\max(e)$, $s=\max(\hat{e})$, $t=\maxid(e)$, $u=\maxid(\hat{e})$  
       & $\begin{matrix}\dots s\, r \dots \\ r-s=0 \end{matrix}$ 
       & $\begin{matrix}\dots s\, r \dots \\ r-s={1} \end{matrix}$  
	& $\begin{matrix}\dots s\, r \dots \\ r-s={k}+2 \end{matrix}$
	& $\begin{matrix}\dots r\overbrace{\dots}^{{j}} s \dots \\ r-s={1}\end{matrix}$
	& $\begin{matrix}\dots r\overbrace{\dots}^{{j}} s \dots \\ r-s={k}+2\end{matrix}$\\ 
    \hline
       $\invj_{n+1}$ 
	& prefix $0^{\ge1}{{1}^{a} {0}^{1-{b}}} \dots$  
	& $0^{\ge1} \dots$  
	& $0^{\ge1}{1^1} \dots$ 
	& $0^{\ge1}{1^{k+2}} \dots$
	& $0^{\ge1}{1^1}{0}^{{j}+1} \dots$  
	& $0^{\ge1}{1}^{{k}+2}{0}^{{j}+1}$ \dots\\ 
    \hline
       ${\wot_{n+1}}$ 
       & ${a}=\omega$, ${b}=\omega-\delta+1$ where $\delta=\deg^+(v_1)$, $\omega=\wt(v_1)$ 
       & $\delta=0$ $\omega=0$ \linebreak ($v_1$: leaf)
       & $\begin{matrix} \delta={1} \\ \omega={1} \end{matrix}$
       & $\begin{matrix} \delta={k}+2 \\ \omega={k}+2  \end{matrix}$
       & $\begin{matrix} \delta={j}+2 \\ \omega={1} \end{matrix}$ 
	 & $\begin{matrix} \delta={k}+{j}+3  \\ \omega={k}+2 \end{matrix}$ \\ 
  \end{tabular}\\[\baselineskip]
}
\end{center}
\caption{Correspondences between $F$-paths and others 
due to the last step of the $F$-paths.}\label{table:tail} 
\end{table}

Theorem~\ref{thm:main} follows from
Theorems \ref{thm:P-F}, \ref{thm:B-F}, \ref{thm:S-F}, \ref{thm:I-F}, \ref{thm:J-F}, and \ref{thm:T-F}. We can also classify each combinatorial object that corresponds to $F$-paths $Q$ by the last step of $Q$. (See Table~\ref{table:tail}.)
In addition, if 
$$Q= \phi_{\rm P}(P)= \phi_{\rm B}(B) =\phi_{\rm S}(\pi)= \phi_{\rm I}(e) = \phi_{\rm J}(\je)= \phi_{\rm T}(T),$$ 
then by Propositions~\ref{prop:B-F},~\ref{prop:J-F},~\ref{prop:T-F} and Remarks~\ref{rem:P-F},~\ref{rem:S-F},~\ref{rem:I-F}, we obtain Table~\ref{table:stat}.
\begin{table}[t]
\begin{center}
  \begin{tabular}{>{\centering}p{4em}||>{\centering}p{8em}|>{\centering}p{5em}|p{9em}<{\centering}}
     Family & 1st stat. & 2nd stat. & 3rd stat. \\
    \hline\hline 
     $Q\in\Fp_n$ & $\height(Q)$ &  $\north(Q)$ & $\aone(Q)$ \\
    \hline
    $Q'\in\Fp_n$ & $\height(Q')$ &  $\north(Q')$ & $\bone(Q')-\north(Q')$ \\
    \hline
     $P\in\schp_n$ & $\comp(P)$ & $\hdd(P)$ & $\peak(P)$ \\
    \hline
     $B\in\bdp_{n+1}$ & $\last(B)-1$ & $\dasc(B)$ & $\bval(B)$ \\
    \hline
     $\pi\in\perm_{n+1}$ & $\block(\pi)-1$ &  $\asc(\pi)$ & $\crit(\pi)-\asc(\pi)-1$ \\
    \hline
     $e\in\invi_{n+1}$ & $\maxid(e)-\max(e)-1$  & $\omi(e)$ & $\cons(e)$  \\
    \hline
     $\je\in\invj_{n+1}$ & $\first(\je)-1$ &  $\omi(\je)$ & $\single(\je)$ \\
    \hline
     $T\in\wot_{n+1}$ & $\deg(r_T)-1$ &  $\leaf(T)-1$ & $\one(T)$ \\
  \end{tabular}\\[\baselineskip]
\end{center}\caption{Equinumerous statistics in $F$-paths family.}\label{table:stat}
\end{table}

In this section, 
we count the number of objects in \emph{$F$-paths family}
with respect to several statistics. 

\begin{lem}\label{lem:joint}
For nonnegative integers $i$, $j$, $k$, $\ell$, $m$, and $n$, 
let $f_n(i,j,k,\ell,m)$ be the number of $F$-paths $Q\in\Fp_{n}$ that satisfy
\begin{itemize}
\item $|\{s\in Q\deli s=(1,1)\}|=i$,
\item $|\{s\in Q\deli s=(1,b),~b\le0\}|=j$,
\item $|\{s\in Q\deli s=(a,1),~a\ge2\}|=k$,
\item $|\{s\in Q\deli s=(0,1)\}|=\north(Q)=\ell$,
\item $\height(Q)=m$.
\end{itemize}
Set $n':=n-(i+j+k+\ell)$. Then $f_n(i,j,k,\ell,m)$ is equal to
\begin{align}
\frac{m+1}{n+1}\binom{n+1}{i,\,j,\,k,\,\ell+1,\,n'}
\binom{\ell-m-1}{2n'+j+k-1}.
\end{align}
\end{lem}

\begin{proof}

For nonnegative integer $m$, let 
\[
F_m=F_m(x, r, s, t, u)\coloneqq\sum_{n\ge0}f_n(i,j,k,\ell,m) r^i s^j t^k u^{\ell}   x^{n}.
\] 
By decomposition of $F$-paths in \( \bigcup_{n\geq 0}\mathcal{F}_{n} \), we have 
$$F_m=(ux)^m F_0^{m+1}.$$
Given an $F$-path $Q$ of height $m$, $Q(a,b)$ is an $F$-path of height $0$ if and only if $a-b=m$ and $1\le a \le m+1$.
So we get
\begin{equation*}
F_0=1+rxF_0+\sum_{m\ge 1}\left(s+t+m-1\right)x\, F_m.
\end{equation*}
Removing $F_m$ yields that
\begin{align*}
F_0&=1+rxF_0+(s+t)x\sum_{m\ge 1} (ux)^m F_0^{m+1}+x\sum_{m\ge 1}(m-1) (ux)^m F_0^{m+1}.
\end{align*}
By setting $G_m:=uxF_m$, we get 
\begin{align}\label{eq:Tdecomp-lem}
G_m&=G_0^{m+1}, \\
G_0&=ux+rxG_0+(s+t)x\sum_{m\ge 1} G_0^{m+1}+x\sum_{m\ge 1}(m-1) G_0^{m+1}. \notag
\end{align}
Thus, we have
\begin{equation}
{G_0}=x\left(u+rG_0+(s+t)\frac{G_0^2}{1-G_0}+\frac{G_0^3}{(1-G_0)^2} \right).\label{eq:TArstu0feq}
\end{equation}
Applying Lagrange inversion formula~\cite[eq. (2.2.1)]{Ges16} to \eqref{eq:TArstu0feq} with \eqref{eq:Tdecomp-lem}, we obtain
\[
\left[x^{n+1}\right]{G_m}(x,r, s,t,u)=\dfrac{m+1}{n+1}\left[z^{n-m}\right]
\left(u+rz + s\dfrac{z^2}{1-z} + t\dfrac{z^2}{1-z} + \dfrac{z^3}{(1-z)^2} \right)^{n+1}.
\]
Since 
\begin{align*}
f_n(i,j,k,\ell,m)
&=\left[r^i s^j t^k u^{\ell} x^n \right]\left({F_m}(x, r, s,t,u)\right),\\
&=\left[r^i s^j t^k u^{\ell+1} \right]\left(\left[x^{n+1}\right]{G_m}(x, r, s,t,u)\right),
\end{align*}
we get
\begin{align*}
f_n(i,j,k,\ell,m)
&=\left[r^i s^j t^k u^{\ell+1} \right]\frac{m+1}{n+1}\left[z^{n-m}\right] 
\left(zr + \dfrac{z^2}{1-z}s + \dfrac{z^2}{1-z}t + u + \dfrac{z^3}{(1-z)^2} \right)^{n+1} \\
&=\frac{m+1}{n+1}\left[z^{n-m}\right] \binom{n+1}{i,\,j,\,k,\,\ell+1,\,n-i-j-k-\ell}
\, z^i\left(\dfrac{z^2}{1-z}\right)^j\left(\dfrac{z^2}{1-z}\right)^k\left(\dfrac{z^3}{(1-z)^2}\right)^{n-i-j-k-\ell}\\
&=\frac{m+1}{n+1}\binom{n+1}{i,\,j,\,k,\,\ell+1,\,n-i-j-k-\ell}
\left[z^{3\ell+2i+j+k-m-2n}\right]\frac{1}{(1-z)^{2n-2\ell-2i-j-k}} \\
&=\frac{m+1}{n+1}\binom{n+1}{i,\,j,\,k,\,\ell+1,\,n-i-j-k-\ell}
\binom{\ell-m-1}{2n-2\ell-2i-j-k-1},
\end{align*}
which completes the proof.
\end{proof}

\begin{thm}\label{prop:joint} 

For nonnegative integers $m$, $\ell$, $h$, and $n$, 
the following sets are equinumerous.
\begin{enumerate}
\item The set of \( F \)-paths \( Q \in \Fp_n \) with
$$\height(Q)=m,\quad \north(Q)=\ell,\quad \aone(Q)=h.$$
\item The set of Schr\"{o}der paths \( P \in \schp_n \) with 
$$\comp(P)=m, \quad \hdd(P)=\ell,\quad \peak(P)=h.$$
\item The set of restricted bicolored Dyck paths \( B \in \bdp_{n+1} \) with 
$$\last(B)=m+1, \quad \dasc(B)=\ell, \quad \bval(B)=h.$$
\item The set of permutations \( \pi \in \perm_{n+1} \) with 
$$\block(\pi)=m+1, \quad \asc(\pi)=\ell,\quad \crit(\pi)-\asc(\pi)=h+1.$$
\item The set of inversion sequences \( e \in \invi_{n+1} \) with
$$\maxid(e)-\max(e)=m+1, \quad \omi(e)=\ell, \quad \cons(e)=h.$$
\item The set of inversion sequences \( \je \in \invj_{n+1} \) with 
$$\first(\je)=m+1, \quad \omi(\je)=\ell, \quad \single(\je)=h.$$
\item The set of weighted ordered trees \( T\in \wot_{n+1} \) with 
$$\deg^{+}(r_T)=m+1, \quad \leaf(T)=\ell+1,\quad \one(T)=h.$$
\end{enumerate}
The number $a_n(h,\ell,m)$ of elements in each set is 
\[ 
\frac{m+1}{n+1} \binom{n+1}{\ell+1} \binom{n-\ell}{h} \binom{n-m-1}{2n-2\ell-h-1}.
\]
\end{thm}

\begin{proof} From Lemma~\ref{lem:joint}, we have
$$a_n(h,\ell,m)=\sum_{i+j=h}\sum_{k\ge0} f_n(i,j,k,\ell,m).$$
Applying Chu-Vandemonde identities, we obtain
\begin{align*}
\sum_{i=0}^h\sum_{k\ge0} f_n(i,h-i,k,\ell,m)
&= \frac{m+1}{n+1}\binom{n+1}{\ell+1}
\sum_{i=0}^h\binom{n-\ell}{i}\binom{n-\ell-i}{h-i}
\sum_{k\ge0}\binom{n-\ell-h}{k}
\binom{\ell-m-1}{2n-2\ell-h-i-k-1}\\
&=\frac{m+1}{n+1}\binom{n+1}{\ell+1}\sum_{i=0}^h\binom{n-\ell}{i}\binom{n-\ell-i}{h-i}
\binom{n-m-h-1}{2n-2\ell-h-i-1}\\
&=\frac{m+1}{n+1}\binom{n+1}{\ell+1}\binom{n-\ell}{h}\sum_{i=0}^h\binom{h}{i}
\binom{n-m-h-1}{2n-2\ell-h-i-1}\\
&=\frac{m+1}{n+1}\binom{n+1}{\ell+1}\binom{n-\ell}{h}\binom{n-m-1}{2n-2\ell-h-1},
\end{align*}
which completes the proof.
\end{proof}

Given a sequence $\{a(i)\}_{i\ge0}$, let $a(\ast)\coloneqq\sum_{i\ge0}a(i)$. 
Applying Chu-Vandemonde identities to Theorem~\ref{prop:joint}, we have the following result.
\begin{cor}\label{cor:sums}
For nonnegative integers \( h \), \( \ell \), \( m \), and \( n \), we have
\begin{align*}
a_n(h,\ell,\ast)&=\dfrac{1}{n+1}\binom{n+1}{\ell+1}\binom{n-\ell}{h}\binom{n+1}{2\ell+h-n}, \\
a_n(h,\ast,m)&=\frac{m+1}{n+1}\binom{n+1}{h}\sum_{i\ge0}\binom{n-h+1}{i+1}\binom{n-m-1}{2n-h-2i-1},  \\
a_n(\ast,\ell,m)&=\dfrac{m+1}{n+1}\binom{n+1}{\ell+1}\binom{2n-\ell-m-1}{2n-2\ell-1}, \\
a_n(h,\ast,\ast)&=\dfrac{1}{n+1}\binom{n+1}{h}\sum_{i\ge0}\binom{n-h+1}{i}\binom{n+1}{2i+h+1},\\ 
a_n(\ast,\ell,\ast)&=\dfrac{1}{n+1}\binom{n+1}{\ell+1}\binom{2n-\ell+1}{\ell},\\ 
a_n(\ast,\ast,m)&=\frac{m+1}{n+1}\sum_{i\ge0}\binom{n+1}{i}\binom{n-m+i-1}{2i-1},\\ 
a_n(\ast,\ast,\ast)&=\dfrac{1}{n+1}\sum_{i\ge0}\binom{n+1}{i+1}\binom{2n-i+1}{i}. 
\end{align*}
\end{cor}

Here are information of OEIS entries on Corollary~\ref{cor:sums}:
\begin{itemize}
\item $a_n(\ast,\ast,\ast)$: A106228 
\item $a_n(\ast,\ast,0)$: A109081
\item $a_n(\ast,0,\ast)$: A000012 (constant $1$)
\item $a_n(0,\ast,\ast)$: A036765
\item $a_n(\ast,\ast,m)$: A185967
\end{itemize}

Table~\ref{table:a} lists some values of 
$a_n(h,\ast,\ast)$ and $a_n(\ast,\ell,\ast)$, 
which are not available in OEIS.

\begin{table}[t]
\centering
\begin{tabular}{c|cccccccccccccccc}
\noalign{\smallskip}\noalign{\smallskip}
\(n \backslash h\)&& 0 && 1 && 2 && 3 && 4 && 5  \\
\hline
0 && 1 && 0 && 0 && 0 && 0  && 0 \\
1 && 1 && 1 && 0 && 0 && 0 && 0 \\
2 && 2 && 3 && 1 && 0 && 0  && 0 \\
3 && 5 && 9 && 6 && 1 && 0 && 0\\
4 && 13 && 30 && 26 && 10 && 1 && 0 \\
5 && 36 && 100 && 110 && 60 && 15 && 1 \\
\end{tabular}
\qquad 
\begin{tabular}{c|cccccccccccccccc}
\noalign{\smallskip}\noalign{\smallskip}
\(n \backslash \ell\)&& 0 && 1 && 2 && 3 && 4 && 5  \\
\hline
0 && 1 && 0 && 0 && 0 && 0  && 0 \\
1 && 1 && 1 && 0 && 0 && 0 && 0 \\
2 && 1 && 4 && 1 && 0 && 0  && 0 \\
3 && 1 && 9 && 10 && 1 && 0 && 0\\
4 && 1 && 16 && 42 && 20 && 1 && 0 \\
5 && 1 && 25 && 120 && 140 && 35 && 1
\end{tabular}\\[\baselineskip]
\caption{The numbers {$a_n(h,\ast,\ast)$ (left) and $a_n(\ast,\ell,\ast)$ (right).}}\label{table:a}
\end{table}
\begin{rem}
By Corollary~\ref{cor:sums}, all the sets mentioned in Theorem~\ref{thm:main} are of size 
$$
a_n(\ast,\ast,\ast)=\dfrac{1}{n+1}\sum_{i\ge0}\binom{n+1}{i+1}\binom{2n-i+1}{i}.
$$
\end{rem}



We close this section with an example of integrating all the bijections we constructed. 
Let us denote \( \psi_{\rm A} \coloneqq \phi_{\rm A}^{-1} \) for each \( {\rm A } \in \set{\rm P,B,S,I,J,T} \). 

\begin{example}\label{ex:recover}
Let \( Q=s_1s_2\dots s_{15} \) be the \( F \)-path with the steps, \( s_7=(3,-1), ~s_{11}=(1,1), ~s_{12}=(2,1), ~s_{15}=(1,-1)\), and \( s_i=(0,1) \) otherwise. We have
\begin{align*} 
\psi_{\rm P}(Q)&= h^2u\,h\,u\,h^2 d^2 h^2 u^2 d\, h \, d \, u \, h \, u \,d^2,\\ 
\psi_{\rm B}(Q)&=u^7 \dr^2 \db^3 u^4\db \,u\, \db^2 u^3 \dr^2 \db\, u\, \dr^5,\\
\psi_{\rm S}(Q)&={1~2~5~8~3~4~6~7~9~16~12~13~11~10~14~15},\\
\psi_{\rm I}(Q)&=(0,0,0,0,0,3,3,3,3,4,6,7,6,6,0,0),\\ 
\psi_{\rm J}(Q)&=0\,0\,0\,0\,0\,1\,0\,0\,4\,4\,5\,9\,9\,9\,0\,0,
\end{align*}
and \( \psi_{\rm T}(Q)=T \) is depicted in Figure~\ref{fig:reverseT}. 
Figures~\ref{fig:reverseT} and \ref{fig:reverse} show the process for each bijections \( \psi_{\rm T}\), \( \psi_{\rm P}\), \(\psi_{\rm S}\), and \( \psi_{\rm I} \). Here, 
\( Q_{(i)}\coloneqq s_1 s_2 \dots s_i \) for each \( i \in [15] \).
\end{example}


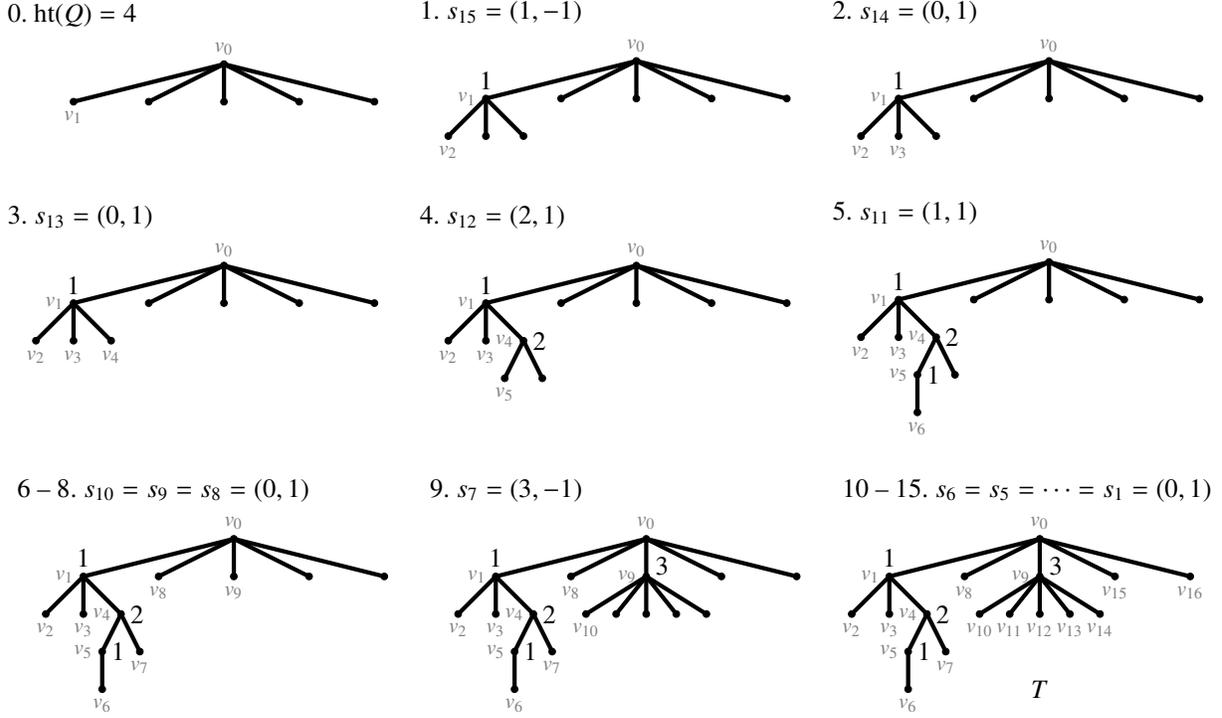
\begin{figure}[t]
  \centering
  {\small
  \begin{tikzpicture}[scale=0.5]
  \coordinate (0) at (0,0);
  \coordinate (1) at (-4,-1);
  \coordinate (2) at (-5,-2);
  \coordinate (3) at (-4,-2);
  \coordinate (4) at (-3,-2);
  \coordinate (5) at (-3.5,-3);
  \coordinate (6) at (-3.5,-4);
  \coordinate (7) at (-2.5,-3);
  \coordinate (8) at (-2,-1);
  \coordinate (9) at (0,-1);
  \coordinate (10) at (-1.6,-2);
  \coordinate (11) at (-0.8,-2);
  \coordinate (12) at (0,-2);
  \coordinate (13) at (0.8,-2);
  \coordinate (14) at (1.6,-2);
  \coordinate (15) at (2,-1);
  \coordinate (16) at (4,-1);
  
  \foreach \i in {0,1,8,9,15,16}
  \filldraw (\i) circle (2.5pt);
  
  \node[above] at (0) {\scriptsize \color{gray} \( v_{0} \)};
  
  \foreach \i in {1}
  \node[below] at (\i) {\scriptsize \color{gray} \( v_{\i} \)};
  
  \foreach \i in {1,8,9,15,16}
  \draw[ultra thick]
  (0) -- (\i);
  
  \node[right] at (-6,1.3) {0. \(\height(Q)=4\)};
  \node at (0,-2.5) {};
  
  \end{tikzpicture}
  \quad 
  \begin{tikzpicture}[scale=0.5]
  \coordinate (0) at (0,0);
  \coordinate (1) at (-4,-1);
  \coordinate (2) at (-5,-2);
  \coordinate (3) at (-4,-2);
  \coordinate (4) at (-3,-2);
  \coordinate (5) at (-3.5,-3);
  \coordinate (6) at (-3.5,-4);
  \coordinate (7) at (-2.5,-3);
  \coordinate (8) at (-2,-1);
  \coordinate (9) at (0,-1);
  \coordinate (10) at (-1.6,-2);
  \coordinate (11) at (-0.8,-2);
  \coordinate (12) at (0,-2);
  \coordinate (13) at (0.8,-2);
  \coordinate (14) at (1.6,-2);
  \coordinate (15) at (2,-1);
  \coordinate (16) at (4,-1);
  
  \foreach \i in {0,1,2,3,4,8,9,15,16}
  \filldraw (\i) circle (2.5pt);
  
  \node[above] at (0) {\scriptsize \color{gray} \( v_{0} \)};
  
  \foreach \i in {1}
  \node[left] at (\i) {\scriptsize \color{gray} \( v_{\i} \)};
  
  \foreach \i in {2}
  \node[below] at (\i) {\scriptsize \color{gray} \( v_{\i} \)};
  
  \foreach \i in {1,8,9,15,16}
  \draw[ultra thick]
  (0) -- (\i);
  
  \foreach \i in {2,3,4}
  \draw[ultra thick]
  (1) -- (\i);
  
  \node[above] at (1) {\( 1 \)};
  
  \node[right] at (-6,1.3) {1. \( s_{15}=(1,-1) \)};
  \node at (0,-2.5) {};
  
  \end{tikzpicture}
  \quad 
  \begin{tikzpicture}[scale=0.5]
  \coordinate (0) at (0,0);
  \coordinate (1) at (-4,-1);
  \coordinate (2) at (-5,-2);
  \coordinate (3) at (-4,-2);
  \coordinate (4) at (-3,-2);
  \coordinate (5) at (-3.5,-3);
  \coordinate (6) at (-3.5,-4);
  \coordinate (7) at (-2.5,-3);
  \coordinate (8) at (-2,-1);
  \coordinate (9) at (0,-1);
  \coordinate (10) at (-1.6,-2);
  \coordinate (11) at (-0.8,-2);
  \coordinate (12) at (0,-2);
  \coordinate (13) at (0.8,-2);
  \coordinate (14) at (1.6,-2);
  \coordinate (15) at (2,-1);
  \coordinate (16) at (4,-1);
  
  \foreach \i in {0,1,2,3,4,8,9,15,16}
  \filldraw (\i) circle (2.5pt);
  
  \node[above] at (0) {\scriptsize \color{gray} \( v_{0} \)};
  
  \foreach \i in {1}
  \node[left] at (\i) {\scriptsize \color{gray} \( v_{\i} \)};
  
  \foreach \i in {2,3}
  \node[below] at (\i) {\scriptsize \color{gray} \( v_{\i} \)};
  
  \foreach \i in {1,8,9,15,16}
  \draw[ultra thick]
  (0) -- (\i);
  
  \foreach \i in {2,3,4}
  \draw[ultra thick]
  (1) -- (\i);
  
  \node[above] at (1) {\( 1 \)};
  
  \node[right] at (-6,1.3) {2. \( s_{14}=(0,1) \)};
  \node at (0,-2.5) {};
  
  \end{tikzpicture}\\[2ex]
  
  \begin{tikzpicture}[scale=0.5]
  \coordinate (0) at (0,0);
  \coordinate (1) at (-4,-1);
  \coordinate (2) at (-5,-2);
  \coordinate (3) at (-4,-2);
  \coordinate (4) at (-3,-2);
  \coordinate (5) at (-3.5,-3);
  \coordinate (6) at (-3.5,-4);
  \coordinate (7) at (-2.5,-3);
  \coordinate (8) at (-2,-1);
  \coordinate (9) at (0,-1);
  \coordinate (10) at (-1.6,-2);
  \coordinate (11) at (-0.8,-2);
  \coordinate (12) at (0,-2);
  \coordinate (13) at (0.8,-2);
  \coordinate (14) at (1.6,-2);
  \coordinate (15) at (2,-1);
  \coordinate (16) at (4,-1);
  
  \foreach \i in {0,1,2,3,4,8,9,15,16}
  \filldraw (\i) circle (2.5pt);
  
  \node[above] at (0) {\scriptsize \color{gray} \( v_{0} \)};
  
  \foreach \i in {1}
  \node[left] at (\i) {\scriptsize \color{gray} \( v_{\i} \)};
  
  \foreach \i in {2,3,4}
  \node[below] at (\i) {\scriptsize \color{gray} \( v_{\i} \)};
  
  \foreach \i in {1,8,9,15,16}
  \draw[ultra thick]
  (0) -- (\i);
  
  \foreach \i in {2,3,4}
  \draw[ultra thick]
  (1) -- (\i);
  
  \node[above] at (1) {\( 1 \)};
  
  \node[right] at (-6,1.3) {3. \( s_{13}=(0,1) \)};
  \node at (0,-4.5) {};
  
  \end{tikzpicture}
  \quad 
  \begin{tikzpicture}[scale=0.5]
  \coordinate (0) at (0,0);
  \coordinate (1) at (-4,-1);
  \coordinate (2) at (-5,-2);
  \coordinate (3) at (-4,-2);
  \coordinate (4) at (-3,-2);
  \coordinate (5) at (-3.5,-3);
  \coordinate (6) at (-3.5,-4);
  \coordinate (7) at (-2.5,-3);
  \coordinate (8) at (-2,-1);
  \coordinate (9) at (0,-1);
  \coordinate (10) at (-1.6,-2);
  \coordinate (11) at (-0.8,-2);
  \coordinate (12) at (0,-2);
  \coordinate (13) at (0.8,-2);
  \coordinate (14) at (1.6,-2);
  \coordinate (15) at (2,-1);
  \coordinate (16) at (4,-1);
  
  \foreach \i in {0,1,2,3,4,5,7,8,9,15,16}
  \filldraw (\i) circle (2.5pt);
  
  \node[above] at (0) {\scriptsize \color{gray} \( v_{0} \)};
  
  \foreach \i in {1,4}
  \node[left] at (\i) {\scriptsize \color{gray} \( v_{\i} \)};
  
  \foreach \i in {2,3,5}
  \node[below] at (\i) {\scriptsize \color{gray} \( v_{\i} \)};
  
  \foreach \i in {1,8,9,15,16}
  \draw[ultra thick]
  (0) -- (\i);
  
  \foreach \i in {2,3,4}
  \draw[ultra thick]
  (1) -- (\i);
  
  \foreach \i in {5,7}
  \draw[ultra thick]
  (4) -- (\i);
  
  \node[above] at (1) {\( 1 \)};
  \node[right] at (4) {\( 2 \)};
  
  \node[right] at (-6,1.3) {4. \( s_{12}=(2,1) \)};
  \node at (0,-4.5) {};
  
  \end{tikzpicture}
  \quad 
  \begin{tikzpicture}[scale=0.5]
  \coordinate (0) at (0,0);
  \coordinate (1) at (-4,-1);
  \coordinate (2) at (-5,-2);
  \coordinate (3) at (-4,-2);
  \coordinate (4) at (-3,-2);
  \coordinate (5) at (-3.5,-3);
  \coordinate (6) at (-3.5,-4);
  \coordinate (7) at (-2.5,-3);
  \coordinate (8) at (-2,-1);
  \coordinate (9) at (0,-1);
  \coordinate (10) at (-1.6,-2);
  \coordinate (11) at (-0.8,-2);
  \coordinate (12) at (0,-2);
  \coordinate (13) at (0.8,-2);
  \coordinate (14) at (1.6,-2);
  \coordinate (15) at (2,-1);
  \coordinate (16) at (4,-1);
  
  \foreach \i in {0,1,2,3,4,5,6,7,8,9,15,16}
  \filldraw (\i) circle (2.5pt);
  
  \node[above] at (0) {\scriptsize \color{gray} \( v_{0} \)};
  
  \foreach \i in {1,4,5}
  \node[left] at (\i) {\scriptsize \color{gray} \( v_{\i} \)};
  
  \foreach \i in {2,3,6}
  \node[below] at (\i) {\scriptsize \color{gray} \( v_{\i} \)};
  
  \foreach \i in {1,8,9,15,16}
  \draw[ultra thick]
  (0) -- (\i);
  
  \foreach \i in {2,3,4}
  \draw[ultra thick]
  (1) -- (\i);
  
  \foreach \i in {5,7}
  \draw[ultra thick]
  (4) -- (\i);
  
  \draw[ultra thick] (5) -- (6);
  
  \node[above] at (1) {\( 1 \)};
  \node[right] at (4) {\( 2 \)};
  \node[right] at (5) {\( 1 \)};

  \node[right] at (-6,1.3) {5. \( s_{11}=(1,1) \)};
  \node at (0,-4.5) {};
  
  \end{tikzpicture}\\[2ex]
   \quad 
  \begin{tikzpicture}[scale=0.5]
  \coordinate (0) at (0,0);
  \coordinate (1) at (-4,-1);
  \coordinate (2) at (-5,-2);
  \coordinate (3) at (-4,-2);
  \coordinate (4) at (-3,-2);
  \coordinate (5) at (-3.5,-3);
  \coordinate (6) at (-3.5,-4);
  \coordinate (7) at (-2.5,-3);
  \coordinate (8) at (-2,-1);
  \coordinate (9) at (0,-1);
  \coordinate (10) at (-1.6,-2);
  \coordinate (11) at (-0.8,-2);
  \coordinate (12) at (0,-2);
  \coordinate (13) at (0.8,-2);
  \coordinate (14) at (1.6,-2);
  \coordinate (15) at (2,-1);
  \coordinate (16) at (4,-1);
  
  \foreach \i in {0,1,2,3,4,5,6,7,8,9,15,16}
  \filldraw (\i) circle (2.5pt);
  
  \node[above] at (0) {\scriptsize \color{gray} \( v_{0} \)};
  
  \foreach \i in {1,4,5}
  \node[left] at (\i) {\scriptsize \color{gray} \( v_{\i} \)};
  
  \foreach \i in {2,3,6,7,8,9}
  \node[below] at (\i) {\scriptsize \color{gray} \( v_{\i} \)};
  
  \foreach \i in {1,8,9,15,16}
  \draw[ultra thick]
  (0) -- (\i);
  
  \foreach \i in {2,3,4}
  \draw[ultra thick]
  (1) -- (\i);
  
  \foreach \i in {5,7}
  \draw[ultra thick]
  (4) -- (\i);
  
  \draw[ultra thick] (5) -- (6);
  
  \node[above] at (1) {\( 1 \)};
  \node[right] at (4) {\( 2 \)};
  \node[right] at (5) {\( 1 \)};
  
  \node[right] at (-6,1.3) {6 -- 8. \( s_{10}=s_{9}=s_{8}=(0,1) \)};
  \node at (0,-4.5) {};
  
  \end{tikzpicture}
  \quad 
  \begin{tikzpicture}[scale=0.5]
  \coordinate (0) at (0,0);
  \coordinate (1) at (-4,-1);
  \coordinate (2) at (-5,-2);
  \coordinate (3) at (-4,-2);
  \coordinate (4) at (-3,-2);
  \coordinate (5) at (-3.5,-3);
  \coordinate (6) at (-3.5,-4);
  \coordinate (7) at (-2.5,-3);
  \coordinate (8) at (-2,-1);
  \coordinate (9) at (0,-1);
  \coordinate (10) at (-1.6,-2);
  \coordinate (11) at (-0.8,-2);
  \coordinate (12) at (0,-2);
  \coordinate (13) at (0.8,-2);
  \coordinate (14) at (1.6,-2);
  \coordinate (15) at (2,-1);
  \coordinate (16) at (4,-1);
  
  \foreach \i in {0,1,2,...,16}
  \filldraw (\i) circle (2.5pt);
  
  \node[above] at (0) {\scriptsize \color{gray} \( v_{0} \)};
  
  \foreach \i in {1,4,5,9}
  \node[left] at (\i) {\scriptsize \color{gray} \( v_{\i} \)};
  
  \foreach \i in {2,3,6,7,8,10}
  \node[below] at (\i) {\scriptsize \color{gray} \( v_{\i} \)};
  
  \foreach \i in {1,8,9,15,16}
  \draw[ultra thick]
  (0) -- (\i);
  
  \foreach \i in {2,3,4}
  \draw[ultra thick]
  (1) -- (\i);
  
  \foreach \i in {5,7}
  \draw[ultra thick]
  (4) -- (\i);
  
  \draw[ultra thick] (5) -- (6);
  
  \foreach \i in {10,11,...,14}
  \draw[ultra thick]
  (9) -- (\i);
  
  \node[above] at (1) {\( 1 \)};
  \node[right] at (4) {\( 2 \)};
  \node[right] at (5) {\( 1 \)};
  \node[above=4pt, right] at (9) {\( 3 \)};
  
  \node[right] at (-6,1.3) {9. \( s_{7}=(3,-1) \)};
  \node at (0,-4.5) {};
  
  \end{tikzpicture}
  \quad
  \begin{tikzpicture}[scale=0.5]
  \coordinate (0) at (0,0);
  \coordinate (1) at (-4,-1);
  \coordinate (2) at (-5,-2);
  \coordinate (3) at (-4,-2);
  \coordinate (4) at (-3,-2);
  \coordinate (5) at (-3.5,-3);
  \coordinate (6) at (-3.5,-4);
  \coordinate (7) at (-2.5,-3);
  \coordinate (8) at (-2,-1);
  \coordinate (9) at (0,-1);
  \coordinate (10) at (-1.6,-2);
  \coordinate (11) at (-0.8,-2);
  \coordinate (12) at (0,-2);
  \coordinate (13) at (0.8,-2);
  \coordinate (14) at (1.6,-2);
  \coordinate (15) at (2,-1);
  \coordinate (16) at (4,-1);
  
  \foreach \i in {0,1,2,...,16}
  \filldraw (\i) circle (2.5pt);
  
  \node[above] at (0) {\scriptsize \color{gray} \( v_{0} \)};
  
  \foreach \i in {1,4,5,9}
  \node[left] at (\i) {\scriptsize \color{gray} \( v_{\i} \)};
  
  \foreach \i in {2,3,6,7,8,10,11,...,16}
  \node[below] at (\i) {\scriptsize \color{gray} \( v_{\i} \)};
  
  \foreach \i in {1,8,9,15,16}
  \draw[ultra thick]
  (0) -- (\i);
  
  \foreach \i in {2,3,4}
  \draw[ultra thick]
  (1) -- (\i);
  
  \foreach \i in {5,7}
  \draw[ultra thick]
  (4) -- (\i);
  
  \draw[ultra thick] (5) -- (6);
  
  \foreach \i in {10,11,...,14}
  \draw[ultra thick]
  (9) -- (\i);
  
  \node[above] at (1) {\( 1 \)};
  \node[right] at (4) {\( 2 \)};
  \node[right] at (5) {\( 1 \)};
  \node[above=4pt, right] at (9) {\( 3 \)};
  
  \node at (0,-4) {\( T \)};
  
  \node[right] at (-5.5,1.3) {10 -- 15. \( s_6=s_5=\cdots=s_1=(0,1)\)};
  \end{tikzpicture}
  }
  \caption{The reverse process of the bijection \( \phi_{\rm T} \).}\label{fig:reverseT}
  
\end{figure} 


\begin{figure}[ht]
  \centering
  \scriptsize{
  
  \begin{tikzpicture}[scale=0.5]
  \foreach \i in {0,1,2,...,30}
  \foreach \j in {0,1,2}
  \filldraw[fill=gray!70, color=gray!70] (0.5*\i,0.5*\j) circle (2.5pt);
  
  \draw (0,0) -- (15,0);
  \draw (0,0) -- (0,1);
  
  \coordinate (0) at (0,0);
  \coordinate (1) at (2/2,0/2);
  \coordinate (2) at (4/2,0/2);
  \coordinate (3) at (6/2,0/2);
  \coordinate (4) at (8/2,0/2);
  \coordinate (5) at (10/2,0/2);
  \coordinate (6) at (12/2,0/2);
  
  \foreach \i in {0,1,2,3,4,5,6}
  \filldraw (\i) circle (2.5pt);
  \draw[ultra thick]
  (0) --(1) --(2) --(3) --(4) --(5);
  \draw [ultra thick, color=black]
  (5) --(6);
  
  \foreach \i in {0}
  \filldraw (\i) circle (2.5pt);
  \draw[ultra thick]
  ;
  
  \node at (7,-0.5) {\( \psi_{\rm P}(Q_{(6)}) \)};
  \node[right] at (15.5,0.8) {\( \psi_{\rm I}(Q_{(6)})=(0,0,0,0,0,0,0) \)}; 
  \node[right] at (15.5,-0.2) {\( \psi_{\rm S}(Q_{(6)})=1\,2\,3\,4\,5\,6\,{7}\)}; 
  \node[right] at  (27.5,0.8) {\( Q_{(6)}=Q_{(5)}(0,1)\)};
  \node[right] at  (27.5,-0.2) {\( \height(Q_{(6)})=6 \)};
  \node at (32.5,0.25) {};
  \end{tikzpicture}\\[0.4ex]
  
  \begin{tikzpicture}[scale=0.5]
  \foreach \i in {0,1,2,...,30}
  \foreach \j in {0,1,2}
  \filldraw[fill=gray!70, color=gray!70] (0.5*\i,0.5*\j) circle (2.5pt);
  
  \draw (0,0) -- (15,0);
  \draw (0,0) -- (0,1);
  
  \coordinate (0) at (0,0);
  \coordinate (1) at (2/2,0/2);
  \coordinate (2) at (4/2,0/2);
  \coordinate (3) at (5/2,1/2);
  \coordinate (4) at (7/2,1/2);
  \coordinate (5) at (8/2,2/2);
  \coordinate (6) at (10/2,2/2);
  \coordinate (7) at (12/2,2/2);
  \coordinate (8) at (13/2,1/2);
  \coordinate (9) at (14/2,0/2);
  
  \foreach \i in {0,1,2,3,4,5,6,7,8,9}
  \filldraw (\i) circle (2.5pt);
  \draw[ultra thick]
  (0) --(1);
  \draw [ultra thick, color=black]
  (1) --(2) --(3) --(4) --(5) --(6) --(7) --(8) --(9);
  
  \foreach \i in {0}
  \filldraw (\i) circle (2.5pt);
  \draw[ultra thick]
  ;
  
  \node at (7,-0.5) {\( \psi_{\rm P}(Q_{(7)}) \)};
  \node[right] at (15.5,0.8) {\( \psi_{\rm I}(Q_{(7)})=(0,0,0,0,0,3,0,0) \)}; 
  \node[right] at (15.5,-0.2) {\( \psi_{\rm S}(Q_{(7)})=1\,2\,{5\,8\,3\,4\,6\,7}\)}; 
  \node[right] at  (27.5,0.8) {\( Q_{(7)}=Q_{(6)}(3,-1)\)};
  \node[right] at  (27.5,-0.2) {\( \height(Q_{(7)})=2 \)};
  \node at (32.5,0.25) {};
  \end{tikzpicture}\\[0.4ex]
  
  \begin{tikzpicture}[scale=0.5]
  \foreach \i in {0,1,2,...,30}
  \foreach \j in {0,1,2}
  \filldraw[fill=gray!70, color=gray!70] (0.5*\i,0.5*\j) circle (2.5pt);
  
  \draw (0,0) -- (15,0);
  \draw (0,0) -- (0,1);
  
  \coordinate (0) at (0,0);
  \coordinate (1) at (2/2,0/2);
  \coordinate (2) at (4/2,0/2);
  \coordinate (3) at (5/2,1/2);
  \coordinate (4) at (7/2,1/2);
  \coordinate (5) at (8/2,2/2);
  \coordinate (6) at (10/2,2/2);
  \coordinate (7) at (12/2,2/2);
  \coordinate (8) at (13/2,1/2);
  \coordinate (9) at (14/2,0/2);
  \coordinate (10) at (16/2,0/2);
  
  \foreach \i in {0,1,2,3,4,5,6,7,8,9,10}
  \filldraw (\i) circle (2.5pt);
  \draw[ultra thick]
  (0) --(1) --(2) --(3) --(4) --(5) --(6) --(7) --(8) --(9);
  \draw [ultra thick, color=black]
  (9) --(10);

  \foreach \i in {0}
  \filldraw (\i) circle (2.5pt);
  \draw[ultra thick]
  ;
  
  \node at (7,-0.5) {\( \psi_{\rm P}(Q_{(8)}) \)};
  \node[right] at (15.5,0.8) {\( \psi_{\rm I}(Q_{(8)})=(0,0,0,0,0,3,3,0,0) \)}; 
  \node[right] at (15.5,-0.2) {\( \psi_{\rm S}(Q_{(8)})=1\,2\,5\,8\,3\,4\,6\,7\,{9}\)}; 
  \node[right] at  (27.5,0.8) {\( Q_{(8)}=Q_{(7)}(0,1)\)};
  \node[right] at  (27.5,-0.2) {\( \height(Q_{(8)})=3 \)};
  \node at (32.5,0.25) {};
  \end{tikzpicture}\\[0.4ex]
  
  \begin{tikzpicture}[scale=0.5]
  \foreach \i in {0,1,2,...,30}
  \foreach \j in {0,1,2}
  \filldraw[fill=gray!70, color=gray!70] (0.5*\i,0.5*\j) circle (2.5pt);
  
  \draw (0,0) -- (15,0);
  \draw (0,0) -- (0,1);
  
  \coordinate (0) at (0,0);
  \coordinate (1) at (2/2,0/2);
  \coordinate (2) at (4/2,0/2);
  \coordinate (3) at (5/2,1/2);
  \coordinate (4) at (7/2,1/2);
  \coordinate (5) at (8/2,2/2);
  \coordinate (6) at (10/2,2/2);
  \coordinate (7) at (12/2,2/2);
  \coordinate (8) at (13/2,1/2);
  \coordinate (9) at (14/2,0/2);
  \coordinate (10) at (16/2,0/2);
  \coordinate (11) at (18/2,0/2);
  
  \foreach \i in {0,1,2,3,4,5,6,7,8,9,10,11}
  \filldraw (\i) circle (2.5pt);
  \draw[ultra thick]
  (0) --(1) --(2) --(3) --(4) --(5) --(6) --(7) --(8) --(9) --(10);
  \draw [ultra thick, color=black]
  (10) --(11);
  \foreach \i in {0}
  \filldraw (\i) circle (2.5pt);
  \draw[ultra thick]
  ;
  
  \node at (7,-0.5) {\( \psi_{\rm P}(Q_{(9)}) \)};
  \node[right] at (15.5,0.8) {\( \psi_{\rm I}(Q_{(9)})=(0,0,0,0,0,3,3,3,0,0) \)}; 
  \node[right] at (15.5,-0.2) {\( \psi_{\rm S}(Q_{(9)})=1\,2\,5\,8\,3\,4\,6\,7\,9\,{10} \)}; 
  \node[right] at  (27.5,0.8) {\( Q_{(9)}=Q_{(8)}(0,1)\)};
  \node[right] at  (27.5,-0.2) {\( \height(Q_{(9)})=4 \)};
  \node at (32.5,0.25) {};
  \end{tikzpicture}\\[0.4ex]
  
  \begin{tikzpicture}[scale=0.5]
  \foreach \i in {0,1,2,...,30}
  \foreach \j in {0,1,2}
  \filldraw[fill=gray!70, color=gray!70] (0.5*\i,0.5*\j) circle (2.5pt);
  
  \draw (0,0) -- (15,0);
  \draw (0,0) -- (0,1);
  
  \coordinate (0) at (0,0);
  \coordinate (1) at (2/2,0/2);
  \coordinate (2) at (4/2,0/2);
  \coordinate (3) at (5/2,1/2);
  \coordinate (4) at (7/2,1/2);
  \coordinate (5) at (8/2,2/2);
  \coordinate (6) at (10/2,2/2);
  \coordinate (7) at (12/2,2/2);
  \coordinate (8) at (13/2,1/2);
  \coordinate (9) at (14/2,0/2);
  \coordinate (10) at (16/2,0/2);
  \coordinate (11) at (18/2,0/2);
  \coordinate (12) at (20/2,0/2);
  
  \foreach \i in {0,1,2,3,4,5,6,7,8,9,10,11,12}
  \filldraw (\i) circle (2.5pt);
  \draw[ultra thick]
  (0) --(1) --(2) --(3) --(4) --(5) --(6) --(7) --(8) --(9) --(10) --(11);
  \draw [ultra thick, color=black]
  (11) --(12);
  \foreach \i in {0}
  \filldraw (\i) circle (2.5pt);
  \draw[ultra thick]
  ;
  
  \node at (7,-0.5) {\( \psi_{\rm P}(Q_{(10)}) \)};
  \node[right] at (15.5,0.8) {\( \psi_{\rm I}(Q_{(10)})=(0,0,0,0,0,3,3,3,3,0,0) \)}; 
  \node[right] at (15.5,-0.2) {\( \psi_{\rm S}(Q_{(10)})=1\,2\,5\,8\,3\,4\,6\,7\,9\,10\,{11}\)}; 
  \node[right] at  (27.5,0.8) {\( Q_{(10)}=Q_{(9)}(0,1)\)};
  \node[right] at  (27.5,-0.2) {\( \height(Q_{(10)})=5\)};
  \node at (32.5,0.25) {};
  \end{tikzpicture}\\[0.4ex]
  
  \begin{tikzpicture}[scale=0.5]
  \foreach \i in {0,1,2,...,30}
  \foreach \j in {0,1,2}
  \filldraw[fill=gray!70, color=gray!70] (0.5*\i,0.5*\j) circle (2.5pt);
  
  \draw (0,0) -- (15,0);
  \draw (0,0) -- (0,1);
  
  \coordinate (0) at (0,0);
  \coordinate (1) at (2/2,0/2);
  \coordinate (2) at (4/2,0/2);
  \coordinate (3) at (5/2,1/2);
  \coordinate (4) at (7/2,1/2);
  \coordinate (5) at (8/2,2/2);
  \coordinate (6) at (10/2,2/2);
  \coordinate (7) at (12/2,2/2);
  \coordinate (8) at (13/2,1/2);
  \coordinate (9) at (14/2,0/2);
  \coordinate (10) at (16/2,0/2);
  \coordinate (11) at (18/2,0/2);
  \coordinate (12) at (20/2,0/2);
  \coordinate (13) at (21/2,1/2);
  \coordinate (14) at (22/2,0/2);
  
  \foreach \i in {0,1,2,3,4,5,6,7,8,9,10,11,12,13,14}
  \filldraw (\i) circle (2.5pt);
  \draw[ultra thick]
  (0) --(1) --(2) --(3) --(4) --(5) --(6) --(7) --(8) --(9) --(10) --(11);
  \draw [ultra thick, color=black]
  (11) --(12) --(13) --(14);
  \foreach \i in {0}
  \filldraw (\i) circle (2.5pt);
  \draw[ultra thick]
  ;
  
  \node at (7,-0.5) {\( \psi_{\rm P}(Q_{(11)}) \)};
  \node[right] at (15.5,0.8) {\( \psi_{\rm I}(Q_{(11)})=(0,0,0,0,0,3,3,3,3,4,0,0) \)}; 
  \node[right] at (15.5,-0.2) {\( \psi_{\rm S}(Q_{(11)})= 
  1\,2\,5\,8\,3\,4\,6\,7\,9\,10\,{12\,11}\)}; 
  \node[right] at  (27.5,0.8) {\( Q_{(11)}=Q_{(10)}(1,1)\)};
  \node[right] at  (27.5,-0.2) {\( \height(Q_{(11)})=5 \)};
  \node at (32.5,0.25) {};
  \end{tikzpicture}\\[0.4ex]
  
  \begin{tikzpicture}[scale=0.5]
  \foreach \i in {0,1,2,...,30}
  \foreach \j in {0,1,2}
  \filldraw[fill=gray!70, color=gray!70] (0.5*\i,0.5*\j) circle (2.5pt);
  
  \draw (0,0) -- (15,0);
  \draw (0,0) -- (0,1);
  
  \coordinate (0) at (0,0);
  \coordinate (1) at (2/2,0/2);
  \coordinate (2) at (4/2,0/2);
  \coordinate (3) at (5/2,1/2);
  \coordinate (4) at (7/2,1/2);
  \coordinate (5) at (8/2,2/2);
  \coordinate (6) at (10/2,2/2);
  \coordinate (7) at (12/2,2/2);
  \coordinate (8) at (13/2,1/2);
  \coordinate (9) at (14/2,0/2);
  \coordinate (10) at (16/2,0/2);
  \coordinate (11) at (18/2,0/2);
  \coordinate (12) at (19/2,1/2);
  \coordinate (13) at (20/2,2/2);
  \coordinate (14) at (21/2,1/2);
  \coordinate (15) at (23/2,1/2);
  \coordinate (16) at (24/2,0/2);

  \foreach \i in {0,1,2,3,4,5,6,7,8,9,10,11,12,13,14,15,16}
  \filldraw (\i) circle (2.5pt);
  \draw[ultra thick]
  (0) --(1) --(2) --(3) --(4) --(5) --(6) --(7) --(8) --(9) --(10);
  \draw [ultra thick, color=black]
  (10) --(11) --(12) --(13) --(14) --(15) --(16);
  \foreach \i in {0}
  \filldraw (\i) circle (2.5pt);
  \draw[ultra thick]
  ;
  
  \node at (7,-0.5) {\( \psi_{\rm P}(Q_{(12)}) \)};
  \node[right] at (15.5,0.8) {\( \psi_{\rm I}(Q_{(12)})=(0,0,0,0,0,3,3,3,3,4,6,0,0) \)}; 
  \node[right] at (15.5,-0.2) {\( \psi_{\rm S}(Q_{(12)})
  =1\,2\,5\,8\,3\,4\,6\,7\,9\,{12\,13\,11\,10}\)}; 
  \node[right] at  (27.5,0.8) {\( Q_{(12)}=Q_{(11)}(2,1)\)};
  \node[right] at  (27.5,-0.2) {\( \height(Q_{(12)})=4\)};
  \node at (32.5,0.25) {};
  \end{tikzpicture}\\[0.4ex]
  
  \begin{tikzpicture}[scale=0.5]
  \foreach \i in {0,1,2,...,30}
  \foreach \j in {0,1,2}
  \filldraw[fill=gray!70, color=gray!70] (0.5*\i,0.5*\j) circle (2.5pt);
  
  \draw (0,0) -- (15,0);
  \draw (0,0) -- (0,1);
  
  \coordinate (0) at (0,0);
  \coordinate (1) at (2/2,0/2);
  \coordinate (2) at (4/2,0/2);
  \coordinate (3) at (5/2,1/2);
  \coordinate (4) at (7/2,1/2);
  \coordinate (5) at (8/2,2/2);
  \coordinate (6) at (10/2,2/2);
  \coordinate (7) at (12/2,2/2);
  \coordinate (8) at (13/2,1/2);
  \coordinate (9) at (14/2,0/2);
  \coordinate (10) at (16/2,0/2);
  \coordinate (11) at (18/2,0/2);
  \coordinate (12) at (19/2,1/2);
  \coordinate (13) at (20/2,2/2);
  \coordinate (14) at (21/2,1/2);
  \coordinate (15) at (23/2,1/2);
  \coordinate (16) at (24/2,0/2);
  \coordinate (17) at (26/2,0/2);
  
  \foreach \i in {0,1,2,3,4,5,6,7,8,9,10,11,12,13,14,15,16,17}
  \filldraw (\i) circle (2.5pt);
  \draw[ultra thick]
  (0) --(1) --(2) --(3) --(4) --(5) --(6) --(7) --(8) --(9) --(10) --(11) --(12) --(13) --(14) --(15) --(16);
  \draw [ultra thick, color=black]
  (16) --(17);

  \foreach \i in {0}
  \filldraw (\i) circle (2.5pt);
  \draw[ultra thick]
  ;
  
  \node at (7,-0.5) {\( \psi_{\rm P}(Q_{(13)}) \)};
  \node[right] at (15.5,0.8) {\( \psi_{\rm I}(Q_{(13)})=(0,0,0,0,0,3,3,3,3,4,6,6,0,0) \)}; 
  \node[right] at (15.5,-0.2) {\( \psi_{\rm S}(Q_{(13)})=1\,2\,5\,8\,3\,4\,6\,7\,9\,12\,13\,11\,10\,{14}\)}; 
  \node[right] at  (27.5,0.8) {\( Q_{(13)}=Q_{(12)}(0,1)\)};
  \node[right] at  (27.5,-0.2) {\( \height(Q_{(13)})=5 \)};
  \node at (32.5,0.25) {};
  \end{tikzpicture}\\[0.4ex]
  
  \begin{tikzpicture}[scale=0.5]
  \foreach \i in {0,1,2,...,30}
  \foreach \j in {0,1,2}
  \filldraw[fill=gray!70, color=gray!70] (0.5*\i,0.5*\j) circle (2.5pt);
  
  \draw (0,0) -- (15,0);
  \draw (0,0) -- (0,1);
  
  \coordinate (0) at (0,0);
  \coordinate (1) at (2/2,0/2);
  \coordinate (2) at (4/2,0/2);
  \coordinate (3) at (5/2,1/2);
  \coordinate (4) at (7/2,1/2);
  \coordinate (5) at (8/2,2/2);
  \coordinate (6) at (10/2,2/2);
  \coordinate (7) at (12/2,2/2);
  \coordinate (8) at (13/2,1/2);
  \coordinate (9) at (14/2,0/2);
  \coordinate (10) at (16/2,0/2);
  \coordinate (11) at (18/2,0/2);
  \coordinate (12) at (19/2,1/2);
  \coordinate (13) at (20/2,2/2);
  \coordinate (14) at (21/2,1/2);
  \coordinate (15) at (23/2,1/2);
  \coordinate (16) at (24/2,0/2);
  \coordinate (17) at (26/2,0/2);
  \coordinate (18) at (28/2,0/2);
  
  \foreach \i in {0,1,2,3,4,5,6,7,8,9,10,11,12,13,14,15,16,17,18}
  \filldraw (\i) circle (2.5pt);
  \draw[ultra thick]
  (0) --(1) --(2) --(3) --(4) --(5) --(6) --(7) --(8) --(9) --(10) --(11) --(12) --(13) --(14) --(15) --(16) --(17);
  \draw [ultra thick, color=black]
  (17) --(18);

  \foreach \i in {0}
  \filldraw (\i) circle (2.5pt);
  \draw[ultra thick]
  ;
  
  \node at (7,-0.5) {\( \psi_{\rm P}(Q_{(14)}) \)};
  \node[right] at (15.5,0.8) {\( \psi_{\rm I}(Q_{(14)})=(0,0,0,0,0,3,3,3,3,4,6,6,6,0,0) \)}; 
  \node[right] at (15.5,-0.2) {\( \psi_{\rm S}(Q_{(14)})=1\,2\,5\,8\,3\,4\,6\,7\,9\,12\,13\,11\,10\,14\,{15}\)}; 
  \node[right] at  (27.5,0.8) {\( Q_{(14)}=Q_{(13)}(0,1)\)};
  \node[right] at  (27.5,-0.2) {\( \height(Q_{(14)})=6 \)};
  \node at (32.5,0.25) {};
  \end{tikzpicture}\\[0.4ex]
  
  \begin{tikzpicture}[scale=0.5]
  \foreach \i in {0,1,2,...,30}
  \foreach \j in {0,1,2}
  \filldraw[fill=gray!70, color=gray!70] (0.5*\i,0.5*\j) circle (2.5pt);
  
  \draw (0,0) -- (15,0);
  \draw (0,0) -- (0,1);
  
  \coordinate (0) at (0,0);
  \coordinate (1) at (2/2,0/2);
  \coordinate (2) at (4/2,0/2);
  \coordinate (3) at (5/2,1/2);
  \coordinate (4) at (7/2,1/2);
  \coordinate (5) at (8/2,2/2);
  \coordinate (6) at (10/2,2/2);
  \coordinate (7) at (12/2,2/2);
  \coordinate (8) at (13/2,1/2);
  \coordinate (9) at (14/2,0/2);
  \coordinate (10) at (16/2,0/2);
  \coordinate (11) at (18/2,0/2);
  \coordinate (12) at (19/2,1/2);
  \coordinate (13) at (20/2,2/2);
  \coordinate (14) at (21/2,1/2);
  \coordinate (15) at (23/2,1/2);
  \coordinate (16) at (24/2,0/2);
  \coordinate (17) at (25/2,1/2);
  \coordinate (18) at (27/2,1/2);
  \coordinate (19) at (28/2,2/2);
  \coordinate (20) at (29/2,1/2);
  \coordinate (21) at (30/2,0/2);
  
  \foreach \i in {0,1,2,3,4,5,6,7,8,9,10,11,12,13,14,15,16,17,18,19,20,21}
  \filldraw (\i) circle (2.5pt);
  \draw[ultra thick]
  (0) --(1) --(2) --(3) --(4) --(5) --(6) --(7) --(8) --(9) --(10);
  \draw [ultra thick, color=black]
  (10) --(11) --(12) --(13) --(14) --(15) --(16) --(17) --(18) --(19) --(20) --(21);
  
  \foreach \i in {0}
  \filldraw (\i) circle (2.5pt);
  \draw[ultra thick]
  ;
  
  \node at (7,-0.5) {\( \psi_{\rm P}(Q_{(15)}) \)};
  \node[right] at (15.5,0.8) {\( \psi_{\rm I}(Q_{(15)})=(0,0,0,0,0,3,3,3,3,4,6,7,6,6,0,0) \)}; 
  \node[right] at (15.5,-0.2) {\( \psi_{\rm S}(Q_{(15)})=1\,2\,5\,8\,3\,4\,6\,7\,9\,{16\,12\,13\,11\,10\,14\,15}\)}; 
  \node[right] at (27.5,0.8) {\( Q_{(15)}=Q_{(14)}(1,-1)\)};
  \node[right] at (27.5,-0.2) {\( \height(Q_{(15)})=4 \)};
  \node at (32.5,0.25) {};
  \end{tikzpicture}
  }
  \caption{The reverse process of the bijections \( \phi_{\rm P}\), \(\phi_{\rm I}\), and \( \phi_{\rm S} \).}\label{fig:reverse}
\end{figure}
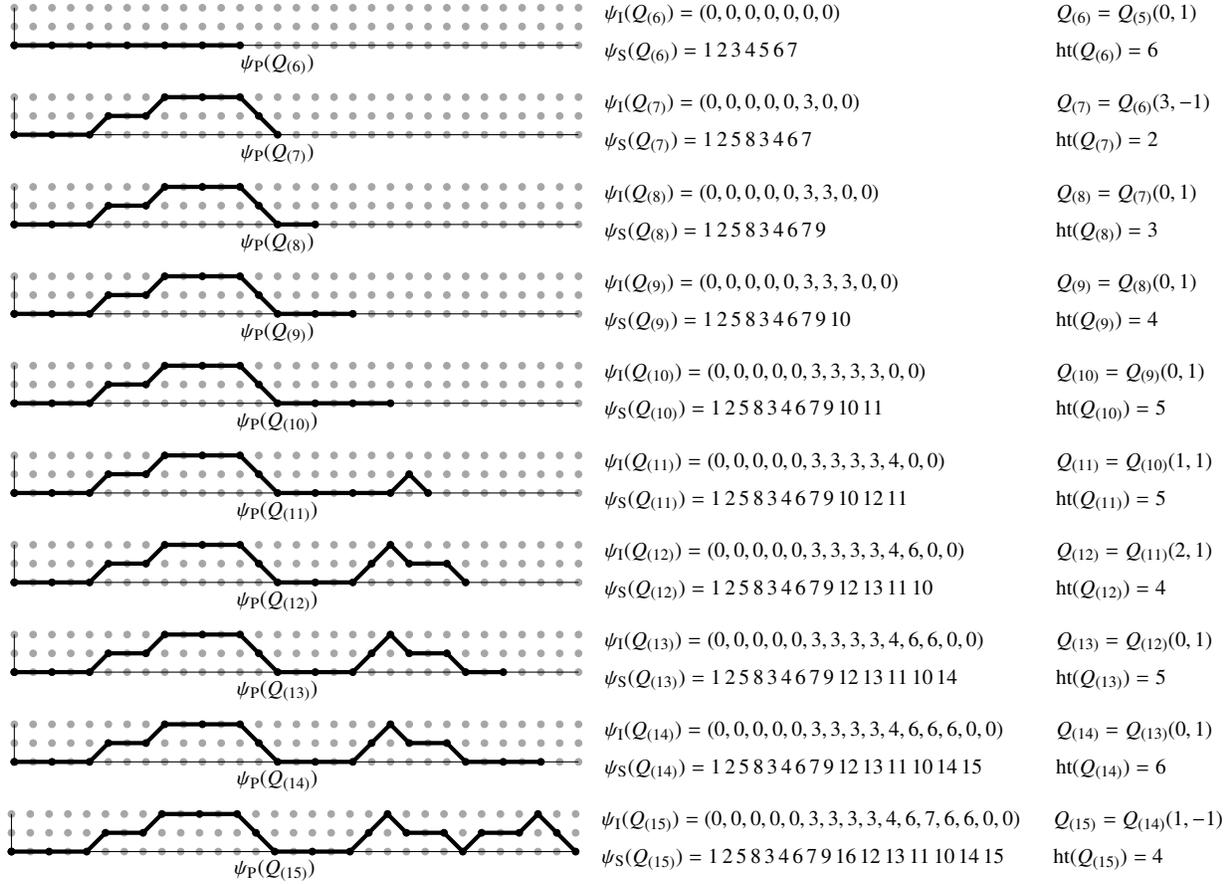 

\section{The direct sums}
\label{sec:D-S}
Recall that we have the unique decomposition {\( Q_1 (0,1) Q_2 \dots (0,1) Q_{r+1} \)} of an \( F \)-path \( Q \) with {\( \height(P)=r \)}, and the unique decomposition {\( P_1 h P_2 \dots h P_{r+1} \)} of a Schr\"{o}der path \( P \) with {\( \comp(P)=r \)}, where each \( P_i \) is a small Schr\"{o}der path. Furthermore, 
a {\( (2341, 2431, 3241) \)}-avoiding permutation \( \pi \) with {\( \block(\pi)=r+1 \)} can be written uniquely as 
{\( \pi_1 \oplus \pi_2 \oplus \cdots \oplus \pi_{r+1} \) }with indecomposable {\( (2341, 2431, 3241) \)}-avoiding permutations \( \pi_i \). Compare the following example with Example~\ref{ex:recover}.

\begin{example}
The {\( (2341, 2431, 3241) \)}-avoiding permutation
\( \pi={1~2~5~8~3~4~6~7~9~16~12~13~11~10~14~15} \)
can be written as \( \pi_1 \oplus \pi_2 \oplus \pi_3 \oplus \pi_4 \oplus \pi_5 \) with indecomposable permutations \( \pi_1=\pi_2=\pi_4=1 \), \( \pi_3={361245} \), and \( \pi_5={7342156} \). 
Let \( R= \phi_{\rm S}(\pi) \) and \( R_{i} = \phi_{\rm S}(\pi_i) \) 
for \( 1\leq i \leq 5 \). See Figure~\ref{fig:factor} for the other related objects which are in bijection with each \( R_i \). One can see that  
\begin{align*} 
&R=R_1 (0,1) R_2 (0,1) R_3 (0,1) R_4 (0,1) R_5,\\
&\psi_{\rm P}(R)=\psi_{\rm P}(R_1) \,h\, \psi_{\rm P}(R_2) \,h\, \psi_{\rm P}(R_3) \,h\, \psi_{\rm P}(R_4) \,h\, \psi_{\rm P}(R_5),\\
&\psi_{\rm B}(R)=\overline{\psi_{\rm B}(R_1)} ~ \overline{\psi_{\rm B}(R_2)} ~ \overline{\psi_{\rm B}(R_3)} ~ \overline{\psi_{\rm B}(R_4)} ~ \overline{\psi_{\rm P}(R_5)} \,\dr^5,\\
&\psi_{\rm S}(R)=\psi_{\rm S}(R_1) \oplus \psi_{\rm S}(R_2) \oplus \psi_{\rm S}(R_3) \oplus \psi_{\rm S}(R_4) \oplus \psi_{\rm S}(R_5),\\
&\psi_{\rm T}(R)=(\overline{\psi_{\rm T}(R_1)}, \overline{\psi_{\rm T}(R_2)}, \overline{\psi_{\rm T}(R_3)}, \overline{\psi_{\rm T}(R_4)}, \overline{\psi_{\rm T}(R_5)}),
\end{align*}
where \( \overline{\psi_{\rm B}(R_i)} \) is the path \( \psi_{\rm B}(R_i) \) without the last descent run 
(see Figure~\ref{fig:LDyck})
and \( \overline{\psi_{\rm T}(R_i)} \) is the weighted ordered tree \( \psi_{\rm T}(R_i) \) without the root vertex and its incident edge.   
\end{example}

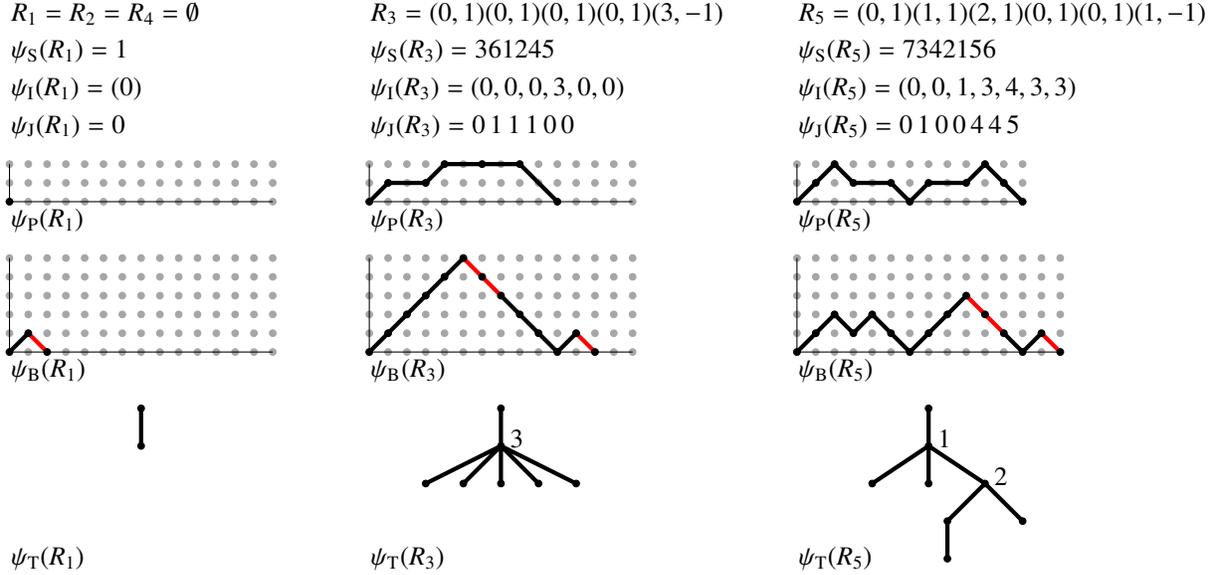
\begin{figure}
\centering
\small{
\begin{tikzpicture}[scale=0.5]
\foreach \i in {0,1,2,...,14}
\foreach \j in {0,1,2}
\filldraw[fill=gray!70, color=gray!70] (0.5*\i,0.5*\j) circle (2.5pt);

\draw (0,0) -- (7,0);
\draw (0,0) -- (0,1);

\coordinate (0) at (0,0);

\foreach \i in {0}
\filldraw (\i) circle (2.5pt);

\node[right] at (-0.2,-0.5) {\( \psi_{\rm P}(R_1) \)};

\node[right] at (-0.2,5) {\( R_1=R_2=R_4=\emptyset\)};
\node[right] at (-0.2,4) {\( \psi_{\rm S}(R_1)=1\)}; 
\node[right] at (-0.2,3) {\( \psi_{\rm I}(R_1)=(0) \)}; 
\node[right] at (-0.2,2) {\( \psi_{\rm J}(R_1)=0 \)}; 

\foreach \i in {0,1,2,...,14}
\foreach \j in {0,1,2,3,4,5}
\filldraw[fill=gray!70, color=gray!70] (0.5*\i,-4+0.5*\j) circle (2.5pt);

\draw (0,-4) -- (7,-4);
\draw (0,-4) -- (0,-1.5);

\coordinate (0) at (0,-4);
\coordinate (1) at (1/2,-4+1/2);
\coordinate (2) at (2/2,-4);

\draw[ultra thick]
(0) -- (1);

\draw[ultra thick, color=red]
(1) -- (2);

\foreach \i in {0,1,2}
\filldraw (\i) circle (2.5pt);

\node[right] at (-0.2,-4.5) {\( \psi_{\rm B}(R_1) \)};

\coordinate (0) at (3.5,-5.5);
\coordinate (1) at (3.5,-6.5);

\draw[ultra thick]
(0) -- (1); 

\foreach \i in {0,1}
\filldraw (\i) circle (2.5pt);

\node[right] at (-0.2,-9.5) {\( \psi_{\rm T}(R_1) \)};

\end{tikzpicture}
\qquad \quad
\begin{tikzpicture}[scale=0.5]
\foreach \i in {0,1,2,...,14}
\foreach \j in {0,1,2}
\filldraw[fill=gray!70, color=gray!70] (0.5*\i,0.5*\j) circle (2.5pt);

\draw (0,0) -- (7,0);
\draw (0,0) -- (0,1);

\coordinate (0) at (0,0);
\coordinate (1) at (1/2,1/2);
\coordinate (2) at (3/2,1/2);
\coordinate (3) at (4/2,2/2);
\coordinate (4) at (6/2,2/2);
\coordinate (5) at (8/2,2/2);
\coordinate (6) at (10/2,0/2);

\foreach \i in {0,1,...,6}
\filldraw (\i) circle (2.5pt);
\draw[ultra thick]
(0) -- (1) -- (2) -- (3) -- (4) -- (5) -- (6);

\node[right] at (-0.2,-0.5) {\( \psi_{\rm P}(R_3) \)};

\node[right] at (-0.2,5) {\( R_3=(0,1)(0,1)(0,1)(0,1)(3,-1) \)};
\node[right] at (-0.2,4) {\( \psi_{\rm S}(R_3)={361245}\)}; 
\node[right] at (-0.2,3) {\( \psi_{\rm I}(R_3)=(0,0,0,3,0,0) \)}; 
\node[right] at (-0.2,2) {\( \psi_{\rm J}(R_3)=0\,1\,1\,1\,0\,0 \)}; 

\foreach \i in {0,1,2,...,14}
\foreach \j in {0,1,2,3,4,5}
\filldraw[fill=gray!70, color=gray!70] (0.5*\i,-4+0.5*\j) circle (2.5pt);

\draw (0,-4) -- (7,-4);
\draw (0,-4) -- (0,-1.5);

\coordinate (0) at (0,-4);
\coordinate (1) at (1/2,-4+1/2);
\coordinate (2) at (2/2,-4+2/2);
\coordinate (3) at (3/2,-4+3/2);
\coordinate (4) at (4/2,-4+4/2);
\coordinate (5) at (5/2,-4+5/2);
\coordinate (6) at (6/2,-4+4/2);
\coordinate (7) at (7/2,-4+3/2);
\coordinate (8) at (8/2,-4+2/2);
\coordinate (9) at (9/2,-4+1/2);
\coordinate (10) at (10/2,-4);
\coordinate (11) at (11/2,-4+1/2);
\coordinate (12) at (12/2,-4);

\draw[ultra thick]
(0) -- (1) -- (2) -- (3) -- (4) -- (5)
(7) -- (8) -- (9) -- (10) -- (11);

\draw[ultra thick, color=red]
(5) -- (6) -- (7)
(11) -- (12);

\foreach \i in {0,1,...,12}
\filldraw (\i) circle (2.5pt);

\node[right] at (-0.2,-4.5) {\( \psi_{\rm B}(R_3) \)};

\coordinate (0) at (3.5,-5.5);
\coordinate (1) at (3.5,-6.5);
\coordinate (2) at (1.5,-7.5);
\coordinate (3) at (2.5,-7.5);
\coordinate (4) at (3.5,-7.5);
\coordinate (5) at (4.5,-7.5);
\coordinate (6) at (5.5,-7.5);

\draw[ultra thick]
(0) -- (1) -- (2)
(1) -- (3) 
(1) -- (4) 
(1) -- (5) 
(1) -- (6); 

\foreach \i in {0,1,...,6}
\filldraw (\i) circle (2.5pt);

\node[right] at (3.5,-6.3) {\( 3 \)}; 
\node[right] at (-0.2,-9.5) {\( \psi_{\rm T}(R_3) \)};
\end{tikzpicture}
\qquad
\begin{tikzpicture}[scale=0.5]
\foreach \i in {0,1,2,...,12}
\foreach \j in {0,1,2}
\filldraw[fill=gray!70, color=gray!70] (0.5*\i,0.5*\j) circle (2.5pt);

\draw (0,0) -- (6,0);
\draw (0,0) -- (0,1);

\coordinate (0) at (0,0);
\coordinate (1) at (1/2,1/2);
\coordinate (2) at (2/2,2/2);
\coordinate (3) at (3/2,1/2);
\coordinate (4) at (5/2,1/2);
\coordinate (5) at (6/2,0/2);
\coordinate (6) at (7/2,1/2);
\coordinate (7) at (9/2,1/2);
\coordinate (8) at (10/2,2/2);
\coordinate (9) at (11/2,1/2);
\coordinate (10) at (12/2,0/2);

\foreach \i in {0,1,...,10}
\filldraw (\i) circle (2.5pt);
\draw[ultra thick]
(0) -- (1) -- (2) -- (3) -- (4) -- (5) -- (6) -- (7) -- (8) -- (9) -- (10);

\node[right] at (-0.2,-0.5) {\( \psi_{\rm P}(R_5) \)};

\node[right] at (-0.2,5) {\( R_5=(0,1)(1,1)(2,1)(0,1)(0,1)(1,-1) \)};
\node[right] at (-0.2,4) {\( \psi_{\rm S}(R_5)={7342156}\)}; 
\node[right] at (-0.2,3) {\( \psi_{\rm I}(R_5)=(0,0,1,3,4,3,3) \)}; 
\node[right] at (-0.2,2) {\( \psi_{\rm J}(R_5)=0\,1\,0\,0\,4\,4\,5 \)}; 

\foreach \i in {0,1,2,...,14}
\foreach \j in {0,1,2,3,4,5}
\filldraw[fill=gray!70, color=gray!70] (0.5*\i,-4+0.5*\j) circle (2.5pt);

\draw (0,-4) -- (7,-4);
\draw (0,-4) -- (0,-1.5);

\coordinate (0) at (0,-4);
\coordinate (1) at (1/2,-4+1/2);
\coordinate (2) at (2/2,-4+2/2);
\coordinate (3) at (3/2,-4+1/2);
\coordinate (4) at (4/2,-4+2/2);
\coordinate (5) at (5/2,-4+1/2);
\coordinate (6) at (6/2,-4);
\coordinate (7) at (7/2,-4+1/2);
\coordinate (8) at (8/2,-4+2/2);
\coordinate (9) at (9/2,-4+3/2);
\coordinate (10) at (10/2,-4+2/2);
\coordinate (11) at (11/2,-4+1/2);
\coordinate (12) at (12/2,-4);
\coordinate (13) at (13/2,-4+1/2);
\coordinate (14) at (14/2,-4);

\draw[ultra thick]
(0) -- (1) -- (2) -- (3) -- (4) -- (5) -- (6) --(7) -- (8) -- (9) 
(11) -- (12) --(13);

\draw[ultra thick, color=red]
(9) -- (10) -- (11)
(13) -- (14);

\foreach \i in {0,1,...,14}
\filldraw (\i) circle (2.5pt);

\node[right] at (-0.2,-4.5) {\( \psi_{\rm B}(R_5) \)};

\coordinate (0) at (3.5,-5.5);
\coordinate (1) at (3.5,-6.5);
\coordinate (2) at (2,-7.5);
\coordinate (3) at (3.5,-7.5);
\coordinate (4) at (5,-7.5);
\coordinate (5) at (4,-8.5);
\coordinate (6) at (4,-9.5);
\coordinate (7) at (6,-8.5);

\draw[ultra thick]
(0) -- (1) -- (2)
(1) -- (3) 
(1) -- (4) -- (5) -- (6)
(4) -- (7); 

\foreach \i in {0,1,...,7}
\filldraw (\i) circle (2.5pt);

\node[right] at (3.5,-6.3) {\( 1 \)}; 
\node[right] at (5,-7.3) {\( 2 \)}; 

\node[right] at (-0.2,-9.5) {\( \psi_{\rm T}(R_5) \)};

\end{tikzpicture}
}
\caption{Three small examples of the bijections.}\label{fig:factor}
\end{figure} 

It is natural to define the direct sums of our paths and weighted ordered trees as follows. 

\begin{defn}
Let \( n\) and \( m \) are nonnegative integers.
\begin{enumerate}
\item For any Schr\"{o}der paths without triple descents \( P_1 \in \schp_n \) and \( P_2 \in \schp_m \), the \emph{direct sum} of \( P_1 \) and \( P_2 \) is the Schr\"{o}der path of semilength \( m+n+1 \) without triple descents that is defined by \( P_1 \oplus P_2 \coloneqq P_1 h P_2 \).
\item For any \( F \)-paths \( Q_1\in \Fp_n \) and \( Q_2 \in \Fp_m \), the \emph{direct sum} of \( Q_1 \) and \( Q_2 \) is the \( F \)-path of length \( m+n+1 \) defined by \( Q_1 \oplus Q_2 \coloneqq Q_1 (0,1) Q_2 \).
\item For any restricted bicolored Dyck paths \( B_1\in \bdp_{n+1} \) and \( B_2 \in \bdp_{m+1} \), the \emph{direct sum} of \( B_1 \) and \( B_2 \) is the restricted bicolored Dyck path of semilength \( m+n+2 \) defined by \( B_1 \oplus B_2 \coloneqq \overline{B_1} B_2 \dr^{\last(B_1)} \), where \( \overline{B_1} \) is the path \( B_1 \) without the last descent run \( \dr^{\last(B_1)} \).
\item For any weighted ordered trees \( T=(T_1,T_2,\dots,T_{n+1}) \) and \( S=(S_1, S_2, \dots, S_{m+1}) \), the \emph{direct sum} of \( T \) and \( S \) is the weighted ordered tree defined by \( T\oplus S=(T_1,\dots,T_{n+1}, S_1,\dots, S_{m+1}) \). 
\end{enumerate}
\end{defn}

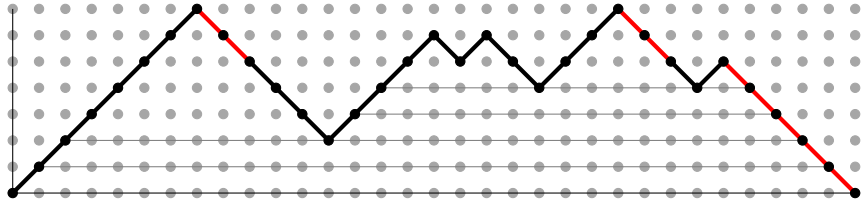
\begin{figure}
\centering
\begin{tikzpicture}[scale=0.7] 
\foreach \i in {0,1,2,...,32}
\foreach \j in {0,1,2,...,7}
\filldraw[fill=gray!70, color=gray!70] (0.5*\i,0.5*\j) circle (2.5pt);

\draw (0,0) -- (16,0);
\draw (0,0) -- (0,3.5);

\foreach \i/\j in {0/0,1/1,2/2,3/3,4/4,5/5,6/6,7/7,8/6,9/5,10/4,11/3,12/2,13/3,14/4,15/5,16/6,17/5,18/6,19/5,20/4,21/5,22/6,23/7,
24/6,25/5,26/4,27/5,28/4,29/3,30/2,31/1,32/0}
\coordinate (\i) at (0.5*\i,0.5*\j);

\draw[color=gray] 
(1) -- (31)
(2) -- (30)
(13) -- (29)
(14) -- (28);

\draw[ultra thick]
(0) -- (1) -- (2) -- (3) -- (4) -- (5) -- (6) -- (7)
(9) -- (10) -- (11) -- (12) -- (13) -- (14) -- (15) -- (16) -- (17) -- (18) -- (19) -- (20) -- (21) -- (22) -- (23) 
(25) -- (26) -- (27);

\draw[ultra thick, color=red]
(7) -- (8) -- (9)
(23) -- (24) -- (25)
(27) -- (28) -- (29) -- (30) -- (31) -- (32);

\foreach \i in {0,1,2,...,32}
\filldraw (\i) circle (2.5pt);

\end{tikzpicture}
\caption{A decomposition of a restricted bicolored Dyck path.}\label{fig:LDyck}
\end{figure} 

Now we consider the direct sums of our pattern avoiding inversion sequences.
For the \( (101,102) \)-avoiding case, let us give a lemma.

\begin{lem}\label{lem:directI}
Let \( e=(e_1, e_2, \dots, e_n) \) and \( f=(f_1,f_2, \dots, f_m) \) be \( (101,102) \)-avoiding inversion sequences
{with \( \maxid(e)=k \) and $\maxid(f)=\ell$}.
{
If 
\[ g=(g_1,g_2,\dots,g_{n+m}) \coloneqq (e_1, \dots, e_k, e_k+f_1, \dots, e_k+f_m, e_{k+1}, \dots, e_n ), \] 
then \( g \) is an \( (101,102) \)-avoiding inversion sequence
with 
\[ \maxid(g)-\max(g)=k-e_k+\ell-f_{\ell}. \] 
Moreover, for \( i>k \), we have 
\[ i-g_i > k - g_k.\]
}
\end{lem}

\begin{proof}
{
Suppose that, given \( 1\leq a<b<c \leq n+m\),
 $g_a g_b g_c$ contained in $g$ forms the pattern $101$ or $102$. 
 Clearly, $g_a > g_b$, $g_b < g_c$ and $g_a \le g_c$.
 Consider the position of $b$.
\begin{itemize}
\item If $b<k$, then $g_a g_b g_k=e_a e_b e_k$ also forms the pattern $101$ or $102$. This yields a contradiction because $e$ contains $e_a e_b e_k$, which forms the pattern $101$ or $102$.

\item If $k\le b \le k+m$, then $g_a > g_b$ yields $a>k$ and  $g_b < g_c$ yields $c<k+m$ 
because $g_i \ge \max(e)$ for $k<i\le k+m$.
Thus, $k<a<b<c<k+m$ and $g_a g_b g_c = (e_k+f_{a-k})(e_k+f_{b-k})(e_k+f_{c-k})$ holds. 
This yields a contradiction because $f$ contains $f_{a-k} f_{b-k} f_{c-k}$, which forms the pattern $101$ or $102$.

\item If $b>k+m$, then $a\le k$ or $a>k+m$ otherwise 
$g_a\ge \min\{g_{k+1},\dots, g_{k+m}\}>
\max\{g_{k+m+1}, \dots, g_{n+m}\}\ge g_c$. 
For $a\le k$, $g_a g_b g_c=e_a e_{b-m} e_{c-m}$ holds. This yields a contradiction because $e$ contains $e_a e_{b-m} e_{c-m}$, which forms the pattern $101$ or $102$.
For $a> k+m$, $g_a g_b g_c=e_{a-m} e_{b-m} e_{c-m}$ holds. This yields a contradiction because $e$ contains $e_{a-m} e_{b-m} e_{c-m}$, which forms the pattern $101$ or $102$.
\end{itemize}
Therefore, $g$ is an $(101, 102)$-avoiding inversion sequence.
Note that \( \maxid(g)=k+\ell \) and \( \max(g)=e_k+ f_{\ell} \). Hence, 
\[
\maxid(g)-\max(g)=k-e_k+\ell-e_{\ell}.
\]
For \( k<i\le k+m \), we have
\[
i-g_i =(k+j) - (e_k+f_j) =(k-e_k)+(j-f_j)>k-e_k=k-g_k.
\]
For $i>k+m$, we have
\[
i-g_i =(k+m+j) - e_{k+j} > (k+m+j)- e_k > k-e_k=k-g_k.
\]
Thus, for \( k<i\le k+m \), $i-g_i > k-g_k$ holds.

}
\end{proof}
 
{Due to Lemma~\ref{lem:directI},} we define the direct sum of \( (101,102) \)-avoiding inversion sequences as follows.

\begin{defn} 
Let \( n \) and \( m \) be positive integers.
For any \( (101,102) \)-avoiding inversion sequences \( e=(e_1, e_2, \dots, e_n) \) and \( f=(f_1,f_2, \dots, f_m) \), 
 the \emph{direct sum} of \( e \) and \( f \) is the \( (101,102) \)-avoiding inversion sequence 
 {$e \oplus f$} of length \( n+m \),  defined by 
\[
(e \oplus f)_i\coloneqq \begin{cases}
e_i & \mbox{ if \( 1\leq i \leq \maxid(e)\),}\\
f_{i-\maxid(e)}+\max(e) & \mbox{ if \( \maxid(e)+1 \leq i \leq \maxid(e)+m \),}\\
e_{i-m} & {\mbox{ if \( \maxid(e)+m+1 \leq i \leq n+m \).}}
\end{cases}
\]
An \( (101,102) \)-avoiding inversion sequence \( e \) is said to be \emph{connected} if \( \maxid(e)=\max(e)+1 \). Otherwise, \( e \) is said to be \emph{disconnected}. 
\end{defn}
{Note that \( \max(e \oplus f)=\max(e)+\max(f) \) and \( \maxid(e \oplus f)=\maxid(e)+\maxid(f) \). }
From Lemma~\ref{lem:directI}, 
for any disconnected \( (101, 102) \)-avoiding inversion sequence \( g \) of {length $N$}, 
we can decompose it as \( e \oplus f \) 
with the connected \( (101, 102) \)-avoiding inversion sequence \( f \coloneqq (0, g_{k+2}-g_{k+1}, \dots, g_{k+m}-g_{k+1} ) \),
where \( k \) is the largest index such that \( {k-g_k}=\maxid(g)-\max(g)-1 \) and {\( m \) is the largest positive integer} 
such that \( g_{k+m}\geq g_{k+1} \). 
{
Also, $e:=(g_1,\dots,g_k, g_{k+m+1},\dots,g_{N})$ is the \( (101, 102) \)-avoiding inversion sequence
with  $\maxid(e)-\max(e)=\maxid(g)-\max(g)-1$.
}
Therefore, an \( (101, 102) \)-avoiding inversion sequence {\( g \) with \( \maxid(g)-\max(g)=r \)} can be written uniquely as a direct sum {of \( r \)} connected \( (101, 102) \)-avoiding inversion sequences. 

\begin{example}
For the running example, we have 
\[
\psi_{\rm I}(R)=\psi_{\rm I}(R_1) \oplus \psi_{\rm I}(R_2) \oplus \psi_{\rm I}(R_3) \oplus \psi_{\rm I}(R_4) \oplus \psi_{\rm I}(R_5),
\]
where each \( \psi_{\rm I}(R_i) \) is a connect \( (101,102) \)-avoiding inversion sequence.
Details are as follows.
{
\begin{align*}
(0,0,0,0,0,3,3,3,3,4,6,7,6,6,0,0)
&=(0,0,0,0,0,3,3,0,0)\oplus (0,0,1,3,4,3,3)\\
&=(0,0,0,0,0,3,0,0) \oplus (0) \oplus (0,0,1,3,4,3,3)\\
&=(0,0)\oplus (0,0,0,3,0,0) \oplus (0) \oplus (0,0,1,3,4,3,3)\\
&=(0) \oplus (0)\oplus (0,0,0,3,0,0) \oplus (0) \oplus (0,0,1,3,4,3,3).
\end{align*}
}
\end{example}

Lastly, we define the direct sum of \( (101,021) \)-avoiding inversion sequences as follows.
{
\begin{lem}\label{lem:directII}
Let \(\je=\je_1 \je_2 \dots \je_n \) and \( \jf=\jf_1 \jf_2 \dots \jf_m \) be \( (101,021) \)-avoiding inversion sequences
with \( \first(\je)=k \) and $\first(\jf)=\ell$.
Let
\[ \jg=\jg_1 \jg_2\dots \jg_{n+m} \coloneqq 
{\je_1}\dots{\je_k} \,\jf_1 \dots \jf_m \,\tilde{\je}_{k+1} \dots \tilde{\je}_n, \]
where $\tilde{\je}_i = \je_i + m\left(1-\delta_{0,\je_i}\right)$.
Then \( \jg \) is an \( (101,021) \)-avoiding inversion sequence with 
\[ \first(\jg)=k+\ell. \] 
\end{lem}
\begin{proof}
Since $\je$ and $\jf$ is of form~\eqref{jinv-form}, so does $\jg$. Thus, $\jg$ is an \( (101,021) \)-avoiding inversion sequence.
$\je_1=\dots=\je_k=0<\je_{k+1}$ and $\jf_1=\dots=\jf_{\ell}=0<\jf_{\ell+1}$ yields \( \first(\jg)=k+\ell\). 
\end{proof}
}

{Due to Lemma~\ref{lem:directII}, we define the direct sum of \( (101,021) \)-avoiding inversion sequences as follows.}
\begin{defn} 
Let \( n \) and \( m \) be positive integers.
For any \( (101,021) \)-avoiding inversion sequences 
{
\( \je=\je_1 \je_2 \dots \je_n \) and \( \jf=\jf_1 \jf_2 \dots \jf_m\),
} 
the \emph{direct sum} of \( \je \) and \( \jf \) is the \( (101,021) \)-avoiding inversion sequence 
 {$\je \oplus \jf$} of length \( n+m \), defined by 
\[
(\je \oplus \jf)_i\coloneqq \begin{cases}
\jf_{i-\first(\je)} & \mbox{ if \( \first(\je)+1 \leq i \leq \first(\je)+m \),}\\
\je_{i-m}+m &\mbox{ if \( i>\first(\je)+m \) and \( \je_{i-m}>0 \)},\\
0 & \mbox{otherwise}.
\end{cases}
\]
An \( (101,021) \)-avoiding inversion sequence \( \je \) is said to be \emph{connected} if \( \first(\je)=1 \). Otherwise, \( \je \) is said to be \emph{disconnected}. 
\end{defn}

{
Note that \( \first(\je\oplus \jf)=\first(\je)+\first(\jf) \). From Lemma~\ref{lem:directII},} 
for any disconnected \( (101,021) \)-avoiding inversion sequence 
{\( \jg\) of length $N$,}
we can decompose it into two \( (101,021) \)-avoiding inversion sequences, \( \je\oplus \jf \), by defining 
\( \jf\coloneqq \jg_{r}\jg_{r+1} \dots \jg_{r+m-1} \), where
\( \first(\jg)=r \) and {\( m \) is the smallest positive integer such that
\( \jg_{r+m} \ge m+1\),} so that $\jf$ is the connected.
{
Also, $\je:=\jg_1 \dots \jg_{r-1} \, \jg'_{r+m}\dots\jg'_{N}$ is the \( (101, 021) \)-avoiding inversion sequence
with  $\first(\je)=r-1$, 
where $\jg'_i = \jg_i + m(\delta_{0,\jg_i}-1)$.
}
{
Therefore, an \( (101, 021) \)-avoiding inversion sequence \( \jg \) with \(\first(\jg)=r \) can be written uniquely as a direct sum of \( r \) connected \( (101, 021) \)-avoiding inversion sequences. 
}
\begin{example}
For the running example, we have 
\[
\psi_{\rm J}(R)=\psi_{\rm J}(R_1) \oplus \psi_{\rm J}(R_2) \oplus \psi_{\rm J}(R_3) \oplus \psi_{\rm J}(R_4) \oplus \psi_{\rm J}(R_5),
\]
where each \( \psi_{\rm J}(R_i) \) is a connect \( (101,021) \)-avoiding inversion sequence.
Details are as follows.
{
\begin{align*}
0\,0\,0\,0\,0\,1\,0\,0\,4\,4\,5\,9\,9\,9\,0\,0 
&=0\,0\,0\,0\,2\,2\,2\,0\,0\oplus 0\,1\,0\,0\,4\,4\,5 \\
&=0\,0\,0\,1\,1\,1\,0\,0\oplus 0 \oplus 0\,1\,0\,0\,4\,4\,5\\
&=0\,0\oplus 0\,1\,1\,1\,0\,0\oplus 0 \oplus 0\,1\,0\,0\,4\,4\,5\\ 
&=0 \oplus 0\oplus 0\,1\,1\,1\,0\,0\oplus 0 \oplus 0\,1\,0\,0\,4\,4\,5 .
\end{align*}
}
\end{example}

By the construction of each bijection, we have the following result. 

\begin{prop}
Given an \( F \)-path \( Q \) with \( \height(Q)=k \), let \( Q=Q_1 \oplus Q_2 \oplus \dots \oplus Q_{k+1} \), where each \( Q_i \) is an \( F \)-path with \( \height(Q_i) =0 \) for \( 1\leq i \leq k+1 \). Then, for any \({\rm A} \in \set{\rm P,B,S,I,J,T} \), we have
\[
\psi_{\rm A}(Q)=\psi_{\rm A}(Q_1) \oplus \psi_{\rm A}(Q_2) \oplus \dots \oplus \psi_{\rm A}(Q_{k+1}).\\
\]
\end{prop}


\section*{Acknowledgement}

For the first author, 
this work was partially supported by 
Basic Science Research Program through the National Research Foundation of Korea(NRF) funded by the Ministry of Education \linebreak[4]
(No.\ NRF-2021R1A6A1A10044950). 
For the second author, 
this work was partially supported by 
the National Research Foundation of Korea (NRF) grant funded by the Korea government (MSIT) \linebreak[4]
(No.\ 2021R1F1A1062356).
For the third author, 
this work was partially supported by 
the National Research Foundation of Korea (NRF) grant funded by the Korea government (MSIT) (No.\ 2021R1F1A1062356).
For the fourth author, 
this work was partially supported by 
the National Research Foundation of Korea (NRF) grant funded by the Korea government (MSIT) (No.\ 2017R1C1B2008269).



\end{document}